\documentclass{article}

\usepackage{microtype}
\usepackage{graphicx}
\usepackage{subcaption}
\usepackage{booktabs} 

\usepackage{hyperref}

\usepackage[preprint]{icml2026}

\usepackage{amsmath}
\usepackage{amssymb}
\usepackage{mathtools}
\usepackage{amsthm}
\usepackage{aliascnt}

\usepackage[capitalize,noabbrev]{cleveref}

\theoremstyle{plain}
\newtheorem{theorem}{Theorem}[section]

\newaliascnt{proposition}{theorem}
\newtheorem{proposition}[proposition]{Proposition}
\aliascntresetthe{proposition}

\newaliascnt{lemma}{theorem}
\newtheorem{lemma}[lemma]{Lemma}
\aliascntresetthe{lemma}

\newaliascnt{corollary}{theorem}
\newtheorem{corollary}[corollary]{Corollary}
\aliascntresetthe{corollary}

\theoremstyle{definition}
\newaliascnt{definition}{theorem}
\newtheorem{definition}[definition]{Definition}
\aliascntresetthe{definition}

\newaliascnt{assumption}{theorem}
\newtheorem{assumption}[assumption]{Assumption}
\aliascntresetthe{assumption}

\theoremstyle{remark}
\newaliascnt{remark}{theorem}
\newtheorem{remark}[remark]{Remark}
\aliascntresetthe{remark}

\crefname{theorem}{Theorem}{Theorems}
\Crefname{theorem}{Theorem}{Theorems}

\crefname{lemma}{Lemma}{Lemmas}
\Crefname{lemma}{Lemma}{Lemmas}

\crefname{proposition}{Proposition}{Propositions}
\Crefname{proposition}{Proposition}{Propositions}

\crefname{corollary}{Corollary}{Corollaries}
\Crefname{corollary}{Corollary}{Corollaries}

\crefname{definition}{Definition}{Definitions}
\Crefname{definition}{Definition}{Definitions}

\crefname{assumption}{Assumption}{Assumptions}
\Crefname{assumption}{Assumption}{Assumptions}

\crefname{remark}{Remark}{Remarks}
\Crefname{remark}{Remark}{Remarks}

\usepackage[disable,textsize=tiny]{todonotes}

\usepackage[utf8]{inputenc} 
\usepackage[T1]{fontenc}    
\usepackage{url}            
\usepackage{amsfonts}       
\usepackage{nicefrac}       
\usepackage{xcolor}         
\usepackage{mathdots}

\newcommand{\interior}[1]{%
  {\kern0pt#1}^{\mathrm{o}}%
}

\newcommand{\bbR}{\mathbb{R}}

\newcommand{\bbE}{\mathbb{E}}

\newcommand{\hs}{\hat s}

\newcommand\ignore[1]{}

\newcommand{\p}{p}
\newcommand{\s}{s}
\newcommand{\X}{X}
\newcommand{\T}{T}
\newcommand{\matA}{M}

\newcommand\D[1]{D_{f}\left(#1, x^*\right)}
\newcommand\Dp[2]{D_{f_{#1}}\left(#2_{#1}, x_{#1}^*\right)}
\newcommand\Dphi{D_{\phi}}

\newcommand\gradf[1]{\nabla f\left(#1\right)}
\newcommand\gradfp[2]{\nabla f_{#1}\left(#2_{#1}\right)}

\newcommand{\xe}{\xi}
\newcommand{\xtau}{\tau}
\newcommand{\xga}{\gamma}
\newcommand{\xgt}{\alpha}
\newcommand{\xh}{\chi}

\newcommand{\xconv}{\Pi}

\newcommand{\yc}{\xi}
\newcommand{\yeta}{\tau}
\newcommand{\yq}{\gamma}

\newcommand{\yconv}{\Pi}

\icmltitlerunning{Directly Accelerated Algorithms for Bilinear Saddle Point Problems}

\begin{document}

\twocolumn[
  \icmltitle{Direct Spectral Acceleration of First-Order Methods for Saddle Point Problems with Bilinear Coupling
  }



  \icmlsetsymbol{equal}{*}

\begin{icmlauthorlist}
  \icmlauthor{Meng Li}{berkeley}
  \icmlauthor{Paul Grigas}{berkeley}
\end{icmlauthorlist}

\icmlaffiliation{berkeley}{Industrial Engineering and Operations Research Department, University of California, Berkeley, California, USA}

\icmlcorrespondingauthor{Meng Li}{meng\_li@berkeley.edu}
\icmlcorrespondingauthor{Paul Grigas}{pgrigas@berkeley.edu}

\icmlkeywords{Convex optimization, saddle point problems, first-order methods, block-coordinate methods.} 


  \vskip 0.3in
]



\printAffiliationsAndNotice{}  
\begin{abstract}
We study convex--concave saddle point problems with bilinear coupling, covering linearly constrained convex optimization and more general nonsmooth or constrained models via a proximable term in the dual objective. In linearly convergent regimes, we characterize how spectral properties of the coupling matrix and objective conditioning jointly determine the attainable linear rates.
We propose direct spectral acceleration for first-order primal--dual methods for a class of bilinear-coupled saddle point problems, including affinely constrained smooth strongly convex optimization and extensions with proximable dual terms.
The resulting algorithms distinguish objective-dominated and coupling matrix-dominated regimes and attain optimal linear convergence without Chebyshev inner loops or double-loop designs.
We further develop stochastic block-coordinate extensions in the affinely constrained case with separable objectives; we also establish optimal linear rates matching the block-coordinate lower bound. For both deterministic and stochastic methods, we provide matching worst-case lower bounds via explicit finite-dimensional hard instances.
\end{abstract}

\section{Introduction}
\label{sec:intro}
In this paper, we consider two variants of saddle point problems with bilinear coupling. The first is the following general problem:
\begin{equation}\label{eq:upper bound prob}
    \min_{x\in\bbR^m}\max_{y\in\bbR^n}\quad F(x,y)=f(x)+y^{\top} \matA x-b^{\top}y-\phi(y),
\end{equation}
where \(f:\bbR^m\to\bbR\) is a smooth convex function, \(b\in\bbR^n\), the proximal term \(\phi:\bbR^n\to\bar \bbR := \bbR \cup \{+\infty\}\) is a proper convex and closed function, and \(\matA\in\bbR^{n\times m}\) is the coupling matrix. Another problem we consider is its block-wise extension in the equality-constrained case \((\phi\equiv 0)\), i.e.,
\begin{equation}\label{eq:upper bound prob sto}
\begin{aligned}
\min_{x_i\in\bbR^{m_i},i=1,\ldots,N}
&\max_{y\in\bbR^n}\quad
F(x,y)\ =\ f(x) + y^{\top}\matA x - b^{\top}y \\
&= \sum_{i=1}^N f_i(x_i)
+ y^{\top}\sum_{i=1}^N \matA_i x_i- b^{\top}y,
\end{aligned}
\end{equation}
where the coupling matrix \(\matA = \left(\matA_1,\ldots,\matA_N\right)\in\bbR^{n\times m}\) with \(m=\sum_{i=1}^Nm_i\), and the primal objective \(f(x)=\sum_{i=1}^N f_i(x_i)\) is separable. Problems of these forms arise in many contexts, including convex games \citep{ibrahim2020linear,azizian2020accelerating}, robust/adversarial training problems \citep{mkadry2017towards,bai2020provable}, and reinforcement learning problems \citep{du2017stochastic,dai2018sbeed}. The proximal term $\phi$ allows one to model constraints as well as other nonsmooth problems. For example, a common case of \eqref{eq:upper bound prob} is convex linearly constrained optimization: when $\phi(y)=I_{\mathbb{R}^n_+}(y)$ (i.e., $\phi(y)=0$ if $y\ge0$ and $+\infty$ otherwise), problem \eqref{eq:upper bound prob} is equivalent to
\begin{equation}\label{eq:examplelinear}
\min_x \ f(x) \quad \mathrm{s.t.} \quad \matA x\le b.
\end{equation}
Linearly constrained problems arise in many application domains, including predictive control \citep{borrelli2017predictive}, portfolio optimization \citep{markowitz2000mean}, convex regression \citep{seijo2011nonparametric,lim2012consistency}, and sparse regression \citep{atamturk2018strong,atamturk2019rank,han20232}.

In this paper, our main results focus on the linearly convergent regime where \(f\) is \(\mu\)-strongly convex and \(\matA\) has full row rank.
In the deterministic setting, we allow a general proximable dual term \(\phi\); for the block-coordinate results, we restrict to the equality-constrained case \(\phi\equiv 0\) to obtain optimal linear-rate guarantees.

\noindent\textbf{Literature Review.} Many algorithms have been proposed to solve problems of the form \eqref{eq:upper bound prob}, including PDHG (a.k.a. Chambolle-Pock) for proximable primal terms \citep{chambolle2016ergodic,chambolle2016introduction}, gradient-based methods for smooth objectives \citep{drori2015simple}, and more general primal--dual splitting schemes such as Condat--V\`u and PDDY/PD3O \citep{condat2013primal,vu2013splitting,salim2022dualize}. 
Their stochastic variants have also been studied for the same problem class, including randomized block-coordinate updates \citep{chambolle2018stochastic,latafat2019new,fercoq2019coordinate,gao2019randomized} and randomized proximal mappings, stochastic gradient surrogates, or randomized extrapolation steps \citep{condat2022randprox,salim2022dualize,alacaoglu2020random}.
Closely related special cases in convex linearly constrained optimization are also studied in \citep{tatarenko2018smooth,lan2013iteration}. \ignore{\citep{tatarenko2018smooth,lan2013iteration,li2022new}}

Among recent analyses of primal--dual first-order methods, \citet{salim2022optimal,kovalev2022accelerated,kovalev2024linear} study convergence guarantees for problems of the form \eqref{eq:upper bound prob} in regimes where strong convexity/concavity is absent in one or both of the primal and dual objectives. In such settings, linear (or accelerated) rates typically depend on the spectral properties of $\matA$, in particular $\frac{s_{\max}}{s_{\min}}$, where $s_{\max}$ and $s_{\min}$ denote the largest and smallest singular values of $\matA$. In particular, \citet{salim2022optimal} consider \eqref{eq:upper bound prob} in the case where $f$ is smooth and strongly convex and $\phi\equiv 0$, which corresponds to a strongly convex optimization problem with linear equality constraints. They establish a complexity bound of
$O\left(\frac{s_{\max}}{s_{\min}}\sqrt{\frac{L}{\mu}}\log\left(\frac{1}{\epsilon}\right)\right)$
via a double-loop scheme that applies Chebyshev acceleration in the inner loop. \citet{kovalev2022accelerated} analyze spectral acceleration in a more general framework and recover the same complexity for strongly convex optimization with linear equality constraints using Chebyshev acceleration. While randomized block-coordinate primal--dual methods have been studied in, e.g.,
\citep{chambolle2018stochastic,latafat2019new,fercoq2019coordinate,gao2019randomized}, establishing linear convergence rates that simultaneously exploit the spectral properties of the coupling matrix and the strong convexity on the primal side is comparatively less explored, despite its theoretical and practical importance.

\noindent\textbf{Contributions.} For deterministic primal--dual methods, we propose new algorithms for \eqref{eq:upper bound prob} in the regime where $f(x)$ is smooth and strongly convex, \(\matA\) has full row rank, and $\phi(y)$ is proximable. We distinguish two regimes depending on whether the dominant difficulty comes from the conditioning of the objective, controlled by $L/\mu$, or from the conditioning of the coupling matrix, controlled by $s_{\max}^2/s_{\min}^2$. In the two regimes, we design direct acceleration schemes that act on the primal side and the dual side, respectively.
These accelerations yield the following advantages:  (i) the algorithms have complexities matching the lower bound without requiring inner-loop Chebyshev acceleration, (ii) they exhibit better numerical performance compared with existing algorithms, which need to perform a fixed number of Chebyshev inner iterations (depending on the condition numbers). As far as we are aware, our algorithms are the first that do not require a double-loop structure in this setting.
Furthermore, our framework accommodates a proximable dual term $\phi$, enabling applications to more general problems, including inequality-constrained optimization.

A second contribution is a stochastic extension of our direct acceleration schemes to the block-coordinate setting \eqref{eq:upper bound prob sto}.  By combining primal-side and dual-side stochastic accelerations with techniques for decorrelation and variance reduction, we obtain optimal linear convergence rates for our block-coordinate directly accelerated algorithms when $f$ is separable, matching the corresponding lower bound we derive for the same class of first-order methods. Theoretically, these results sharpen the understanding of how stochastic block updates interact with coupling spectra to yield optimal linear rates; practically, they reduce per-iteration cost and lead to improved empirical performance on large-scale instances. In addition, in the regime where objective conditioning dominates, we extend the accelerated scheme to nonseparable objectives while preserving the optimal linear rate. Table~\ref{tab:alg-summary-4} in \Cref{appensec:summary} summarizes our algorithms and their regimes and complexities.

Finally, we complement our upper bounds with matching worst-case lower bounds that make the dependence on the coupling spectrum explicit. Although broader unified frameworks exist (e.g., \citep{kovalev2024linear}) and related bounds for equality-constrained strongly convex optimization can be inferred from \citep{salim2022optimal,scaman2017optimal}, we focus on a simple, transparent hard-instance construction. Using the classical hard-instance approach \citep{ouyang2021lower,zhang2022lower}, we obtain finite-dimensional certificates with fully specified instances and dimensions, and this explicit viewpoint extends naturally from the deterministic setting to a block-coordinate lower bound. We also give an explicit smooth convex--concave lower bound under a separate assumption set.

\section{{Directly Accelerated Primal--Dual Algorithms}}
\label{sec:upper}
In this section, we propose and analyze our directly accelerated primal--dual algorithms on \eqref{eq:upper bound prob}. 
We start with preliminaries, including formal assumptions in Assumption~\ref{assum:Assumption Linear New}. Proposition~\ref{prop:saddle_cond_2} provides the existence and uniqueness of the saddle point. Then, we introduce \Cref{alg:x-accelerate-prox} and \Cref{alg:y-accelerate-prox}, with special direct acceleration structures based on the condition numbers. Finally, we state the (optimal) convergence rates of our algorithms. Proofs and additional remarks are deferred to \Cref{appensec:proof x alg,appensec:proof y alg}. Throughout this section, Assumption~\ref{assum:Assumption Linear New} is made concerning problem \eqref{eq:upper bound prob}. 
\begin{assumption}\label{assum:Assumption Linear New}
Suppose problem \eqref{eq:upper bound prob} satisfies:\\
(i) \(x\in\bbR^m\), \(y\in\bbR^n\), \(m\ge n\), \(\matA\in\mathbb{R}^{n\times m}\) has full row rank, with maximal singular value no larger than \(s_{\max}\), and minimal singular value \(s_{\min}>0\);\\
(ii) $f: \bbR^m\to\bbR$ is globally $\mu$-strongly convex for some $\mu > 0$, and globally $L$-smooth for some $L \ge \mu>0$;\\
(iii) the proximal term \(\phi:\bbR^n\to\bar \bbR\) is a proper convex and closed function, where \(\bar \bbR\) denotes the extended real values.
\end{assumption}

Note that \eqref{eq:upper bound prob} is an extension of problem (1) of  \citet{salim2022optimal}. A special and useful case is \(\phi(y)=I_{\bbR^n_+}(y)\), which corresponds to convex linearly constrained optimization \eqref{eq:examplelinear}.

\begin{proposition}\label{prop:saddle_cond_2}
For problem \eqref{eq:upper bound prob} under Assumption~\ref{assum:Assumption Linear New}, there exists a unique {saddle point} \((x^*,y^*)\). Furthermore, \((x^*,y^*)\) is a saddle point if and only if
\begin{equation}\label{eq:saddle_cond_2}
\gradf{x^*}+\matA^{\top} y^*=0;\quad \matA x^* -b \in\partial \phi\left(y^*\right).
\end{equation}
\end{proposition}

We introduce the following new notations that will be useful in the following sections:  \(D_f\left(x^1, x^2\right)=f(x^1)-f(x^2)-\left\langle\gradf{x^2},x^1-x^2\right\rangle\) is the Bregman distance associated with \(f\), and \(\left\|y\right\|_{G}^2=y^{\top}Gy\) or \(\left\|x\right\|_G^2=x^{\top}Gx\) when \(G\in\bbR^{n\times n}\) (or \(G\in\bbR^{m\times m}\)) is positive semidefinite. Similarly, since \(\matA x^* - b \in \partial \phi (y^*)\), let \({g^*} = \matA x^* -b\) and \(\Dphi\left(y,y^*\right)=\phi(y)-\phi(y^*)-\left\langle g^*, y-y^*\right\rangle \geq 0\). We let \(U(\cdot)\) denote the uniform distribution on a finite set. We let \(\langle \cdot, \cdot \rangle\) denote inner products.

\subsection{Accelerated Primal--Dual Algorithm \texorpdfstring{($x$-Side)}{(x-Side)}}
\label{subsec:upperx}
Our algorithm is motivated by Algorithm 3 of \citet{salim2022optimal}, which is an intermediate algorithm to solve strongly convex problems with linear equality constraints. In Proposition 1 of \citet{salim2022optimal}, the authors show that their intermediate algorithm solves strongly convex problems with linear equality constraints with complexity \(O\left(\frac{s_{\max}}{s_{\min}}\sqrt{\frac{L}{\mu}}\log(\frac{1}{\epsilon})\right)\) when \(\frac{s_{\max}}{s_{\min}}\sqrt{\frac{\mu}{L}} =O(1)\). As an extension, we propose \Cref{alg:x-accelerate-prox}, which now incorporates the proximal term and also has a complexity \(O\left(\frac{s_{\max}}{s_{\min}}\sqrt{\frac{L}{\mu}}\log(\frac{1}{\epsilon})\right)\) when \(\frac{s_{\max}}{s_{\min}}\sqrt{\frac{\mu}{L}} =O(1)\).

\begin{algorithm}[t]
\caption{$x$-side accelerated primal--dual algorithm}
\label{alg:x-accelerate-prox}
\begin{algorithmic}[1]
\REQUIRE Parameters $t, s, \hat s, \xh, \xga > 0, \xe > 1$
\STATE Initialize $x^{0}=z^{0} \in \bbR^{m}$, $y^{0} \in \bbR^{n}$
\REPEAT
  \STATE $\hat x^{k} = \xe z^{k} - (\xe - 1) x^{k}$
\STATE $y^{k+1} = \mathrm{prox}_{\xh s \phi}\left(y^{k} + \xh s(\matA \hat x^{k} - b)\right.$
\STATE \hspace{1.5em}$\left. - \hat s \matA(\matA^{\top} y^{k} + \nabla f(z^{k}))\right)$
  \STATE $x^{k+1} = z^{k} - t\left(\nabla f(z^{k}) + \matA^{\top} y^{k+1}\right)$
  \STATE $z^{k+1} = (1+\xga) x^{k+1} - \xga x^{k}$
\UNTIL{convergence}
\end{algorithmic}
\end{algorithm}

\begin{theorem}\label{thm:x-accelerate-prox-short}
Consider applying Algorithm~\ref{alg:x-accelerate-prox} to solve \eqref{eq:upper bound prob}. Let
\begin{equation}\label{eq:x-accelerate-prox para short}
\begin{aligned}
\xgt
&= \min\left(
\frac{1}{5},
\frac{s_{\max}}{s_{\min}}\sqrt{\frac{\mu}{8L}}
\right), \\
\xconv
&= \max\left(
\frac{s_{\max}^2}{s_{\min}^2} \cdot \frac{1}{2\xgt},
3\sqrt{\frac{L}{\mu}} + \frac{L}{\mu} \cdot 4\xgt
\right).
\end{aligned}
\end{equation}
For certain settings of ${\hat s}=\frac{1}{s_{\max}^2}, t, s, \xe, \xtau, \xga, \xh,
\Xi_f = 1,
\Xi_y > 0,
\Xi_v > 0$
(detailed in \Cref{appensec:proof x alg}), let
\begin{equation*}
\begin{aligned}
\Psi^k
&= \Xi_y \left\|y^k-y^*\right\|^2_{\left(I-(1-2\xgt)\hat s\matA\matA^{\top}\right)} \\
&
+ \Xi_f \D{x^k}
+ \Xi_v \left\|v^k-x^*\right\|_2^2.
\end{aligned}
\end{equation*}
where $v^k = (1+\xtau)z^k - \xtau x^k$, we have
\begin{equation*}
\Psi^{k+1} \le (1-1/\xconv)\Psi^k.
\end{equation*}
\end{theorem}

\subsection{Accelerated Primal--Dual Algorithm \texorpdfstring{($y$-side)}{(y-Side)}}
\label{subsec:uppery}
\citet{salim2022optimal} (and similarly \citet{kovalev2022accelerated}) insert the Chebyshev acceleration into the intermediate algorithm to achieve the complexity \(O\left(\frac{s_{\max}}{s_{\min}}\sqrt{\frac{L}{\mu}}\log\left(\frac{1}{\epsilon}\right)\right)\) when \(\frac{s_{\max}}{s_{\min}}\sqrt{\frac{\mu}{L}} =\Omega(1)\). Here, we propose the following direct acceleration Algorithm~\ref{alg:y-accelerate-prox}, which also has a complexity \(O\left(\frac{s_{\max}}{s_{\min}}\sqrt{\frac{L}{\mu}}\log(\frac{1}{\epsilon})\right)\) when \(\frac{s_{\max}}{s_{\min}}\sqrt{\frac{\mu}{L}} =\Omega(1)\). It has multiple advantages compared with the Chebyshev accelerated algorithm in \citet{salim2022optimal}: it has a simpler structure compared with a 2-loop algorithm; it could easily incorporate the proximal terms, which enables more applications of \eqref{eq:upper bound prob}; it has better performance in numerical experiments.

\begin{algorithm}[t]
\caption{$y$-side accelerated primal--dual algorithm}
\label{alg:y-accelerate-prox}
\begin{algorithmic}[1]
\REQUIRE Parameters $\tilde t, s, \hat s, \yq \ge 0, \yeta \ge 0, \yc \ge 1$
\STATE Initialize $x^{0} \in \bbR^{m}$, $y^{0}=w^{0}=u^{0} \in \bbR^{n}$
\REPEAT
\STATE $y^{k+1} = \mathrm{prox}_{s\phi}\left(w^{k} + s(\matA x^{k} - b)\right.$
\STATE \hspace{1.5em}$\left. - \hat s \matA(\matA^{\top} w^{k} + \gradf{x^{k}})\right)$
  \STATE $w^{k+1} = (1+\yq) y^{k+1} - \yq y^{k}$
  \STATE $u^{k+1} = (1+\yeta) w^{k+1} - \yeta y^{k+1}$
  \STATE $x^{k+1} = x^{k} - \tilde t(\gradf{x^{k}} + \matA^{\top} u^{k+1})$
\UNTIL{convergence}
\end{algorithmic}
\end{algorithm}

\begin{theorem}\label{thm:y-accelerate-prox-short}
Consider applying Algorithm~\ref{alg:y-accelerate-prox} to solve \eqref{eq:upper bound prob}. Let
\begin{equation}\label{eq:y-accelerate-prox para-short}
\begin{aligned}
\yc
&= \max\left(
1,
\frac{1}{\sqrt{2}}\frac{s_{\max}}{s_{\min}}\sqrt{\frac{\mu}{L}}
\right), \\
\yconv
&= \max\left(
\frac{2}{\yc}\frac{s_{\max}^2}{s_{\min}^2},
4\yc\frac{L}{\mu}
\right).
\end{aligned}
\end{equation}
For certain settings of ${\hat s}=\frac{1}{s_{\max}^2}, t=\frac{1}{2L},
\tilde t=\frac{t}{2\yc},
s,
\yeta,
\yq,
\Xi_h = 1,
\Xi_u > 0,
\Xi_x > 0$
(detailed in \Cref{appensec:proof y alg}), let
\begin{equation*}
\begin{aligned}
\Psi^k
&= \Xi_x \cdot \left[
\left\|x^k - x^*\right\|_2^2
- 2\left(t-\tilde t\right)\D{x^k}
\right] \\
&\quad + \Xi_h \cdot \frac{1}{2}\left\|\matA^{\top}\left(y^k-y^*\right)\right\|_2^2
+ \frac{1}{t}\Dphi(y^k, y^*) \\
&\quad + \Xi_u \left\|u^k - y^*\right\|^2_{\left(I - \frac{\hat s}{2}\matA\matA^{\top}\right)},
\end{aligned}
\end{equation*}
we have
\begin{equation*}
\Psi^{k+1} \le (1-1/\yconv)\Psi^k.
\end{equation*}
\end{theorem}

Note that when \(\frac{s_{\max}}{s_{\min}}\sqrt{\frac{\mu}{L}} = O(1)\), \Cref{thm:x-accelerate-prox-short} implies, and when \(\frac{s_{\max}}{s_{\min}}\sqrt{\frac{\mu}{L}} = \Omega(1)\), \Cref{thm:y-accelerate-prox-short} implies, that the number of iterations required to achieve \(\Psi^k \le \epsilon\) satisfies
\begin{equation*}
O\left(\frac{\log(1/\epsilon)}{\log\left(1/(1-1/\Pi)\right)}\right)
= O\left(\frac{s_{\max}}{s_{\min}}\sqrt{\frac{L}{\mu}}\log\left(\frac{1}{\epsilon}\right)\right).
\end{equation*}

\begin{remark}\label{remark:intuitions_of_direct_acceleration}
Let us explain the intuition behind our direct acceleration designs. When the $x$-side condition number, or equivalently the condition number of $f$, dominates, namely when 
\(
\frac{s_{\max}}{s_{\min}}\sqrt{\frac{\mu}{L}} = O(1),
\) 
$x$-side acceleration is necessary to obtain a rate whose dependence on the condition number of $f$ matches the $\sqrt{\frac{L}{\mu}}$ scale, as in a simpler problem $\min_x f(x)$. On the other hand, when 
\(
\frac{s_{\max}}{s_{\min}}\sqrt{\frac{\mu}{L}} = \Omega(1),
\) 
$y$-side acceleration is necessary to obtain a rate whose dependence on the condition number of the coupling matrix matches the $\frac{s_{\max}}{s_{\min}}$ scale, as in a simpler problem $\min_y \left\|\matA^{\top}y-c\right\|_2^2$.

From the $y$-side perspective, linear convergence of first-order methods can be understood through the strong concavity of the value function 
\(
\Phi_y(y) := \inf_{x \in \bbR^m} F(x,y),
\) 
which is induced by the full row rank of the bilinear coupling matrix $\matA$. In particular, $\Phi_y$ is $\frac{s_{\min}^2}{L}$-strongly concave. Therefore, when $\frac{s_{\min}}{s_{\max}}$ is small, acceleration must act on the $y$-iterates or otherwise leverage the spectral structure of $\matA$. This viewpoint motivates our Lyapunov function, in which the term \(
\frac{1}{2}\left\|\matA^\top(y^k-y^*)\right\|_2^2 
\) plays a role analogous to the Bregman distance $D_f(x^k,x^*)$ used in $x$-side acceleration. Moreover, the x-update step of the algorithm uses a smaller step size on the $x$-side (below the $\frac{1}{L}$ scale) to balance the convergence of the two variables. Our direct spectral acceleration removes the inner-loop Chebyshev preconditioning, admits a straightforward incorporation of proximable dual terms \(\phi\), and empirically converges faster in our experiments.
In contrast to our algorithm, \citet{salim2022optimal} apply Chebyshev iterations with a fixed number of inner steps to precondition the linear system.
\end{remark}

\section{Stochastic Block-Coordinate Directly Accelerated Primal--Dual Algorithms}
\label{sec:upper sto}
In this section, we develop and analyze stochastic block-coordinate extensions of our directly accelerated primal--dual algorithms for problem \eqref{eq:upper bound prob sto}. Proofs and additional remarks are deferred to \Cref{appensec:proof x stoc alg,appensec:proof y stoc alg}.
We modify Assumption~\ref{assum:Assumption Linear New} to match the block-wise formulation \eqref{eq:upper bound prob sto}.

\begin{assumption}[Block-wise affinely constrained strongly convex-minimization case]\label{assum:Assumption Linear Stoc}
Suppose problem \eqref{eq:upper bound prob sto} satisfies:\\
(i) For $i=1,\ldots,N$, $x_i \in \bbR^{m_i}$, and $y \in \bbR^n$, \(m=\sum_{i=1}^N m_i\ge n\), $\matA \in \bbR^{n\times m}$ has full row rank and minimal singular value $s_{\min} > 0$. Each block matrix $\matA_i$ has maximal singular value no larger than $\bar s_{\max}$;\\
(ii) For $i=1,\ldots,N$, $f_i:\bbR^{m_i}\to\bbR$ is globally $\mu$-strongly convex and globally $\bar L$-smooth for some $\bar L \ge \mu>0$.\\
\end{assumption}
Throughout this section, Assumption~\ref{assum:Assumption Linear Stoc} is made concerning problem \eqref{eq:upper bound prob sto}. Notice \Cref{prop:saddle_cond_2} still works for the current block-wise problem, yielding optimality conditions:
\begin{equation}\label{eq:saddle_cond_sto}
\gradfp{i}{x^*}+\matA_i^{\top} y^*=0,\ \forall i;\quad \sum_{i=1}^N\matA_i x_i^* - b=0.
\end{equation}

\subsection{Stochastic Block-Coordinate Accelerated Primal--Dual Algorithm \texorpdfstring{($x$-Side)}{(x-Side)}}
\label{subsec:upperx stoc}
We present the $x$-side block-coordinate algorithm (Algorithm~\ref{alg:x-accelerate-stoc}) and its linear convergence (\Cref{thm:x-accelerate-stoc-short}), with remarks on complexity, decorrelation, and extensions.
\begin{algorithm}[t]
\caption{$x$-side stochastic block-coordinate accelerated primal--dual algorithm}
\label{alg:x-accelerate-stoc}
\begin{algorithmic}[1]
\REQUIRE Parameters $t, s, \hat s, \xh, \xga, \xtau > 0, \xe > 1$
\STATE Initialize $x_i^{0}=z_i^{0}=v_i^{0} \in \bbR^{m_i}$ for \(i=1,\dots,N\), $y^{0} \in \bbR^{n}$
\REPEAT
  \STATE $\hat x^{k} = \xe z^{k} - (\xe-1)x^{k}$
  \STATE Randomly sample $i \sim U(\{1,\ldots,N\})$
  \STATE $y^{k+1} = y^{k} + \frac{\xh s}{N}(\matA \hat x^{k} - b)$
  \STATE \hspace{1.5em}$- \hat s \matA_i(\matA_i^{\top} y^{k} + \gradfp{i}{z^{k}})$
  \STATE Randomly sample $j \sim U(\{1,\ldots,N\})$
  \STATE $x_j^{k+1} = z_j^{k} - t(\gradfp{j}{z^{k}} + \matA_j^{\top} y^{k+1})$
  \STATE $z_j^{k+1} = (1+\xga)x_j^{k+1} - \xga x_j^{k}$
  \STATE $v_j^{k+1} = (1+\xtau)z_j^{k+1} - \xtau x_j^{k+1}$
  \STATE $(x_l^{k+1}, z_l^{k+1}, v_l^{k+1}) = (x_l^{k}, z_l^{k}, v_l^{k})$ for $l \in \{1,\ldots,N\}\backslash j$
\UNTIL{convergence}
\end{algorithmic}
\end{algorithm}
\begin{theorem}\label{thm:x-accelerate-stoc-short}
Consider applying Algorithm~\ref{alg:x-accelerate-stoc} to solve \eqref{eq:upper bound prob sto}. Let
\begin{equation}\label{eq:x-accelerate-stoc para-short}
\begin{aligned}
\xgt &= \min\left(\frac{1}{10},\ \sqrt{\frac{4}{7}}\frac{\bar s_{\max}}{s_{\min}}\sqrt{\frac{\mu}{\bar L}}\right),\\
\xconv &= N\max\left(\frac{\bar s_{\max}^2}{s_{\min}^2 } \cdot \frac{16}{7\xgt},\ 
\sqrt{\frac{14\bar L}{\mu}}+\frac{\bar L}{\mu}\cdot 4\xgt \right).
\end{aligned}
\end{equation}
For certain settings of
\({\hat s} = \frac{7}{32}\frac{1}{\bar s_{\max}^2}\), \(s\), \(t\), \(\xe\), \(\xtau\), \(\xga\), \(\xh,\ 
\Xi_f=1,\ \Xi_y,\ \Xi_v>0
\)
(detailed in \Cref{appensec:proof x stoc alg}), let
\begin{equation*}
\begin{aligned}
\Psi^k
&= \Xi_y \left\|y^k-y^*\right\|^2_{
\left(I-(1-2\xgt)\frac{{\hat s}}{N}\matA\matA^{\top}\right)} \\
&
+ \D{x^k}
+ \Xi_v \left\|v^k-x^*\right\|_2^2,
\end{aligned}
\end{equation*}
we have
\begin{equation*}
\bbE\Psi^{k+1}\le \left(1-\frac{1}{\xconv}\right)\bbE\Psi^k.
\end{equation*}
\end{theorem}

\begin{remark}\label{remark:stoc_x_complexity}
By \Cref{thm:x-accelerate-stoc-short}, when
\(\frac{\bar s_{\max}}{s_{\min}}\sqrt{\frac{\mu}{\bar L}} =O(1)\),
the number of iterations required to reach \(\bbE[\Psi^k]\le \epsilon\) satisfies
\begin{equation*}
O\left(\frac{\log(1/\epsilon)}{\log\left(\frac{1}{1-1/\xconv}\right)}\right)
=
O\left(N\frac{\bar s_{\max}}{s_{\min}}\sqrt{\frac{\bar L}{\mu}}\log\left(\frac{1}{\epsilon}\right)\right).
\end{equation*}
Similarly, \Cref{thm:y-accelerate-stoc-short} implies that when
\(\frac{\bar s_{\max}}{s_{\min}}\sqrt{\frac{\mu}{\bar L}} =\Omega(1)\),
the iteration complexity to achieve \(\bbE[\Psi^k]\le \epsilon\) is again
\(O\left(N\frac{\bar s_{\max}}{s_{\min}}\sqrt{\frac{\bar L}{\mu}}\log\left(\frac{1}{\epsilon}\right)\right).\)
This convergence rate is shown to be optimal given the lower bound in \Cref{cor:main_stoc_short}. Compared with the deterministic algorithms, which require
\(O\left(\frac{s_{\max}}{s_{\min}}\sqrt{\frac{L}{\mu}}\log\left(\frac{1}{\epsilon}\right)\right)\)
full matrix multiplications and full oracle access, the block-coordinate variants are advantageous when
\(\bar{s}_{\max} \ll s_{\max}\), a regime that typically arises when \(N\) is large.

Our convergence rate can also be interpreted from the dual perspective. The following problem is equivalent to
\eqref{eq:upper bound prob sto}:
\begin{equation*}
\min_{y\in \bbR^n}\ 
G(y)=b^{\top}y+\sum_{i=1}^N G_i(y)
=
b^{\top}y+\sum_{i=1}^N f_i^*\left(-\matA_i^{\top}y\right),
\end{equation*}
where \(f_i^*\) denotes the Fenchel conjugate of \(f_i\).
Under Assumption~\ref{assum:Assumption Linear Stoc}, \(G\) is
\(\frac{s_{\min}^2}{\bar L}\)-strongly convex, and each component \(G_i\) is
\(\frac{\bar s_{\max}^2}{\mu}\)-smooth. If one could access component oracles for \(\nabla G_i\) directly,
then classical accelerated stochastic methods for finite-sum optimization
apply and yield the oracle complexity
\begin{equation*}
\begin{aligned}
&O\left(\left(N+\sqrt{\frac{N\cdot \sum_{i=1}^N \frac{\bar s_{\max}^2}{\mu}}{\frac{s_{\min}^2}{\bar L}}}\right)
\log\left(\frac{1}{\epsilon}\right)\right)\\
=\ &
O\left(N\frac{\bar s_{\max}}{s_{\min}}\sqrt{\frac{\bar L}{\mu}}
\log\left(\frac{1}{\epsilon}\right)\right),    
\end{aligned}
\end{equation*}
see, e.g., \citet{allen2017katyusha,lin2018catalyst,li2021anita} for representative results. This matches the scaling of our bounds, even though our algorithm does not assume direct access to \(\nabla G_i\).
\end{remark}

\begin{remark}\label{remark:stoc_x_ij}
We sample two block indices \(i\) and \(j\) independently. This decoupling is essentially used in the convergence proof: if one enforces \(j=i\), then a direct adaptation of the deterministic Lyapunov argument would involve terms such as
\begin{equation*}
\|y^{k+1}-y^{*}\|_{(I-\hat s\,\matA_i\matA_i^\top)}^2,
\end{equation*}
whose conditional expectation is hard to control since \(y^{k+1}\) depends on the same random index \(i\). Independent sampling makes the expected Lyapunov decrease tractable; we use the same idea in Algorithm~\ref{alg:y-accelerate-stoc}.

The trade-off is a higher per-iteration oracle/matrix-access cost. Empirically, however, setting \(j=i\) performs similarly to independent sampling in our experiments (Figure~\ref{Fig_cst_stoc_appen}).
\end{remark}

\begin{remark}\label{remark:stoc_x_nonseparable}
The \(x\)-updates in Algorithm~\ref{alg:x-accelerate-stoc} are reminiscent of accelerated coordinate descent methods \cite{allen2016even}. In fact, consider a more general setting in which \eqref{eq:upper bound prob sto} is replaced by a problem with a non-separable objective \(f(x)=f(x_1,\dots,x_N)\). If we replace the component gradients \(\nabla f_j(x_j)\) by the block partial gradients \(\nabla_j f(x)\), then with modifications to the \(x\)-side steps of Algorithm~\ref{alg:x-accelerate-stoc} (see \Cref{appensubsec:x-stoc-nonsep}), one can establish a comparable convergence rate
\(
O\left(N\frac{\bar s_{\max}}{s_{\min}}\sqrt{\frac{L}{\mu}}\log\left(\frac{1}{\epsilon}\right)\right),
\)
under the regime \(\frac{\bar s_{\max}}{s_{\min}}\sqrt{\frac{\mu}{L}} =O(1)\), where \(f\) is \(L\)-smooth in \(x\). On the other hand, the \(y\)-side block-coordinate method Algorithm~\ref{alg:y-accelerate-stoc} does not extend naturally, since it relies on maintaining a stored quantity \(\matA \nabla f(x^k)\) while updating only a single block contribution \(\matA_i \nabla_i f(x^k)\) per iteration, a decomposition that is unavailable when \(f\) is non-separable.
\end{remark}
\subsection{Stochastic Block-Coordinate Accelerated Primal--Dual Algorithm \texorpdfstring{($y$-Side)}{(y-Side)}}
\label{subsec:uppery stoc}
We present the $y$-side block-coordinate algorithm (Algorithm~\ref{alg:y-accelerate-stoc}) and its linear convergence (\Cref{thm:y-accelerate-stoc-short}), with remarks on per-iteration complexity and variance reduction.
\begin{algorithm}[t]
\caption{$y$-side stochastic block-coordinate accelerated primal--dual algorithm}
\label{alg:y-accelerate-stoc}
\begin{algorithmic}[1]
\REQUIRE Parameters $\tilde t, s, \hat s, \yeta > 0, \yc > 1$
\STATE Initialize $x_i^{0} \in \bbR^{m_i}$ for $i=1,\ldots,N$, $y^{0}=w^{0}=u^{0} \in \bbR^{n}$
\REPEAT
  \STATE Randomly sample $i \sim U(\{1,\ldots,N\})$
  \STATE $\tilde y^{k+1} = w^{k} + \frac{s}{N}(\matA x^{k} - b)$
  \STATE \hspace{1.5em}$- \frac{{\hat s}}{N}\matA(\matA^{\top} y^{k} + \gradf{x^{k}}) - \hat s \matA_i\matA_i^{\top}(w^{k} - y^{k})$
\STATE $
y^{k+1} =
\left\{
\begin{aligned}
&\tilde y^{k+1} & \text{with probability } p=\frac{1}{N} \\
&y^{k} & \text{with probability } 1-p
\end{aligned}
\right.
$
  \STATE $u^{k+1} = \yc \tilde y^{k+1} - (\yc-1)y^{k}$
  \STATE $w^{k+1} = \frac{\yeta}{1+\yeta}y^{k+1} + \frac{1}{1+\yeta}u^{k+1}$
  \STATE Randomly sample $j \sim U(\{1,\ldots,N\})$
  \STATE $x_j^{k+1} = x_j^{k} - \tilde t(\gradfp{j}{x^{k}} + \matA_j^{\top} u^{k+1})$
  \STATE $x_l^{k+1} = x_l^{k}$ for $l \in \{1,\ldots,N\}\backslash j$
\UNTIL{convergence}
\end{algorithmic}
\end{algorithm}

\begin{theorem}\label{thm:y-accelerate-stoc-short}
Consider applying Algorithm~\ref{alg:y-accelerate-stoc} to solve \eqref{eq:upper bound prob sto}. Let
\begin{equation}\label{eq:y-accelerate-stoc para-short}
\begin{aligned}
\yc &= \max\left(\frac{1}{1-\sqrt{\frac23}},\ \sqrt{2}\frac{\bar s_{\max}}{s_{\min}}\sqrt{\frac{\mu}{\bar L}}\right),\\
\quad \yconv &= N\max\left(\frac{8}{\yc}\cdot\frac{\bar s_{\max}^2}{s_{\min}^2},\ 4\yc\cdot\frac{\bar L}{\mu}\right),
\end{aligned}
\end{equation}
For certain settings of 
\begin{align*}
\hat s = \frac{1}{4}\cdot\frac{1}{\bar s_{\max}^2},
\quad t = \frac{1}{2\bar L},\qquad
\tilde t = \frac{t}{2\yc},
\end{align*}
and the remaining parameters (detailed in \Cref{appensec:proof y stoc alg}), let
\begin{equation*}
\begin{aligned}
&\Psi^k
= \Xi_x \cdot \left[
\left\|x^k - x^*\right\|_2^2
- 2(t-\tilde t)\D{x^k}
\right] \\
&\quad + \frac12\left\|\matA^{\top}\left(y^k-y^*\right)\right\|_2^2
+ \Xi_u \left\|u^k-y^*\right\|^2_
{\left(I-\frac{\hat s}{2N}\matA\matA^{\top}\right)}.
\end{aligned}
\end{equation*}
Then
\begin{equation*}
\bbE\Psi^{k+1}\le (1-1/\yconv)\bbE\Psi^k.
\end{equation*}
\end{theorem}

\begin{remark}\label{remark:stoc_y_oraclecomplexity}
Apart from the one-time cost of computing and storing $\gradf{x^0}$ and $\matA\left(x^0-t\gradf{x^0}\right)$ at initialization, each iteration of Algorithm~\ref{alg:y-accelerate-stoc} consists of a deterministic update plus a randomized refresh.

In the deterministic part, each iteration uses at most one oracle access (to update $\gradfp{j}{x^{k+1}}$) and $O(1)$ block matrix multiplications (block matrix--vector products), e.g., $\matA_i\matA_i^{\top}(w^k-y^k)$ and $\matA_j^{\top}u^{k+1}$. Moreover, $\matA\left(x^k-t\gradf{x^k}\right)$ can be stored and updated incrementally since only the $j$-th blocks of $x$ and $\gradf{x}$ change.

For the refresh of \(\matA\matA^{\top} y^{k+1}\), we set \(p = \frac{1}{N}\). When \(y^{k+1}\) is refreshed (i.e., \(y^{k+1} = \tilde{y}^{k+1}\)), we compute the additional full \(\matA/\matA^{\top}\) products needed to update; otherwise we reuse the stored value. Hence the expected cost of these refresh-related products is \(O(1)\) per iteration.
A similar $O(1)$ per-iteration oracle/matrix-access bound holds for the $x$-side method; see Remark~\ref{remark:stoc_x_oraclecomplexity} in the appendix.
\end{remark}

\begin{remark}\label{remark:stoc_y_variancereduction}
The \(y\)-updates in Algorithm~\ref{alg:y-accelerate-stoc} are closely related to accelerated stochastic methods for finite-sum problems
\cite{allen2017katyusha,lin2018catalyst,li2021anita}. It is well known that, to obtain linear convergence in such settings, one typically needs a variance-reduction mechanism.

To illustrate the point, consider the quadratic finite-sum problem
\begin{equation*}
\min_{y\in\bbR^n}\ \sum_{l=1}^N \left\|\matA_l^{\top}y-c_l\right\|_2^2 .
\end{equation*}
Using only stochastic gradients based on a single sampled index \(l\) generally does not yield the linear rate in terms of block matrix multiplications. In contrast, variance-reduced accelerated methods can achieve a complexity on the order of
\(O\left(N\frac{\bar s_{\max}}{s_{\min}}\log\left(\frac{1}{\epsilon}\right)\right)\)
block matrix multiplications.

In Algorithm~\ref{alg:y-accelerate-stoc}, we incorporate a variance-reduction step, similar in spirit to the loopless scheme in \cite{li2021anita}, to obtain the optimal rate while preserving a loopless structure.

In comparison, Algorithm~\ref{alg:x-accelerate-stoc} does not require an explicit variance-reduction mechanism to achieve the optimal rate in the regime
\( \frac{\bar s_{\max}}{s_{\min}}\sqrt{\frac{\mu}{\bar L}}=O(1)\).
In this regime, the accelerated coordinate-descent behavior on the \(x\)-side dominates the stochastic-gradient effects coming from the \(y\)-side.
\end{remark}

\section{Lower Bound Results}
\label{sec:lower}

We establish worst-case lower bounds for deterministic and stochastic block-coordinate first-order methods on
\eqref{eq:upper bound prob} and \eqref{eq:upper bound prob sto}. The bounds explicitly expose the
role of the coupling spectrum and match the rates of our direct acceleration
schemes, implying optimality up to constants. Existing works provide more general unified bounds
(e.g., \citep{kovalev2024linear}) and, for the equality-constrained strongly convex setting,
related lower bounds can be derived from \citep{salim2022optimal,scaman2017optimal}. In contrast,
our lower bound certificates use a simplified construction that makes the hard instances and their
dimensions explicit, and it also enables a direct extension from the deterministic result
(\Cref{thm:main_c1_short}) to the block-coordinate setting (\Cref{cor:main_stoc_short}). In addition, we include a smooth convex--concave lower bound (\Cref{thm:main_c2_short}) under different assumptions.
All formal statements and proofs, and detailed comparisons with existing results, are deferred to Appendix~\ref{appensec:detail construction}.

Our deterministic results apply to the standard linear-span first-order oracle model, and the
stochastic results apply to block-coordinate first-order methods that query $(\matA_i^\top y,\matA_i x_i)$
and $\nabla f_i(x_i)$ (or $\nabla_i f(x)$). Formal definitions are in Appendix~\ref{appensec:detail construction}.
We measure progress by the dual error $\|y^k-y^*\|_2$ (and the duality gaps in the appendix).

\begin{assumption}[Affinely constrained strongly convex-minimization case]\label{assum:Assumption Linear Lower 1}
Suppose problem \eqref{eq:upper bound prob} satisfies Assumption~\ref{assum:Assumption Linear New} with $\phi(y)=0$.
\end{assumption}

\begin{theorem}
\label{thm:main_c1_short}
Fix $L \ge \mu > 0$ and $s_{\max} \ge \sqrt{5}s_{\min} > 0$.
For any $\epsilon \in (0,1)$, there exists an instance of \eqref{eq:upper bound prob}
satisfying Assumption~\ref{assum:Assumption Linear Lower 1} such that any deterministic first-order method
requires
\begin{equation*}
k = \Omega\left(\frac{s_{\max}}{s_{\min}}\sqrt{\frac{L}{\mu}}
\log\left(\frac{1}{\epsilon}\right)\right)
\end{equation*}
iterations to produce an iterate with $\|y^k-y^*\|_2 \le \epsilon$.
\end{theorem}

The following corollary can be extended from \Cref{thm:main_c1_short} by applying its lower bound certificate.
\begin{corollary}\label{cor:main_stoc_short}
Fix $\bar L \ge \mu > 0$ and $\bar s_{\max} \ge \sqrt{5}s_{\min} > 0$. For any \(\epsilon \in (0,1)\),
there exists an instance of \eqref{eq:upper bound prob sto} satisfying Assumption~\ref{assum:Assumption Linear Stoc}
such that any block-coordinate first-order method needs
\begin{equation*}
\Omega\left(
N\frac{\bar s_{\max}}{s_{\min}}\sqrt{\frac{\bar L}{\mu}}
\log\left(\frac{1}{\epsilon}\right)
\right)
\end{equation*}
block-coordinate iterations to produce $\|y^k-y^*\|_2 \le \epsilon$.
\end{corollary}

Beyond the strongly convex-minimization setting where the dual objective is affine, we also consider a smooth convex--concave case \eqref{eq:problem_lower_2} in which the primal and dual dimensions match (\(m=n\)) and \(b^{\top}y\) is replaced by a convex function; neither strong convexity nor strong concavity is assumed.
Related convex--concave lower bounds appear in prior work \citep{kovalev2024linear}, but we include this case because our certificate uses the same construction technique as above and yields an explicit finite-dimensional instance with a more interpretable dependence on the problem parameters.
\begin{equation}\label{eq:problem_lower_2}
\min_{x\in \bbR^n}\max_{y\in\bbR^n} f(x)+y^{\top} \matA x - g(y)
\end{equation}

\begin{assumption}[Convex--concave case]\label{assum:Assumption Linear Lower 2}
Problem \eqref{eq:problem_lower_2} satisfies the properties:
(i) $x,y\in\bbR^n$, $m=n$, $\matA\in\bbR^{n\times n}$ has full rank, with maximal singular value no larger than $s_{\max}$,
and minimal singular value $s_{\min}>0$;
(ii) $f:\bbR^n\to\bbR$ is globally convex and $L_x$-smooth;
(iii) $g:\bbR^n\to\bbR$ is globally convex and $L_y$-smooth.
\end{assumption}

\begin{theorem}
\label{thm:main_c2_short}
Fix $L_x,L_y>0$ and $s_{\max} \ge \sqrt{5}s_{\min} > 0$.
For any $\epsilon \in (0,1)$, there exists a smooth convex--concave instance of \eqref{eq:problem_lower_2}
satisfying Assumption~\ref{assum:Assumption Linear Lower 2} such that any deterministic first-order method requires
\begin{equation*}
k = \Omega\left(
\sqrt{\frac{s_{\max}^2}{s_{\min}^2}+\frac{L_xL_y s_{\max}^2}{s_{\min}^4}}
\log\left(\frac{1}{\epsilon}\right)
\right)
\end{equation*}
iterations to produce $\|y^k-y^*\|_2 \le \epsilon$.
\end{theorem}

\section{Numerical Experiments}
\label{sec:numerical}
We examine the performance of Algorithms~\ref{alg:x-accelerate-prox} and \ref{alg:y-accelerate-prox}
on a compressed sensing task (CST) and a non-smooth estimation (NSE) problem for a linear
time-invariant system. We also evaluate Algorithms~\ref{alg:x-accelerate-stoc} and
\ref{alg:y-accelerate-stoc} on a larger CST instance. We compare our methods with the algorithms
of \citet{drori2015simple,salim2022dualize,salim2022optimal}. We report the relative errors
\(\frac{\|x^k - x^{*}\|_2}{\|x^{*}\|_2}\)
(or \(\frac{\|\X^k - \X^{*}\|_F}{\|\X^{*}\|_F}\) for \Cref{subsec:nse})
across different deterministic methods as a function of the iteration \(k\); 
for the block-coordinate experiments we report errors versus the number of block matrix multiplications (BMMs) to match arithmetic cost.
The reference solutions \(x^{*}\) (or \(\X^*\)) are either computed using the Gurobi Optimizer
(version \texttt{12.0.2}) via the \texttt{gurobipy} Python interface (version \texttt{12.0.2})
in Python \texttt{3.9}, or are available in closed form due to the problem construction in
\Cref{subsec:ineq}. All solutions were obtained in Python/Jupyter notebooks. (Some curves eventually plateau due to the finite accuracy of the reference solution; see \Cref{Fig_cst_metric} in the appendix.)
We use \texttt{PAPC} to denote the proximal alternating predictor--corrector algorithm \citep{drori2015simple}
and adopt the parameter settings of \citet{salim2022dualize}, for which the complexity is
\(O\!\left(\left(\frac{L}{\mu}+\frac{s_{\max}^2}{s_{\min}^2}\right)\log\!\left(\frac{1}{\epsilon}\right)\right)\)
(Theorem~6.2). We use \texttt{CAPD} (Chebyshev-accelerated primal--dual) to denote Algorithm~1 of
\citet{salim2022optimal} with Chebyshev acceleration. We denote Algorithms~\ref{alg:x-accelerate-prox}
and \ref{alg:y-accelerate-prox} by \texttt{x-DAPD} and \texttt{y-DAPD}, and
Algorithms~\ref{alg:x-accelerate-stoc} and \ref{alg:y-accelerate-stoc} by \texttt{x-SBC-DAPD} and
\texttt{y-SBC-DAPD}, respectively (\texttt{DAPD}: directly accelerated primal--dual; \texttt{SBC}: stochastic block-coordinate).
\subsection{Compressed-Sensing-Type Experiment}
\label{subsec:cst}
We first illustrate the performance of our algorithms, especially \Cref{alg:y-accelerate-prox}, in the compressed-sensing-type experiment considered by \citet{salim2022optimal}. The goal is to estimate a sparse vector \(x^{\sharp}\in\bbR^m\) with \(m=1000\), having 50 randomly chosen nonzero elements (equal to 1), from measurements $b=\matA x^{\sharp} \in \mathbb{R}^n$, with $n=250$, where $\matA$ has random i.i.d. Gaussian elements and its nonzero singular values are modified so that they span the interval $[s_{\min}, s_{\max}]$ for given \(0<s_{\min}<s_{\max}\). The objective \(f:\bbR^m\to\bbR\) is a strongly convex approximation of \(\ell_1\)-norm: \(f(x)=\sum_{i=1}^m \left[\sqrt{x_i^2+e^2}+\frac{e}{2}x_i^2\right]\), with \(e=\sqrt{1/(\kappa-1)}\), so that \(f\) has a condition number \(\kappa\) for given \(\kappa> 1\). Specifically, \(f(x)\) is a combination of a Pseudo-Huber loss and a ridge regularization. Hence, the problem becomes
\begin{equation}\label{eq:CST numerical}
\min_{x\in\bbR^m}\max_{y\in\bbR^n}\quad f(x) + y^{\top}(\matA x -b).
\end{equation}

Note that the solution \(x^*\) to \eqref{eq:CST numerical} is not equal to \(x^{\sharp}\) for general settings of \(s_{\min},s_{\max},\mu,L\). We test two different settings of the condition number parameters. The first is \(\frac{s_{\max}^2}{s_{\min}^2} = 10^5\) and \(\frac{L}{\mu}=10^4\), which is the same setting as in \citep{salim2022optimal}, and the second is \(\frac{s_{\max}^2}{s_{\min}^2} = 10^6\) and \(\frac{L}{\mu}=10^3\). 
The former is shown in the left panels of \Cref{Fig_cst_main} and \Cref{Fig_cst_appen} and the latter in the right panels of \Cref{Fig_cst_main} and \Cref{Fig_cst_appen}, both of which show a single simulation varying the number of iterations.
We notice that in both settings, \texttt{y-DAPD} converges much faster than the other algorithms. Though \texttt{CAPD} has the same \(O\left(\frac{s_{\max}}{s_{\min}}\sqrt{\frac{L}{\mu}}\log(\frac{1}{\epsilon})\right)\) complexity as \texttt{y-DAPD} in these situations, it requires a constant number of inner loop iterations (317 for the first setting and 1000 for the second), which slows down its convergence. For the problems with larger ratios of \(\frac{s_{\max}}{s_{\min}}\) to \(\sqrt{\frac{L}{\mu}}\), \texttt{CAPD} converges slower than \texttt{PAPC} at the beginning, but catches up in the later iterations due to a better convergence rate. (See \Cref{appensec:supp cst} for \Cref{Fig_cst_appen} and further details.)
\begin{figure}[t]
\begin{minipage}{0.5\linewidth}
\centering
\includegraphics[width=\linewidth]{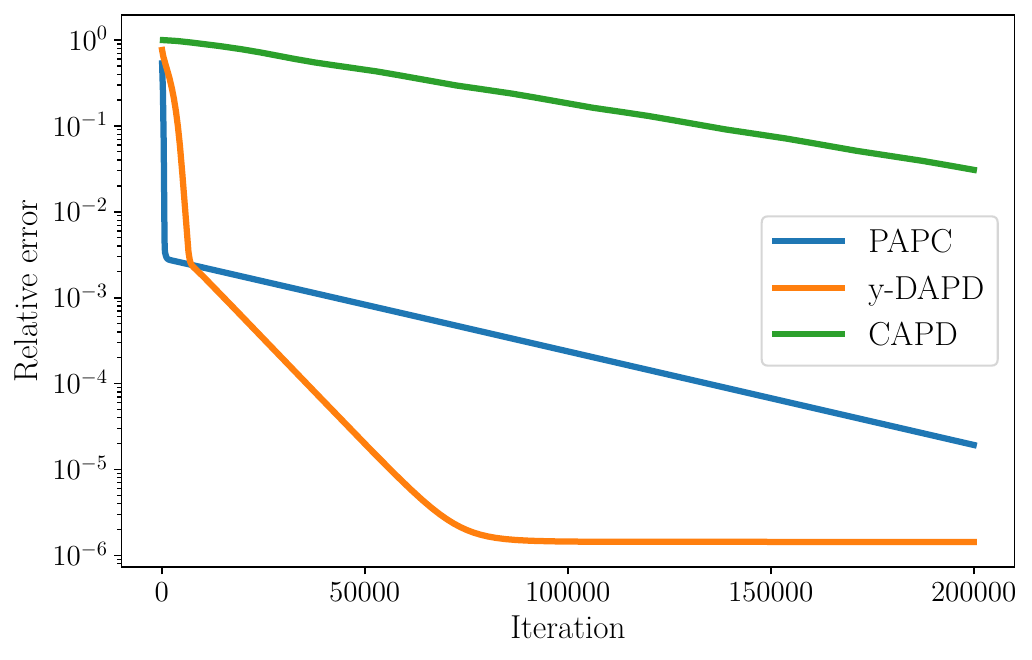}
\end{minipage}%
\begin{minipage}{0.5\linewidth}
\centering
\includegraphics[width=\linewidth]{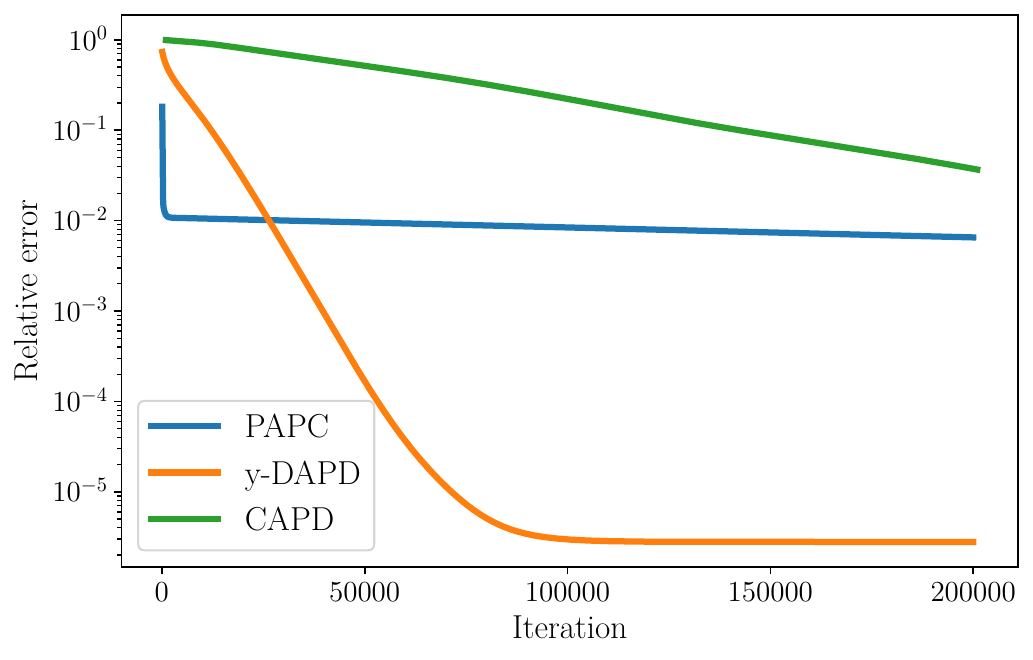}
\end{minipage}
\caption{Results for the compressed-sensing-type (CST) experiment. Left: \(\frac{s_{\max}^2}{s_{\min}^2} = 10^5\), \(\frac{L}{\mu}=10^4\); right: \(\frac{s_{\max}^2}{s_{\min}^2} = 10^6\), \(\frac{L}{\mu}=10^3\).}
\label{Fig_cst_main}
\end{figure}

We additionally consider a larger instance to evaluate the block-coordinate algorithms
(Algorithms~\ref{alg:x-accelerate-stoc} and \ref{alg:y-accelerate-stoc}). In \Cref{Fig_cst_stoc},
we set \(N=200\) and, for each block \(\matA_i\), \(m_i=50\) and \(n=100\), so that \(m=Nm_i=10000\), and in total \(500\) randomly chosen nonzero elements in \(x^{\sharp}\);
all other settings, including the data generation procedure, remain the same (the condition-number regime is slightly different). We plot the relative
error versus the number of block matrix multiplications (for deterministic methods this equals \(2N\)
per iteration, while it equals \(4\) for \texttt{x-SBC-DAPD} and \(6\) for \texttt{y-SBC-DAPD}).
As expected, the \(x\)-side and \(y\)-side accelerated methods converge faster than the others under
their respective favorable condition-number regimes. Moreover, since the block-wise maximal singular
values satisfy \(\bar{s}_{\max} \approx 0.11\, s_{\max}\), \texttt{SBC-DAPD} exhibits faster progress than
its deterministic counterpart in this large-\(N\) setting.

\begin{figure}[t]
\begin{minipage}{0.5\linewidth}
\centering
\includegraphics[width=\linewidth]{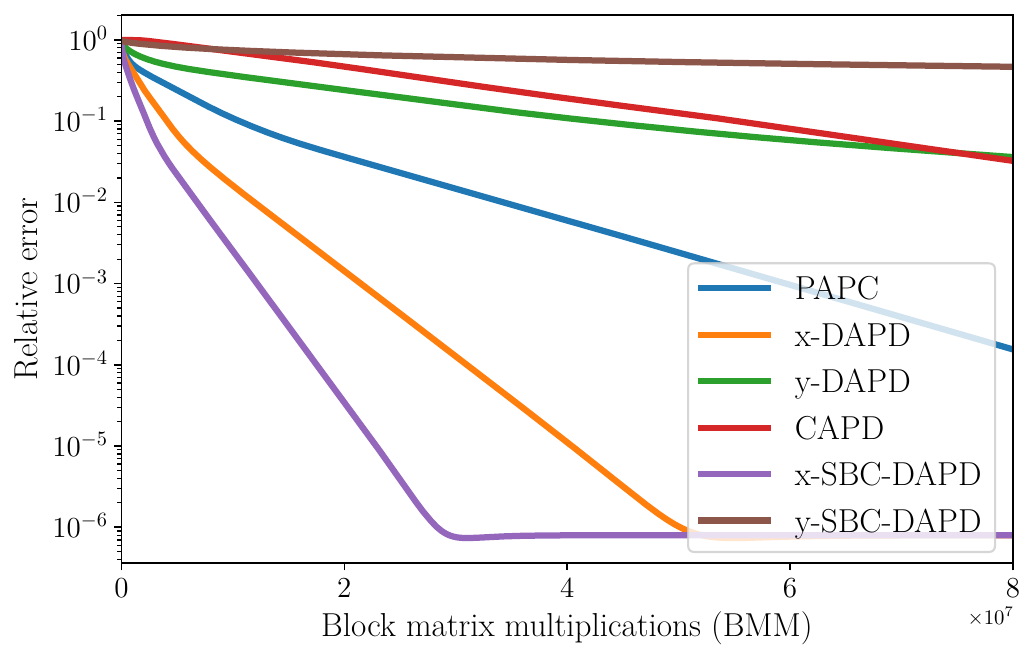}
\end{minipage}%
\begin{minipage}{0.5\linewidth}
\centering
\includegraphics[width=\linewidth]{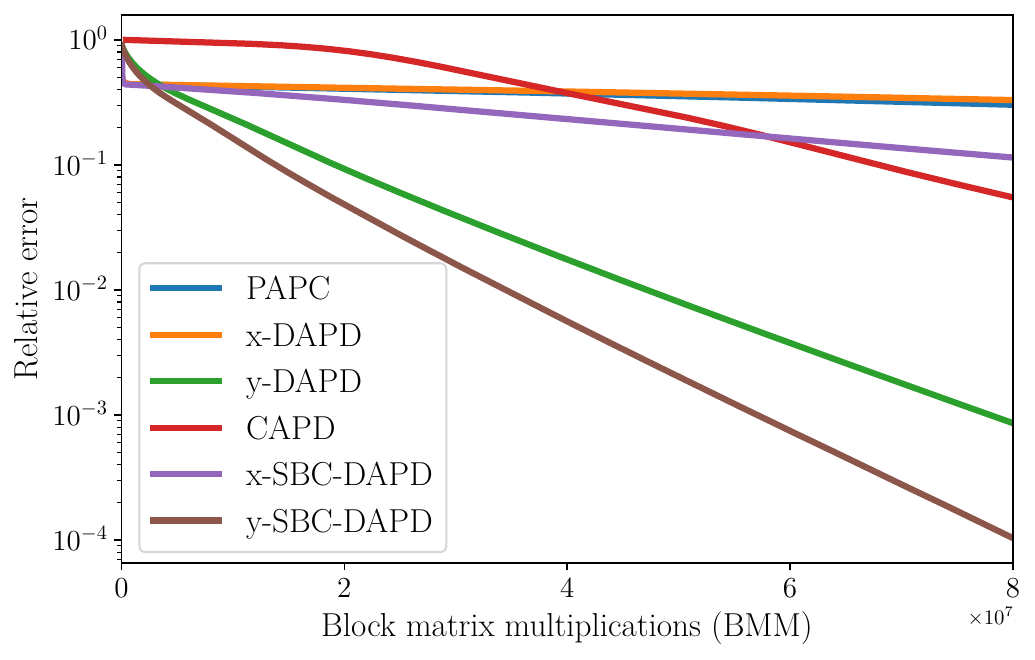}
\end{minipage}
\caption{Results for the CST experiment with \texttt{SBC-DAPD}. Left: \(\frac{s_{\max}^2}{s_{\min}^2} = 10^4\), \(\frac{L}{\mu}=10^5\); right: \(\frac{s_{\max}^2}{s_{\min}^2} = 10^6\), \(\frac{L}{\mu}=10^3\).}
\label{Fig_cst_stoc}
\end{figure}

\subsection{Non-Smooth Estimator of the Linear Time-Invariant System}
\label{subsec:nse}
We consider a linear time-invariant dynamical system of order $\p$ with the system update equation
\begin{equation}\label{eq:exp system}
    \s_{t}=\bar{\X} \s_{t-1}+\bar{d}_{t-1}, \quad t=1, \ldots, \T,
\end{equation}
where $\bar{\X} \in \mathbb{R}^{\p \times \p}$ is the unknown system matrix and $\bar{d}_t \in \mathbb{R}^{\p}$ are unknown system disturbances. \citet{yalcin2025subgradient} proposed the non-smooth estimator (NSE), \(\min _{\X \in \mathbb{R}^{\p \times \p}} \sum_{t=1}^{\T}\left\|\s_{t}-\X \s_{t-1}\right\|_2\), to address robustness to sparsely adversarial disturbance vectors.
Here, we add a similar strongly convex regularization as \Cref{subsec:cst} to recover the sparsity of \(\bar {\X}\), and our problem becomes:
\begin{equation}\label{eq:nse reg short}
\begin{aligned}
\min _{\X \in \mathbb{R}^{\p \times \p}}\quad f\left(\mathrm{vec}(\X)\right)+\lambda\sum_{t=1}^{\T}\left\|\s_{t}-\X \s_{t-1}\right\|_2
\end{aligned}
\end{equation}
where \(\mathrm{vec}(\cdot)\) is the vectorization operator. In \Cref{appensec:supp nse}, we demonstrate that this problem can be reformulated as \eqref{eq:upper bound prob} with a \(\p\T \times \p^2\) bilinear coupling matrix, and we present efficient methods for performing the first-order updates and computing the singular values of the coupling matrix. We note that, in contrast to \citet{yalcin2025subgradient} who study an online algorithm, our algorithms may be considered as offline estimators in this situation.

For simulations, we adopt a similar random problem generation procedure (detailed in \Cref{appensec:supp nse}) as that used in \citep{yalcin2025subgradient}. In \Cref{Fig_NSE}, we observe that when the condition number \(\frac{L}{\mu}\) is significantly larger than \(\frac{s_{\max}^2}{s_{\min}^2}\), the \texttt{x-DAPD} method converges faster, whereas \texttt{y-DAPD} exhibits a slightly slower convergence rate, similar to that of \texttt{PAPC}. In contrast, when \(\frac{s_{\max}^2}{s_{\min}^2}\) dominates, the \texttt{y-DAPD} method achieves significantly faster convergence, while \texttt{x-DAPD} performs comparably to \texttt{PAPC}.
\begin{figure}[t]
\begin{minipage}{0.5\linewidth}
\centering
\includegraphics[width=\linewidth]{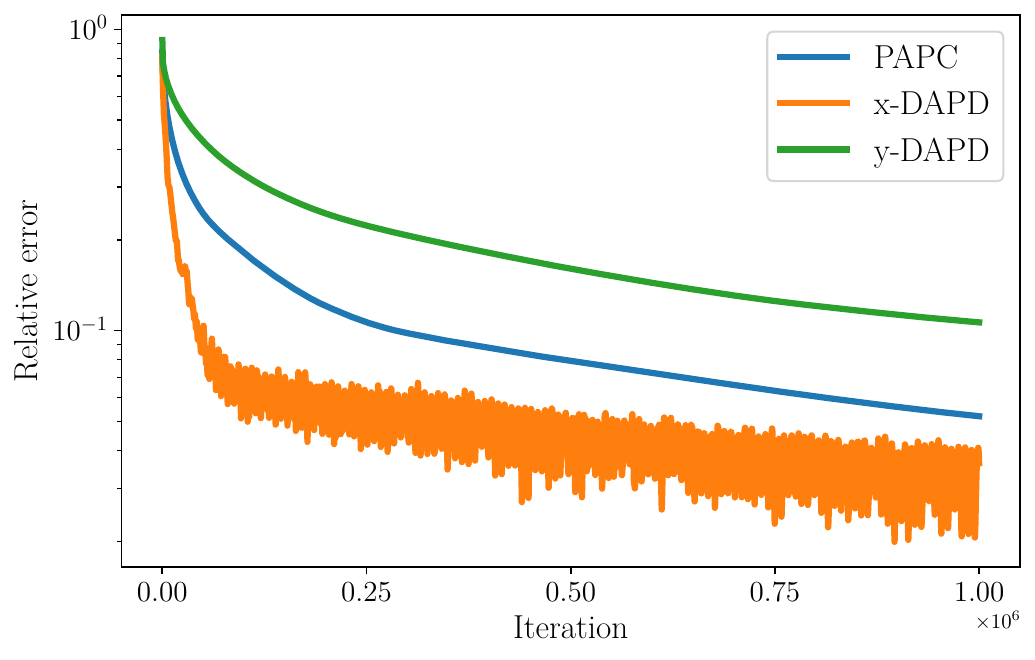}
\end{minipage}%
\begin{minipage}{0.5\linewidth}
\centering
\includegraphics[width=\linewidth]{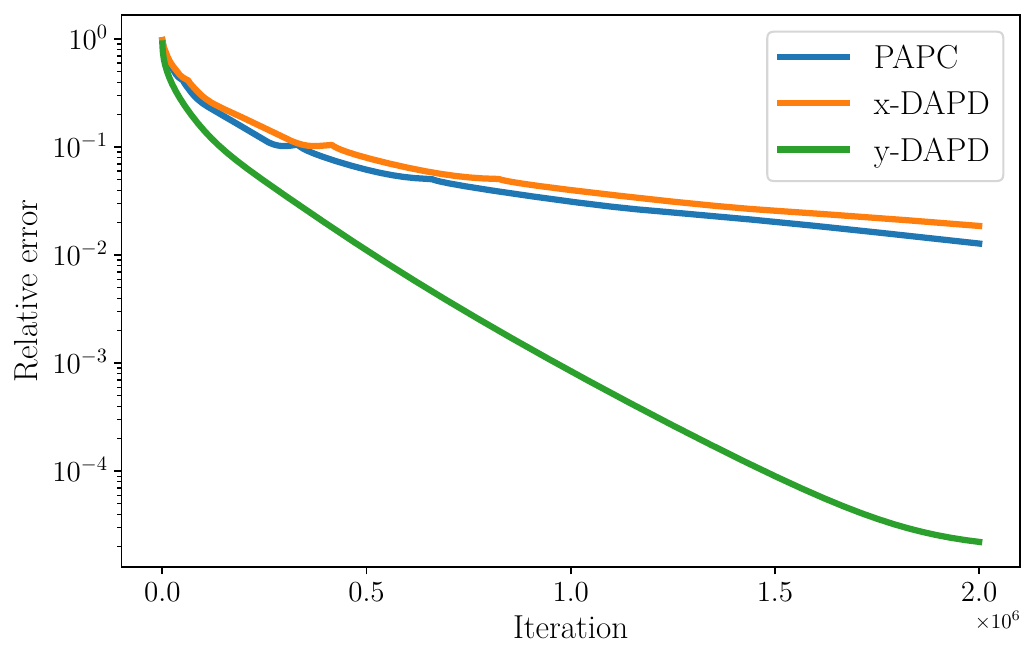}
\end{minipage}
\caption{Results for the NSE experiment \eqref{eq:nse reg short}. Left: \(\frac{L}{\mu}=10^9\), \(\p=40\), \(\T=10\), \(\frac{s_{\max}^2}{s_{\min}^2} \approx 2.7 \times 10^5\) in the simulation; right: \(\frac{L}{\mu}=10^4\), \(\p=100\), \(\T=20\), \(\frac{s_{\max}^2}{s_{\min}^2} \approx 1.3 \times 10^9\) in the simulation.}
\label{Fig_NSE}
\end{figure}

\section{Conclusions}
\label{sec:conclusions}

We developed single-loop direct spectral acceleration for first-order primal--dual methods on bilinear saddle-point problems and showed how objective conditioning and the coupling spectrum jointly determine linear rates. Our deterministic methods handle a proximable dual term and achieve optimal complexity without double-loop Chebyshev preconditioning. We further obtained optimal stochastic block-coordinate rates for separable objectives in the equality-constrained case and provided matching lower bounds via explicit finite-dimensional hard instances, with experiments confirming the practical gains.

Promising directions include extending our stochastic block-coordinate direct acceleration beyond the equality-constrained separable setting---incorporating proximable dual terms and nonseparable objectives---and applying the same objective-conditioning/coupling-spectrum perspective to broader (block-wise) primal--dual problems.

\section*{Acknowledgements}
Paul Grigas acknowledges the support of the NSF AI Institute for Advances in Optimization, Award 2112533.


\bibliography{bilinear_260103}
\bibliographystyle{icml2026}

\newpage
\appendix
\onecolumn

\section{Notations}
\label{appensec:notations}
We introduce the following notation that will be useful in the appendix. Let \(e_i\) denote the \(i\)-th unit vector, which has value 1 at the \(i\)-th component and 0 elsewhere. We define \(H(y) = \frac{1}{2} \left\| \matA^{\top}(y - y^*) \right\|_2^2\).

\section{Summary of Main Algorithms}
\label{appensec:summary}
We summarize Algorithms~\ref{alg:x-accelerate-prox}--\ref{alg:y-accelerate-stoc} in Table~\ref{tab:alg-summary-4}.
\begin{table}[!ht]
\centering
\caption{Summary of main accelerated primal--dual algorithms in this paper.}
\label{tab:alg-summary-4}
\small
\setlength{\tabcolsep}{6pt}
\begin{tabular}{lllllll}
\toprule
Algorithm & Accel.\ side & Type & Objective & $\phi(y)$ & Preferred regime & Complexity \\
\midrule
Alg.~\ref{alg:x-accelerate-prox}
& $x$-side
& deterministic
& nonseparable
& proximable
& $\frac{s_{\max}}{s_{\min}}\sqrt{\frac{\mu}{L}} = O(1)$
& $O\!\left(\frac{s_{\max}}{s_{\min}}\sqrt{\frac{L}{\mu}}\log\!\left(\frac{1}{\epsilon}\right)\right)$ \\

Alg.~\ref{alg:y-accelerate-prox}
& $y$-side
& deterministic
& nonseparable
& proximable
& $\frac{s_{\max}}{s_{\min}}\sqrt{\frac{\mu}{L}} = \Omega(1)$
& $O\!\left(\frac{s_{\max}}{s_{\min}}\sqrt{\frac{L}{\mu}}\log\!\left(\frac{1}{\epsilon}\right)\right)$ \\

Alg.~\ref{alg:x-accelerate-stoc}\textsuperscript{$\dagger$}
& $x$-side
& stochastic
& separable\textsuperscript{$\ddagger$}
& $\phi\equiv 0$
& $\frac{\bar s_{\max}}{s_{\min}}\sqrt{\frac{\mu}{\bar L}} = O(1)$
& $O\!\left(N\frac{\bar s_{\max}}{s_{\min}}\sqrt{\frac{\bar L}{\mu}}\log\!\left(\frac{1}{\epsilon}\right)\right)$ \\

Alg.~\ref{alg:y-accelerate-stoc}\textsuperscript{$\dagger$}
& $y$-side
& stochastic
& separable
& $\phi\equiv 0$
& $\frac{\bar s_{\max}}{s_{\min}}\sqrt{\frac{\mu}{\bar L}} = \Omega(1)$
& $O\!\left(N\frac{\bar s_{\max}}{s_{\min}}\sqrt{\frac{\bar L}{\mu}}\log\!\left(\frac{1}{\epsilon}\right)\right)$ \\
\bottomrule
\end{tabular}

\vspace{0.5ex}
{\footnotesize
\textsuperscript{$\dagger$} Stochastic = random block-coordinate updates.
\quad
\textsuperscript{$\ddagger$} A nonseparable extension is given in Alg.~\ref{alg:x-accelerate-stoc-nonseparable} in Section~\ref{appensubsec:x-stoc-nonsep}.
\quad
Complexities for deterministic methods count iterations with full first-order oracle access and full matrix multiplications (e.g., involving $\matA$ and $\matA^\top$).
Complexities for stochastic block-coordinate methods count iterations with block/partial first-order oracle access (\(\nabla f_i(x)\) or \(\nabla_i f(x)\))  and block matrix multiplications (e.g., involving $\matA_i$ and $\matA_i^\top$).
}
\end{table}

\section{Proof of Proposition~\ref{prop:saddle_cond_2}}
\label{appensec:proof saddle cond 2}
\begin{proof}
Similar to the proof of Lemma 1 in \citet{kovalev2022accelerated}, notice that 
\begin{equation*}
\Phi_y(y)=\inf_{x\in\bbR^m} F(x,y) = -b^{\top}y-\phi(y) + \inf_x\left[f(x)+x^{\top}\matA^{\top}y\right]=-b^{\top}y-\phi(y)-f^*\left(-\matA^{\top}y\right),
\end{equation*}
where \(f^*\) is the Fenchel conjugate of \(f\), and thus \(\frac{1}{\mu}\)-smooth and \(\frac{1}{L}\)-strongly convex (by Theorem 6 of \citet{kakade2009duality}). Therefore, \(-f^*\left(-\matA^{\top}y\right)\) is \(\frac{s_{\max}^2}{\mu}\)-smooth and \(\frac{s_{\min}^2}{L}\)-strongly concave. Hence, the solution 
\begin{equation*}y^*=\arg\max_{y\in\bbR^n} \Phi_y(y).\end{equation*} 
exists and is unique. By the strong convexity of \(f\), the solution \(x^*=\arg\min_{x\in\bbR^m} F(x,y^*)\) is unique, and it satisfies \(\gradf{x^*}+\matA^{\top} y^*=0\). By Danskin's theorem applying on \(f^*\left(-\matA^{\top}y\right)=\sup_x\left[-f(x)-x^{\top}\matA^{\top}y\right]\), \(\nabla_{y} f^*\left(-\matA^{\top}y\right)|_{y=y^*}=-\matA x^*\). Hence, \(\matA x^* -b \in\partial \phi\left(y^*\right)\). Therefore, \((x^*,y^*)\) satisfies \eqref{eq:saddle_cond_2}.

Since \(\matA x^* -b \in\partial \phi\left(y^*\right)\), \(y^*\in\arg\max_{y} \left[{x^*}^{\top}\matA^{\top}y-b^{\top}y-\phi(y)\right]\). Hence, by Danskin's theorem, \(\matA^{\top}y^*\in \partial_x \sup_{y}\left[{x}^{\top}\matA^{\top}y-b^{\top}y-\phi(y)\right]|_{x=x^*}\), and thus 
\begin{equation*}
0=\gradf{x^*}+\matA^{\top} y^*\in \partial_x\sup_{y}F(x,y)|_{x=x^*}=\partial_x \Phi_x\left(x^*\right).
\end{equation*} 
In conclusion, there exists a unique {saddle point} \((x^*,y^*)\), and it satisfies \eqref{eq:saddle_cond_2}.

For \((x^*,y^*)\) satisfying \eqref{eq:saddle_cond_2}, we could know 
\(x^*\in \arg\min_{x} F\left(x,y^*\right)\), \(y^*\in \arg\max_{y} F\left(x^*,y\right)\). By Danskin's theorem, 
\begin{equation*}
    0\in\partial_x {F}\left(x,y^*\right)|_{x=x^*}=\partial_x {\Phi_x}\left(x^*\right)\quad \text{and}\quad 0 \in \partial_y {F}\left(x^*,y\right)|_{y=y^*}=\partial_y {\Phi_y}\left(y^*\right).
\end{equation*}
Hence, \((x^*,y^*)\) is the unique saddle point.
\end{proof}

\section{Proof of Theorem~\ref{thm:x-accelerate-prox-short}}
\label{appensec:proof x alg}
We begin by stating the full version of \Cref{thm:x-accelerate-prox-short}.
\begin{theorem}\label{thm:x-accelerate-prox}
Consider applying Algorithm~\ref{alg:x-accelerate-prox} to solve \eqref{eq:upper bound prob}. Let \begin{equation}\label{eq:x-accelerate-prox para}
\begin{aligned}
{\hat s} & = \frac{1}{s_{\max}^2}, 
& \quad \xgt & = \min\left(\frac{1}{5}, \frac{s_{\max}}{s_{\min}} \sqrt{\frac{\mu}{8L}} \right), \\
t & = \frac{1 - 4\xgt}{L + 4L\xgt}, 
& \quad s & = \frac{\hat s}{t}, \\
\xconv & = \max\left(\frac{s_{\max}^2}{s_{\min}^2} \cdot \frac{1}{2\xgt},\ \sqrt{\frac{1}{\mu t}}+\frac{L}{\mu}\cdot 4\xgt \right), 
& \quad \xe & = \frac{1 + 4L\xgt t}{1/\xconv + 4L\xgt t}, \\
\xtau & = \frac{\xe - 1}{1 - 1/\xconv}, 
& \quad \xga & = \frac{\xe - 1}{\xtau + 1}, \\
\Xi_f & = 1, 
& \quad \Xi_y & = \frac{1}{2s} = \frac{t}{2\hat s}, \\
\Xi_v & = \frac{1 + 4L\xgt t}{2\xe^2 t}, 
& \quad \xh & = 2\Xi_v \xe t = \frac{\Xi_v \xe t}{\Xi_y s}.
\end{aligned}
\end{equation}
Then, let \begin{equation*}
\Psi^k = {\Xi_y}\left\|y^k-y^*\right\|^2_{\left(I-(1-2\xgt){\hat s} \matA \matA^{\top}\right)}+\D{x^k}+{\Xi_v}\left\|{v}^k-x^*\right\|_2^2,
\end{equation*} where \({v}^k = (1+\xtau){z}^k-\xtau x^k\). We have \begin{equation*}\Psi^{k+1}\le (1-1/\xconv)\Psi^k.\end{equation*}
    
\end{theorem}
(Notice that with our choices of \(\xgt\) and \(\hat s\), the matrix \(I-(1-2\xgt){\hat s} \matA \matA^{\top}\) is positive definite. By the choice of \(\xgt\le\frac15
\) and \(t\), \(\sqrt{\frac{1}{\mu t}}\le3\sqrt{\frac{L}{\mu}}\). )

\begin{remark}\label{remark:range new}
Despite the incorporation of the proximal term (to allow for more general applications) in our Algorithm~\ref{alg:x-accelerate-prox}, a key difference between our Theorem~\ref{thm:x-accelerate-prox-short} and Proposition~1 regarding Algorithm~3 in~\citet{salim2022optimal} is that the latter does not require \(\matA\) to have full row rank—they only rely on the minimal positive singular value, rather than the smallest singular value. The reason for this difference is that, in the absence of a proximal term, the dual variable \(y\) remains in \(\mathrm{range}(\matA)\) throughout the iterations when initialized with \(y^0 = 0\), which is equivalent to having full rank \(\matA\) on the corresponding subspace. However, the introduction of the proximal mapping breaks this property, and thus our analysis requires \(\matA\) to have full row rank. This requirement also applies to Theorem~\ref{thm:y-accelerate-prox-short}. Notably, when the proximal term is trivial, i.e., \(\phi(y) = 0\), our results in Theorems~\ref{thm:x-accelerate-prox-short} and~\ref{thm:y-accelerate-prox-short} also hold when using the minimal positive singular value of \(\matA\).
\end{remark}

We divide the proof into several parts. Specifically, we analyze the individual components of \(\Psi^{k+1}\) separately in Propositions~\ref{prop:x-accelerate-prox-1}, \ref{prop:x-accelerate-prox-2}, and \ref{prop:x-accelerate-prox-3}, and then combine them to complete the full argument.

\begin{proposition}\label{prop:x-accelerate-prox-1}
Consider applying Algorithm~\ref{alg:x-accelerate-prox} to solve \eqref{eq:upper bound prob}. Using the parameters in \Cref{thm:x-accelerate-prox}, we have 
\begin{equation}\label{eq:x-acce-prox_df}
\begin{aligned}
\D{x^{k+1}} 
&\le (1 + 4L \xgt t)\, \D{z^k} 
- t \left(1 - \frac{L}{2} t - 2\xgt \right) 
  \left\| \gradf{z^k} + \matA^{\top} y^{k+1} \right\|_2^2 \\
&\quad - \xgt t \left\| \matA^{\top} \left(y^{k+1} - y^*\right) \right\|_2^2 
+ t \left\langle \matA^{\top} \left(y^{k+1} - y^*\right),\, 
  \gradf{z^k} + \matA^{\top} y^{k+1} \right\rangle.
\end{aligned}
\end{equation}
\end{proposition}

\begin{proof}
\begin{equation*}
\begin{aligned}
&\D{x^{k+1}} \\
&\le \D{z^k} 
- t \left\langle \gradf{z^k} - \gradf{x^*},\, 
              \gradf{z^k} + \matA^{\top} y^{k+1} \right\rangle 
+ \frac{L}{2} t^2 \left\| \gradf{z^k} + \matA^{\top} y^{k+1} \right\|_2^2 \\
&= \D{z^k} 
- t \left(1 - \frac{L}{2} t \right) 
  \left\| \gradf{z^k} + \matA^{\top} y^{k+1} \right\|_2^2 \\
&\quad + t \left\langle \matA^{\top} \left(y^{k+1} - y^*\right),\, 
              \gradf{z^k} + \matA^{\top} y^{k+1} \right\rangle \\
&\le \D{z^k} 
- t \left(1 - \frac{L}{2} t - 2\xgt \right) 
  \left\| \gradf{z^k} + \matA^{\top} y^{k+1} \right\|_2^2 
- \xgt t \left\| \matA^{\top} \left(y^{k+1} - y^*\right) \right\|_2^2 \\
&\quad + 2 \xgt t \left\| \gradf{z^k} - \gradf{x^*} \right\|_2^2 
+ t \left\langle \matA^{\top} \left(y^{k+1} - y^*\right),\, 
              \gradf{z^k} + \matA^{\top} y^{k+1} \right\rangle \\
&\le (1 + 4L \xgt t)\, \D{z^k} 
- t \left(1 - \frac{L}{2} t - 2\xgt \right) 
  \left\| \gradf{z^k} + \matA^{\top} y^{k+1} \right\|_2^2 \\
&\quad - \xgt t \left\| \matA^{\top} \left(y^{k+1} - y^*\right) \right\|_2^2 
+ t \left\langle \matA^{\top} \left(y^{k+1} - y^*\right),\, 
              \gradf{z^k} + \matA^{\top} y^{k+1} \right\rangle.
\end{aligned}
\end{equation*}
where the first inequality comes from the \(L\)-smoothness of \(f\); the first equality comes from the fact \(\gradf{x^*}+\matA^{\top} y^*=0\); the second inequality comes from \(\left\|a\right\|_2^2\le 2\left\|b\right\|_2^2+2\left\|a+b\right\|_2^2\) (\(a=\matA^{\top}(y^{k+1}-y^*)\) and \(b=\gradf{z^k}-\gradf{x^*}\)); and the third inequality comes from \begin{equation*}\left\|\gradf{z^k}-\gradf{x^*}\right\|_2^2\le 2L \D{z^k}.\end{equation*}
\end{proof}

\begin{proposition}\label{prop:x-accelerate-prox-2}
Consider applying Algorithm~\ref{alg:x-accelerate-prox} to solve \eqref{eq:upper bound prob}. Using the parameters in \Cref{thm:x-accelerate-prox}, we have 
\begin{equation}\label{eq:x-acce-prox_v_3}
\begin{aligned}
\left\|v^{k+1} - x^*\right\|_2^2
&\le \left(1 - \frac{1}{\xconv}\right) \left\|v^k - x^*\right\|_2^2
- 2\xe^2 t\, \D{z^k}
+ 2\xe(\xe - 1)t\, \D{x^k} \\
&\quad - 2\xe t \left\langle \matA^{\top} \left(y^{k+1} - y^*\right),\,
\hat{x}^k - x^* \right\rangle + \xe^2 t^2 \left\| \nabla f(z^k) + \matA^{\top} y^{k+1} \right\|_2^2.
\end{aligned}
\end{equation}
where \({v}^k = (1+\xtau){z}^k-\xtau x^k\)
\end{proposition}
\begin{proof}
For the term \(\left\|v^{k+1}-x^*\right\|_2^2\), we have the following equalities:
\begin{equation}\label{eq:x-acce-prox-2-init}
\begin{aligned}
&\left\|v^{k+1} - x^*\right\|_2^2 
= \left\|(1 + \xtau) z^{k+1} - \xtau x^{k+1} - x^*\right\|_2^2 
= \left\| \xe x^{k+1} - (\xe - 1) x^k - x^* \right\|_2^2 \\
=\ & \left\| \xe z^k - (\xe - 1) x^k - x^* \right\|_2^2 
- 2 \xe t \left\langle 
  \nabla f\left(z^k\right) - \gradf{x^*},\, 
  \xe z^k - (\xe - 1) x^k - x^* 
\right\rangle \\
&\quad - 2 \xe t \left\langle 
  \matA^{\top} \left( y^{k+1} - y^* \right),\, 
  \xe z^k - (\xe - 1) x^k - x^*
\right\rangle 
+ \xe^2 t^2 \left\| \nabla f\left(z^k\right) + \matA^{\top} y^{k+1} \right\|_2^2 \\
=\  &\left\| \xe z^k - (\xe - 1) x^k - x^* \right\|_2^2 
- 2 \xe t \left\langle 
  \nabla f\left(z^k\right) - \gradf{x^*},\, 
  \xe z^k - (\xe - 1) x^k - x^* 
\right\rangle \\
&\quad - 2 \xe t \left\langle 
  \matA^{\top} \left( y^{k+1} - y^* \right),\, 
  \hat{x}^k - x^* 
\right\rangle 
+ \xe^2 t^2 \left\| \nabla f\left(z^k\right) + \matA^{\top} y^{k+1} \right\|_2^2.
\end{aligned}
\end{equation}
where the second equality comes from \(\xe=1+\xga+\xga\xtau\) (see \eqref{eq:x-accelerate-prox para}), and the third equality comes from the momentum step of Algorithm~\ref{alg:x-accelerate-prox}. Then, by the \(\mu\)-strong convexity of \(f\) (and \(\D{x}\)), we have 
\begin{equation}\label{eq:x-acce-prox-2-mu-strong-convexity}
\begin{aligned}
\D{z^k}& - \left\langle \nabla f\left(z^k\right)-\gradf{x^*}, z^k-x^k\right\rangle   \le \D{x^k}\\
\D{z^k}& - \left\langle \nabla f\left(z^k\right)-\gradf{x^*}, z^k-x^*\right\rangle   \le -\frac{\mu}{2}\left\|z^k-x^*\right\|_2^2.\\
\end{aligned}
\end{equation}
Therefore, 
\begin{equation*}
\begin{aligned}
\left\|v^{k+1} - x^*\right\|_2^2 
&= \left\|\xe z^k - (\xe - 1) x^k - x^*\right\|_2^2 
- 2\xe t \left\langle  \gradf{z^k} - \gradf{x^*},\,
\xe z^k - (\xe - 1) x^k - x^* \right\rangle   \\
&\quad - 2\xe t \left\langle  \matA^{\top}(y^{k+1} - y^*),\, 
\hat{x}^k - x^* \right\rangle   
+ \xe^2 t^2 \left\|\gradf{z^k} + \matA^{\top} y^{k+1} \right\|_2^2 \\
&\le \left\|\xe z^k - (\xe - 1) x^k - x^*\right\|_2^2 
- 2\xe^2 t\, \D{z^k} \\
&\quad + 2\xe(\xe - 1)t \left[ \D{z^k} 
- \left\langle  \gradf{z^k} - \gradf{x^*},\, z^k - x^k \right\rangle   \right] \\
&\quad + 2\xe t \left[ \D{z^k} 
- \left\langle  \gradf{z^k} - \gradf{x^*},\, z^k - x^* \right\rangle   \right] \\
&\quad - 2\xe t \left\langle  \matA^{\top}(y^{k+1} - y^*),\, 
\hat{x}^k - x^* \right\rangle   
+ \xe^2 t^2 \left\|\gradf{z^k} + \matA^{\top} y^{k+1} \right\|_2^2 \\
&\le \left\|\xe z^k - (\xe - 1) x^k - x^*\right\|_2^2 
- \xe \mu t \left\|z^k - x^*\right\|_2^2 
- 2\xe^2 t\, \D{z^k} \\
&\quad + 2\xe(\xe - 1)t\, \D{x^k} 
- 2\xe t \left\langle  \matA^{\top}(y^{k+1} - y^*),\, 
\hat{x}^k - x^* \right\rangle   \\
&\quad + \xe^2 t^2 \left\|\gradf{z^k} + \matA^{\top} y^{k+1} \right\|_2^2.
\end{aligned}
\end{equation*}

Notice that \(\xconv \xe \mu t \ge 1\) by \eqref{eq:x-accelerate-prox para}. (\(\xconv \xe \mu t =\xconv \mu t \frac{1+4L\xgt t}{4L \xgt t + 1/\xconv}\ge\xconv \mu t \frac{1}{4L \xgt t + 1/\xconv}\). The positive solution to \(\xconv \mu t \frac{1}{4L \xgt t + 1/\xconv}=1\) is \(\xconv^{+}=\frac{4L \xgt t+\sqrt{(4L\xgt t)^2+4\mu t}}{2\mu t}\). By \eqref{eq:x-accelerate-prox para}, \(\xconv\ge \sqrt{\frac{1}{\mu t}}+4\frac{L}{\mu}\xgt\ge \xconv^{+}\).) Since \(\xe - 1 = (1-1/\xconv)\xtau\), by Jensen's inequality
\begin{equation*}\left\|\xe z^{k}-(\xe-1)x^k-x^*\right\|_2^2\le \left(1-\frac1{\xconv}\right)\left\|(1+\xtau)z^{k}-\xtau x^{k}-x^*\right\|_2^2 +\frac1{\xconv} \left\|z^k-x^*\right\|_2^2.\end{equation*}
Thus,
\begin{equation*}
\begin{aligned}
&\left\|v^{k+1} - x^*\right\|_2^2\\
&\le \left\| \xe z^k - (\xe - 1) x^k - x^* \right\|_2^2 
- \frac{1}{\xconv} \left\| z^k - x^* \right\|_2^2 
- 2\xe^2 t\, \D{z^k} 
+ 2\xe(\xe - 1)t\, \D{x^k} \\
&\quad - 2\xe t \left\langle \matA^{\top} \left( y^{k+1} - y^* \right),\, 
\hat{x}^k - x^* \right\rangle 
+ \xe^2 t^2 \left\| \nabla f\left(z^k\right) + \matA^{\top} y^{k+1} \right\|_2^2 \\
&\le \left( 1 - \frac{1}{\xconv} \right) 
\left\| \left(1 + \xtau \right) z^k - \xtau x^k - x^* \right\|_2^2 
- 2\xe^2 t\, \D{z^k} 
+ 2\xe(\xe - 1)t\, \D{x^k} \\
&\quad - 2\xe t \left\langle \matA^{\top} \left( y^{k+1} - y^* \right),\, 
\hat{x}^k - x^* \right\rangle 
+ \xe^2 t^2 \left\| \nabla f\left(z^k\right) + \matA^{\top} y^{k+1} \right\|_2^2 \\
&= \left( 1 - \frac{1}{\xconv} \right) \left\| v^k - x^* \right\|_2^2 
- 2\xe^2 t\, \D{z^k} 
+ 2\xe(\xe - 1)t\, \D{x^k} \\
&\quad - 2\xe t \left\langle \matA^{\top} \left( y^{k+1} - y^* \right),\, 
\hat{x}^k - x^* \right\rangle 
+ \xe^2 t^2 \left\| \nabla f\left(z^k\right) + \matA^{\top} y^{k+1} \right\|_2^2.
\end{aligned}
\end{equation*}
\end{proof}

\begin{proposition}\label{prop:x-accelerate-prox-3}
Consider applying Algorithm~\ref{alg:x-accelerate-prox} to solve \eqref{eq:upper bound prob}. Using the parameters in \Cref{thm:x-accelerate-prox}, we have 
\begin{equation}\label{eq:x-acce-prox_y}
\begin{aligned}
\left\|y^{k+1}-y^*\right\|_{\left(I-{\hat s}\matA \matA^{\top}\right)}^2&\le \left\|y^{k}-y^*\right\|_{\left(I-{\hat s}\matA \matA^{\top}\right)}^2+2\xh s\left\langle \matA^{\top}\left(y^{k+1}-y^*\right),{\hat x}^k - x^*\right\rangle   \\ &\quad -2{\hat s} \left\langle \matA^{\top}\left(y^{k+1}-y^*\right), \matA^{\top}y^{k+1}+\gradf{z^k}\right\rangle  .
\end{aligned}
\end{equation}
\end{proposition}
\begin{proof}
For the term \(\left\|y^{k+1}-y^*\right\|_{\left(I-{\hat s}\matA \matA^{\top}\right)}^2\), first let \begin{equation*}g^{k+1}=- \left[y^{k+1}-\left[y^k+\xh s\left({\matA}{\hat x} ^k-b\right)-\hat s{\matA}({\matA}^{\top}y^k+\nabla f({z}^k))\right]\right] / (\xh s).\end{equation*} Since \(y^{k+1}=\mathrm{prox}_{\xh {s}\phi}\left[y^k+\xh s\left({\matA}{\hat x} ^k-b\right)-\hat s{\matA}({\matA}^{\top}y^k+\nabla f({z}^k))\right]\), we have \(g^{k+1}\in \partial \phi\left(y^{k+1}\right)\).\\
Hence, from the y-update step in Algorithm~\ref{alg:x-accelerate-prox}, as
\begin{equation*}
- \hat s{\matA}\left(\matA^{\top} y^{k} + \gradf{z^{k}}\right)
= - \hat s{\matA}\left(\matA^{\top} y^{k+1} + \gradf{z^{k}}\right)
+ \hat s {\matA}\matA^{\top}\left(y^{k+1} - y^{k}\right),    
\end{equation*}
we have 
\begin{equation*}\begin{aligned}
&\left(I-\hat s \matA\matA^{\top}\right) \left(y^{k+1}-y^*\right)\\=&\left(I-\hat s \matA\matA^{\top}\right) \left(y^{k}-y^*\right) + \xh {s} \left(\matA {\hat x}^k - b\right) - {\hat s} \matA\left(\matA^{\top}y^{k+1}+\gradf{z^k}\right) -\xh {s} g^{k+1}\\
=&\left(I-\hat s \matA\matA^{\top}\right) \left(y^{k}-y^*\right) + \xh {s} \matA \left({\hat x}^k - x^*\right) - {\hat s} \matA\left(\matA^{\top}y^{k+1}+\gradf{z^k}\right) -\xh {s} \left(g^{k+1} - g^*\right),
\end{aligned}\end{equation*}
where \(g^*=\matA x^*-b\in\partial \phi(y^*)\) as defined in \Cref{sec:upper}. Then, 
\begin{equation*}
\begin{aligned}
&\left\| y^{k+1} - y^* \right\|^2_{\left( I - {\hat s} \matA \matA^{\top} \right)} \\
\le\ 
&\left\| y^k - y^* \right\|^2_{\left( I - {\hat s} \matA \matA^{\top} \right)} 
+ 2 \xh s \left\langle y^{k+1} - y^*,\, \matA \left( \hat x^k - x^* \right) \right\rangle \\
&- 2 {\hat s} \left\langle y^{k+1} - y^*,\, \matA \left( \matA^{\top} y^{k+1} + \gradf{z^k} \right) \right\rangle - 2 \xh s \left\langle y^{k+1} - y^*,\, g^{k+1} - g^* \right\rangle \\
\le\ 
&\left\| y^k - y^* \right\|^2_{\left( I - {\hat s} \matA \matA^{\top} \right)} 
+ 2 \xh s \left\langle \matA^{\top} \left( y^{k+1} - y^* \right),\, \hat x^k - x^* \right\rangle \\
&- 2 {\hat s} \left\langle \matA^{\top} \left( y^{k+1} - y^* \right),\, \matA^{\top} y^{k+1} + \gradf{z^k} \right\rangle,
\end{aligned}
\end{equation*}
where the last inequality comes from the monotonicity of subgradients of the convex function \(\phi(y)\).
\end{proof}

\begin{proof}[Proof of \Cref{thm:x-accelerate-prox}]
Consider the term \({\Xi_y} \left\|y^{k+1}-y^*\right\|^2_{\left(I-{\hat s} \matA \matA^{\top}\right)}+\D{x^{k+1}}+{\Xi_v}\left\|{v}^{k+1}-x^*\right\|_2^2\). Combine \eqref{eq:x-acce-prox_df}, \eqref{eq:x-acce-prox_v_3}, and \eqref{eq:x-acce-prox_y}, and by the settings of the weights \({\Xi_f}\), \({\Xi_y}\), \({\Xi_v}\), and other parameters in \eqref{eq:x-accelerate-prox para}, we see that on the right hand side: \begin{itemize}
\item{the term \(\left\langle \matA^{\top}\left(y^{k+1}-y^*\right), \matA^{\top}y^{k+1}+\gradf{z^k}\right\rangle  \) has coefficient \({\Xi_f}{t} - {\Xi_y} 2\hat s=0\);}
\item{the term \(\left\langle \matA^{\top}\left(y^{k+1}-y^*\right),{\hat x}^k - x^*\right\rangle  \) has coefficient \(-{\Xi_v} 2 \xe t + {\Xi_y} 2 \xh s=0\);}
\item{the term \(\left\|\gradf{z^k}+\matA^{\top} y^{k+1}\right\|_2^2\) has coefficient \({\Xi_f} [-t(1-\frac{L}{2}{t}-2\xgt) ] + {\Xi_v} \xe^2 {t}^2 = t(-1+\frac{L}{2}{t}+2\xgt + \frac12 (1+4L\xgt t))=0\);}
\item{the term \(\D{z^k}\) has coefficient \({\Xi_f}(1+4L\xgt t) - {\Xi_v}2\xe^2t=0\).}
\end{itemize}
Also, the coefficient of \(\D{x^k}\) is \({\Xi_v} 2\xe (\xe-1) t = 1-\frac{1}{\xconv} = {\Xi_f}(1-1/\xconv)\). Hence, by moving \(-\xgt t\left\|\matA^{\top}\left(y^{k+1}-y^*\right)\right\|_2^2\) to the left hand side,
\begin{equation*}
\begin{aligned}
&\Xi_y \left\| y^{k+1} - y^* \right\|^2_{\left( I - (1-2\xgt)\hat{s} \matA \matA^{\top} \right)} 
+ \D{x^{k+1}} 
+ \Xi_v \left\| v^{k+1} - x^* \right\|_2^2 \\
\le\ & \Xi_y \left\| y^k - y^* \right\|^2_{\left( I - \hat{s} \matA \matA^{\top} \right)} 
+ \left( 1 - \frac{1}{\xconv} \right) \D{x^k} 
+ \Xi_v \left( 1 - \frac{1}{\xconv} \right) \left\| v^k - x^* \right\|_2^2  \\
\le\ & \Xi_y \left\| y^k - y^* \right\|^2_{\left( I - (1-2\xgt)\hat{s} \matA \matA^{\top} \right)} 
+ \left( 1 - \frac{1}{\xconv} \right) \D{x^k} 
+ \Xi_v \left( 1 - \frac{1}{\xconv} \right) \left\| v^k - x^* \right\|_2^2 \\
&\quad - 2 \Xi_y \xgt \hat{s} s_{\min}^2 \left\| y^{k} - y^* \right\|_2^2.
\end{aligned}
\end{equation*}
Since \(I-(1-2\xgt)\hat s\matA\matA^\top\preceq I\) and \(\xconv \ge \frac{1}{2\xgt \hat s s_{\min}^2}\), we have
\begin{equation*}
2\xgt\hat s s_{\min}^2\left\|y^k-y^*\right\|_2^2 \ge \frac1{\xconv}\left\|y^k-y^*\right\|^2_{\left(I-(1-2\xgt)\hat s\matA\matA^\top\right)},
\end{equation*}
and hence the last line implies \(\Psi^{k+1}\le (1-1/\xconv)\Psi^k\).
\end{proof}

\section{Proof of Theorem~\ref{thm:y-accelerate-prox-short}}
\label{appensec:proof y alg}
We begin by stating the full version of \Cref{thm:y-accelerate-prox-short}.
\begin{theorem}\label{thm:y-accelerate-prox}
Consider applying Algorithm~\ref{alg:y-accelerate-prox} to solve \eqref{eq:upper bound prob}. Let 
\begin{equation}\label{eq:y-accelerate-prox para}
\begin{aligned}
\hat s & = \frac{1}{s_{\max}^2}, 
& \quad t & = \frac{1}{2L}, \\
\alpha &= \frac12,
&\yc & = \max\left(1, \frac{1}{\sqrt 2}\frac{s_{\max}}{s_{\min}}\sqrt{\frac{\mu}{L}}\right), \\
\tilde t & = \frac{\alpha t}{\yc}, & \quad \yconv &\ge \max\left(\frac{2}{\yc}\frac{s_{\max}^2}{s_{\min}^2}, 4\yc\frac{L}{\mu}\right),\\
\yeta & = \frac{\yc - 1}{1 - 1/\yconv}, 
& \quad \yq & = \frac{\yc - 1}{\yeta + 1}, \\
\quad s & = \frac{\hat s}{t}, 
& \quad \Xi_u & = \frac{1}{2\yc^2 \hat s}, \\
\Xi_x & = \Xi_u \cdot \frac{\yc s}{\tilde t} = \frac{1}{t^2}, & \quad \Xi_h & = 1.
\end{aligned}
\end{equation}
Then, let 
\begin{equation*}
\begin{aligned}
\Psi^k =\ &\Xi_x \cdot \left[\left\| x^k - x^*\right\|_2^2-2\left(t-\tilde t\right)\D{x^k}\right] \\
&+ H(y^k) + \frac{1}{t} \Dphi(y^k, y^*) + \Xi_u \left\| u^k - y^* \right\|^2_{ \left(I - \alpha{\hat s} \matA \matA^{\top} \right)},
\end{aligned}
\end{equation*}
we have 
\begin{equation*}\Psi^{k+1}\le (1-1/\yconv)\Psi^k.\end{equation*}
\end{theorem}
(Notice that with our choices of \(\alpha\), \(\hat s\), \(t\), and \(\tilde t\), the matrix
\(I-\alpha{\hat s}\matA\matA^{\top}\) is positive definite, and \(\left\| x^k - x^*\right\|_2^2-2\left(t-\tilde t\right)\D{x^k}\ge \left(1-L\left(t-\tilde t\right)\right)\left\| x^k - x^*\right\|_2^2\ge 0\).)

\begin{remark}\label{remark:num M_MT_grad}
To clearly illustrate the first-order updates and momentum steps in Algorithms~\ref{alg:x-accelerate-prox} and \ref{alg:y-accelerate-prox}, we explicitly write out multiple multiplications involving \(\matA\) and \(\matA^{\top}\), as well as gradient computations. However, each iteration of both algorithms only requires one multiplication with \(\matA\), one with \(\matA^{\top}\), and one gradient evaluation, matching the per-iteration cost of Algorithm 3 in \citet{salim2022optimal} and the scheme in equations (2.3)–(2.5) of \citet{drori2015simple}. For instance, in Algorithm~\ref{alg:y-accelerate-prox}, \(\matA^{\top}y^k\) can be precomputed and reused to recover \(\matA^{\top} u^k\) and \(\matA^{\top} w^k\) from \(\matA^{\top}y^k\) and \(\matA^{\top}y^{k-1}\).
\end{remark}

We divide the proof into several parts. Specifically, we analyze the individual components of \(\Psi^{k+1}\) in Propositions~\ref{prop:y-accelerate-prox-1}, \ref{prop:y-accelerate-prox-2}, and \ref{prop:y-accelerate-prox-3}, and then combine them to complete the full argument.

\begin{proposition}\label{prop:y-accelerate-prox-1}
Consider applying Algorithm~\ref{alg:y-accelerate-prox} to solve \eqref{eq:upper bound prob}. Using the parameters in \Cref{thm:y-accelerate-prox}, we have 
\begin{equation}\label{eq:y-acce-prox_x}
\begin{aligned}
  &\left\|x^{k+1} - x^*\right\|_2^2
    - 2 \left(t - \tilde t\right)\, \D{x^{k+1}}\\
  &\qquad \le
    \left[
      \left\|x^{k} - x^*\right\|_2^2
      - 2 \left(t - \tilde t\right)\, \D{x^{k}}
    \right]
    \left(1 - \frac{1}{\yconv}\right)\\
  &\qquad\quad
    - 2 \tilde t \left\langle
        x^{k} - x^* - t\left(\gradf{x^{k}} - \gradf{x^*}\right),
        \matA^\top \left(u^{k+1} - y^*\right)
      \right\rangle
    + \tilde t^2 \left\| \matA^\top \left(u^{k+1} - y^*\right) \right\|_2^2 .
\end{aligned}
\end{equation}
\end{proposition}

\begin{proof}
From the $x$-update
\(
  x^{k+1} = x^k - \tilde t\left(\gradf{x^k} + \matA^\top u^{k+1}\right)
\)
and the optimality condition
\(
  \gradf{x^*} + \matA^\top y^* = 0
\),
we have
\begin{equation*}
\begin{aligned}
  \left\|x^{k+1} - x^*\right\|_2^2
  &= \left\|
       x^k - x^*
       - \tilde t\left(\gradf{x^k} - \gradf{x^*}\right)
       - \tilde t \matA^\top \left(u^{k+1} - y^*\right)
     \right\|_2^2 \\
  &= \left\|
       x^k - x^*
       - \tilde t\left(\gradf{x^k} - \gradf{x^*}\right)
     \right\|_2^2 \\
  &\quad
     - 2 \tilde t \left\langle
       x^k - x^* - \tilde t\left(\gradf{x^k} - \gradf{x^*}\right),
       \matA^\top \left(u^{k+1} - y^*\right)
     \right\rangle \\
  &\quad
     + \tilde t^2 \left\|\matA^\top \left(u^{k+1} - y^*\right)\right\|_2^2 .
\end{aligned}
\end{equation*}

Moreover, by the definition of \(D_f\),
\begin{equation*}
\begin{aligned}
  \D{x^{k+1}}
  &= f\left(x^{k+1}\right) - f\left(x^*\right)
     - \left\langle \gradf{x^*}, x^{k+1} - x^*\right\rangle \\
  &\ge
    \D{x^{k}}
    + \left\langle \gradf{x^k} - \gradf{x^*},
                 x^{k+1} - x^k \right\rangle \\
  &= \D{x^{k}}
    - \tilde t
      \left\langle
        \gradf{x^k} - \gradf{x^*},
        \gradf{x^k} + \matA^\top u^{k+1}
      \right\rangle.
\end{aligned}
\end{equation*}

Consider now
\begin{equation*}
  \left\|x^{k+1} - x^*\right\|_2^2 - 2 \left(t - \tilde t\right)\, \D{x^{k+1}}.
\end{equation*}
Using the previous two relations, we have
\begin{equation*}
\begin{aligned}
  &\left\|x^{k+1} - x^*\right\|_2^2 - 2 \left(t - \tilde t\right)\, \D{x^{k+1}} \\
  &\le
    \left\|
        x^k - x^*
        - \tilde t \left(\gradf{x^k} - \gradf{x^*}\right)
     \right\|_2^2
    + \tilde t^2 \left\|\matA^\top \left(u^{k+1} - y^*\right)\right\|_2^2 \\
  &\quad
    - 2 \tilde t
      \left\langle
        x^k - x^* - t\left(\gradf{x^k} - \gradf{x^*}\right),
        \matA^\top \left(u^{k+1} - y^*\right)
      \right\rangle \\
  &\quad
    + 2\left(t - \tilde t\right)\tilde t
      \left\|\gradf{x^k} - \gradf{x^*}\right\|_2^2
    - 2 \left(t - \tilde t\right)\, \D{x^k}.
\end{aligned}
\end{equation*}

Now,
\begin{equation*}
\begin{aligned}
  &\left\|
     x^k - x^*
     - \tilde t \left(\gradf{x^k} - \gradf{x^*}\right)
   \right\|_2^2
   + 2\left(t - \tilde t\right)\tilde t
     \left\|\gradf{x^k} - \gradf{x^*}\right\|_2^2 \\
  &\le
    \left\|x^k - x^*\right\|_2^2
    - 2\tilde t
      \left\langle
        \gradf{x^k} - \gradf{x^*},
        x^k - x^*
      \right\rangle
    + 2 t \tilde t
      \left\|\gradf{x^k} - \gradf{x^*}\right\|_2^2 \\
  &\le
    \left\|x^k - x^*\right\|_2^2
    - 2\tilde t \left(1 - L t\right)
      \left\langle
        \gradf{x^k} - \gradf{x^*},
        x^k - x^*
      \right\rangle ,
\end{aligned}
\end{equation*}
where in the last inequality we used \(L\)-smoothness of \(f\),
\begin{equation*}
  \left\|\gradf{x^k} - \gradf{x^*}\right\|_2^2
  \le L \left\langle \gradf{x^k} - \gradf{x^*}, x^k - x^*\right\rangle .
\end{equation*}

By \(\mu\)-strong convexity of \(f\),
\begin{equation*}
  \left\langle
    \gradf{x^k} - \gradf{x^*},
    x^k - x^*
  \right\rangle
  \ge \mu \left\|x^k - x^*\right\|_2^2.
\end{equation*}
Hence
\begin{equation*}
\begin{aligned}
&\left\|
     x^k - x^*
     - \tilde t \left(\gradf{x^k} - \gradf{x^*}\right)
   \right\|_2^2
   + 2\left(t - \tilde t\right)\tilde t
     \left\|\gradf{x^k} - \gradf{x^*}\right\|_2^2\\
  &\le
    \left\|x^k - x^*\right\|_2^2
    - 2 \mu \tilde t \left(1 - L t\right)\, \left\|x^k - x^*\right\|_2^2 \\
  &\le
    \left(1 - \mu \tilde t\right)\, \left\|x^k - x^*\right\|_2^2,
\end{aligned}
\end{equation*}
where in the last step we used $t = \tfrac{1}{2L}$ so that
$2(1 - Lt) = 1$. 
Since \(\tilde t=\frac{\alpha t}{\yc}\) with \(t=\frac1{2L}\) and \(\alpha=\frac12\), we have
\(\mu\tilde t=\frac{\mu}{4\yc L}\ge \frac1{\yconv}\) because \(\yconv\ge 4\yc\frac{L}{\mu}\).
Moreover, \(\D{x^k}\ge 0\) implies
\begin{equation*}
(1-\mu\tilde t)\left\|x^k-x^*\right\|^2 -2(t-\tilde t)\D{x^k}
\le \left(1-\frac1{\yconv}\right)\left[\left\|x^k-x^*\right\|^2-2(t-\tilde t)\D{x^k}\right].    
\end{equation*}
\end{proof}

\begin{proposition}\label{prop:y-accelerate-prox-2}
Let \begin{equation*}g^{k+1}=- \left[y^{k+1}-\left[w^k+s\left({\matA}{x} ^k-b\right)-\hat s{\matA}({\matA}^{\top}w^k+\nabla f({x}^k))\right]\right] /s.\end{equation*}
Consider applying Algorithm~\ref{alg:y-accelerate-prox} to solve \eqref{eq:upper bound prob}. Using the parameters in \Cref{thm:y-accelerate-prox}, we have 
\begin{equation}\label{eq:y-acce-prox_H}
\begin{aligned}
H\left( y^{k+1} \right)
\le\ 
&H\left( w^k \right) 
- \left( \frac{1}{{\hat s}} - \frac{s_{\max}^2}{2} \right) 
\left\| y^{k+1} - w^k \right\|_2^2 \\
&+ \left\langle 
\frac{s}{{\hat s}} \matA \left( x^k - x^* \right)
- \matA \left( \gradf{{x}^k} - \gradf{x^*} \right)
- \frac{s}{{\hat s}} \left( g^{k+1} - g^* \right),\, 
y^{k+1} - w^k 
\right\rangle,
\end{aligned}
\end{equation}
where \(g^*=\matA x^*-b\in\partial \phi(y^*)\) as defined in \Cref{sec:upper}.
\end{proposition}
\begin{proof}
For the term \(H\left(y^{k+1}\right)\), since \(H\) is \(s_{\max}^2\)-smooth, and \(\nabla H\left(w^k\right)=\matA\matA^{\top}\left(w^k-y^*\right),\)
\begin{equation*}
\begin{aligned}
H\left( y^{k+1} \right)
\le\ 
&H\left( w^k \right) 
+ \left\langle 
\matA \matA^{\top} \left( w^k - y^* \right),\,
y^{k+1} - w^k 
\right\rangle 
+ \frac{s_{\max}^2}{2} 
\left\| y^{k+1} - w^k \right\|_2^2 \\
=\ 
&H\left( w^k \right) 
- \left( \frac{1}{{\hat s}} - \frac{s_{\max}^2}{2} \right)
\left\| y^{k+1} - w^k \right\|_2^2 \\
&+ \left\langle 
\frac{s}{{\hat s}} \matA \left( x^k - x^* \right) 
- \matA \left( \gradf{{x}^k} - \gradf{x^*} \right) 
- \frac{s}{{\hat s}} \left( g^{k+1} - g^* \right),\,
y^{k+1} - w^k 
\right\rangle.
\end{aligned}
\end{equation*}
where in the equality we apply Proposition~
\ref{prop:saddle_cond_2} and the definition of \(g^{k+1}\).
\end{proof}

\begin{proposition}\label{prop:y-accelerate-prox-3}
Let \begin{equation*}g^{k+1}=- \left[y^{k+1}-\left[w^k+s\left({\matA}{x} ^k-b\right)-\hat s{\matA}({\matA}^{\top}w^k+\nabla f({x}^k))\right]\right] /s.\end{equation*} Consider applying Algorithm~\ref{alg:y-accelerate-prox} to solve \eqref{eq:upper bound prob}. Using the parameters in \Cref{thm:y-accelerate-prox}, we have 
\begin{equation}\label{eq:y-acce-prox_u_3}
\begin{aligned}
&\left\|{u}^{k+1}-y^*\right\|_2^2\\
\le\ &\left(1-\frac1{\yconv}\right)\left\|{u}^k-y^*\right\|_{\left(I-\alpha{\hat s}\matA \matA^{\top}\right)}^2 + 2(\yc-\alpha)(\yc-1)\hat s H\left(y^k\right)\\ 
&+ 2\yc\left\langle \yc{w}^k-(\yc-1)y^k-y^*,s\matA \left(x^{k}-x^*\right)-{\hat s} \matA \left( \gradf{{x}^{k}}-\gradf{x^*}\right)-s\left(g^{k+1}-g^*\right)\right\rangle\\
&+\yc^2\left\|y^{k+1}-w^k\right\|_2^2 - \left[\left(\yc^2 +(1-\alpha) \yc\right)\hat s-\frac1{\yconv s_{\min}^2}\right]\left\|\matA^{\top}\left(w^k-y^*\right)\right\|_2^2.
\end{aligned}
\end{equation}
\end{proposition}
\begin{proof}
For the term \(\left\|{u}^{k+1}-y^*\right\|_2^2\), by the definitions of \(\yc\) and \(\yeta\), we have the following equalities:
\begin{equation}\label{eq:y-acce-prox_u_start}
\begin{aligned}
&\left\|{u}^{k+1}-y^*\right\|_2^2 =\left\|(1+\yeta){w}^{k+1} - \yeta y^{k+1}-y^*\right\|_2^2 = \left\|\yc {y}^{k+1} - (\yc-1) y^{k}-y^*\right\|_2^2\\
=\ & \left\|\yc {w}^{k} - (\yc-1) y^{k}-y^*\right\|_2^2 + 2\yc\left\langle \yc {w}^{k} - (\yc-1) y^{k}-y^*,y^{k+1}-w^k\right\rangle  + \yc^2\left\|y^{k+1}-w^k\right\|_2^2\\
=\ & \left\|\yc {w}^{k} - (\yc-1) y^{k}-y^*\right\|_2^2 - 2\yc\left\langle \yc {w}^{k} - (\yc-1) y^{k}-y^*,{\hat s} \matA\matA^{\top} \left( {w}^k -y^*\right)\right\rangle\\
& +2\yc\left\langle \yc {w}^{k} - (\yc-1) y^{k}-y^*,s\matA \left(x^{k}-x^*\right)-{\hat s} \matA \left( \gradf{{x}^{k}}-\gradf{x^*}\right)-s\left(g^{k+1}-g^*\right)\right\rangle\\
&+\yc^2\left\|y^{k+1}-w^k\right\|_2^2,
\end{aligned}   
\end{equation}
and 
\begin{equation}\label{eq:y-acce-prox_u_start_2}
\begin{aligned}
&\left\|\matA^{\top}\left(u^k-y^*\right)\right\|_2^2=\left\|\matA^{\top}\left((1+\yeta)w^k-\yeta y^k-y^*\right)\right\|_2^2\\
=\ & \yeta^2\left\|\matA^{\top}\left(w^k-y^k\right)\right\|_2^2+ \left\|\matA^{\top}\left(w^k-y^*\right)\right\|_2^2+2\yeta \left\langle \matA\matA^{\top}\left(w^k-y^k\right),w^k-y^*\right\rangle  .
\end{aligned}
\end{equation}
Notice that 
\begin{equation}\label{eq:y-acce-prox_u_start_3}
    2\left\langle \matA\matA^{\top}\left(w^k-y^k\right),w^k-y^*\right\rangle  =\left\|\matA^{\top}\left(w^k-y^*\right)\right\|_2^2+\left\|\matA^{\top}\left(w^k-y^k\right)\right\|_2^2-\left\|\matA^{\top}\left(y^k-y^*\right)\right\|_2^2.
\end{equation}
Hence, combining \eqref{eq:y-acce-prox_u_start}, \eqref{eq:y-acce-prox_u_start_2}, and \eqref{eq:y-acce-prox_u_start_3}, and using the setting \(\yeta=\frac{\yc-1}{1-1/\yconv}\),  we have

\begin{equation}\label{eq:y-acce-prox_long_ineq}
\begin{aligned}
&\left\| u^{k+1} - y^* \right\|_2^2 
+ \left( 1 - \frac{1}{\yconv} \right) \alpha {{\hat s}} 
\left\| \matA^{\top} \left( u^k - y^* \right) \right\|_2^2 \\
=\ 
&\left\| \yc w^k - (\yc - 1) y^k - y^* \right\|_2^2 
+ \yc^2 \left\| y^{k+1} - w^k \right\|_2^2 \\
&+ 2 \yc \left\langle 
\yc w^k - (\yc - 1) y^k - y^*,\,
s \matA \left( x^k - x^* \right) 
- {\hat s} \matA \left( \gradf{x^k} - \gradf{x^*} \right) 
- s \left( g^{k+1} - g^* \right)
\right\rangle \\
&- 2 \yc (\yc - 1) {\hat s} 
\left\langle w^k - y^k,\,
\matA \matA^{\top} \left( w^k - y^* \right) \right\rangle 
- 2 \yc {\hat s} 
\left\| \matA^{\top} \left( w^k - y^* \right) \right\|_2^2 \\
&+ \left( 1 - \frac{1}{\yconv} \right) \alpha {\hat s} 
\Big[ \yeta^2 \left\| \matA^{\top} \left( w^k - y^k \right) \right\|_2^2 
+ \left\| \matA^{\top} \left( w^k - y^* \right) \right\|_2^2 \\
&+ 2 \yeta \left\langle 
\matA \matA^{\top} \left( w^k - y^k \right),\, 
w^k - y^* \right\rangle \Big] \\
=\ 
&\left\| \yc w^k - (\yc - 1) y^k - y^* \right\|_2^2 
+ \yc^2 \left\| y^{k+1} - w^k \right\|_2^2 \\
&+ 2 \yc \left\langle 
\yc w^k - (\yc - 1) y^k - y^*,\,
s \matA \left( x^k - x^* \right) 
- {\hat s} \matA \left( \gradf{x^k} - \gradf{x^*} \right) 
- s \left( g^{k+1} - g^* \right)
\right\rangle \\
&- 2 \left( \yc - \alpha \right) (\yc - 1) {\hat s} 
\left\langle w^k - y^k,\,
\matA \matA^{\top} \left( w^k - y^* \right) \right\rangle \\
&+ \alpha {{\hat s}} \yeta (\yc - 1) 
\left\| \matA^{\top} \left( w^k - y^k \right) \right\|_2^2 
- \left[ 2 \yc {\hat s} - \left( 1 - \frac{1}{\yconv} \right) \alpha {{\hat s}}\right]
\left\| \matA^{\top} \left( w^k - y^* \right) \right\|_2^2 \\
=\ 
&\left\| \yc w^k - (\yc - 1) y^k - y^* \right\|_2^2 
+ \yc^2 \left\| y^{k+1} - w^k \right\|_2^2 \\
&+ 2 \yc \left\langle 
\yc w^k - (\yc - 1) y^k - y^*,\,
s \matA \left( x^k - x^* \right) 
- {\hat s} \matA \left( \gradf{x^k} - \gradf{x^*} \right) 
- s \left( g^{k+1} - g^* \right)
\right\rangle \\
&+ \left( \yc - \alpha \right) (\yc - 1) {\hat s} 
\left\| \matA^{\top} \left( y^k - y^* \right) \right\|_2^2\\ &- \left[ \left( \yc - \alpha \right) (\yc - 1) 
- \alpha \yeta (\yc - 1) \right] {\hat s} 
\left\| \matA^{\top} \left( w^k - y^k \right) \right\|_2^2 \\
&- \left[ \left( \yc - \alpha \right) (\yc - 1) {\hat s} 
+ 2 \yc {\hat s} 
- \left( 1 - \frac{1}{\yconv} \right) \alpha{{\hat s}} \right] 
\left\| \matA^{\top} \left( w^k - y^* \right) \right\|_2^2 \\
\le\ 
&\left\| \yc w^k - (\yc - 1) y^k - y^* \right\|_2^2 
+ \yc^2 \left\| y^{k+1} - w^k \right\|_2^2 \\
&+ 2 \yc \left\langle 
\yc w^k - (\yc - 1) y^k - y^*,\,
s \matA \left( x^k - x^* \right) 
- {\hat s} \matA \left( \gradf{x^k} - \gradf{x^*} \right) 
- s \left( g^{k+1} - g^* \right)
\right\rangle \\
&+ \left( \yc - \alpha \right) (\yc - 1) {\hat s} 
\left\| \matA^{\top} \left( y^k - y^* \right) \right\|_2^2 - \left( \yc^2 + (1-\alpha) \yc \right) {\hat s} 
\left\| \matA^{\top} \left( w^k - y^* \right) \right\|_2^2,
\end{aligned}
\end{equation}
where \(\left(\yc-\alpha\right)(\yc-1)-\alpha(\yc -1)\yeta\ge 0\) comes from \(\yc\ge \max\left(1,\alpha/(1-\alpha)\right)\), \(\yconv\ge \yc\), and \(\yeta=\frac{\yc-1}{1-1/\yconv}\le \yc\).

Since \(\yc - 1 = (1-1/\yconv)\yeta\), by Jensen's inequality
\begin{equation*}\begin{aligned}
\left\|\yc w^{k}-(\yc-1)y^k-y^*\right\|_2^2&\le \left(1-\frac1{\yconv}\right)\left\|(1+\yeta)w^{k}-\yeta y^{k}-y^*\right\|_2^2 +\frac1{\yconv} \left\|{w}^k-y^*\right\|_2^2\\
&\le \left(1-\frac1{\yconv}\right)\left\|{u}^k-y^*\right\|_2^2 +\frac1{\yconv s_{\min}^2} \left\|\matA^{\top}\left({w}^k-y^*\right)\right\|_2^2.
\end{aligned}\end{equation*}
Thus, combining all these, we have
\begin{equation*}
\begin{aligned}
& \left\|{u}^{k+1}-y^*\right\|_2^2 +\left(1-\frac{1}{\yconv}\right)\alpha {{\hat s}} \left\|\matA^{\top}\left(u^k-y^*\right)\right\|_2^2\\
\le\ &(1-\frac1{\yconv})\left\|{u}^k-y^*\right\|_2^2+\yc^2\left\|y^{k+1}-w^k\right\|_2^2\\
 &+2\yc\left\langle \yc {w}^{k} - (\yc-1) y^{k}-y^*,s\matA \left(x^{k}-x^*\right)-{\hat s} \matA \left( \gradf{{x}^{k}}-\gradf{x^*}\right)-s(g^{k+1}-g^*)\right\rangle\\
  &+\left(\yc-\alpha\right)(\yc-1)\hat s \left\|\matA^{\top}\left(y^k-y^*\right)\right\|_2^2 - \left[\left(\yc^2 +(1-\alpha) \yc\right)\hat s-\frac1{\yconv s_{\min}^2}\right]\left\|\matA^{\top}\left(w^k-y^*\right)\right\|_2^2,
\end{aligned}
\end{equation*}
which concludes the proof for Proposition~\ref{prop:y-accelerate-prox-3}.
\end{proof}

\begin{proof}[Proof of \Cref{thm:y-accelerate-prox}]
Because \(y^{k+1}=\mathrm{prox}_{s\phi}\left[w^k+s\left(\matA{x} ^{k}-b\right)-\hat s{\matA}\left({\matA}^{\top}w^k+\gradf{{x}^{k}}\right)\right]\), \(g^{k+1}\in\partial \phi\left(y^{k+1}\right)\). Hence, \(g^{k+1}-g^*\in\partial \Dphi\left(y^{k+1},y^*\right)\), and by the convexity of \(\Dphi\left(y,y^*\right)\),
\begin{equation*}
\begin{aligned}
&\Dphi\left(y^{k+1},y^*\right)-\left\langle g^{k+1}-g^*, y^{k+1}-\left(1-\frac{1}{\yc}\right)y^{k}-\frac{1}{\yc}y^*\right\rangle\\  \le\ &\left(1-\frac{1}{\yc}\right)\Dphi\left(y^k,y^*\right)\le \left(1-\frac{1}{\yconv}\right)\Dphi\left(y^k,y^*\right).
\end{aligned}
\end{equation*}
Since \(2\Xi_u\yc^2\hat s = \Xi_h=1\), and \(s = \hat s / t\),
\begin{equation*}
\begin{aligned}
&2 \Xi_u \yc \left\langle 
\yc w^k - (\yc - 1) y^k - y^*,\,
s \matA \left( x^k - x^* \right) 
- {\hat s} \matA \left( \gradf{x^k} - \gradf{x^*} \right)
- s \left( g^{k+1} - g^* \right)
\right\rangle \\
&+ \Xi_h \left\langle 
\frac{s}{{\hat s}} \matA \left( x^k - x^* \right) 
- \matA \left( \gradf{x^k} - \gradf{x^*} \right)
- \frac{s}{{\hat s}} \left( g^{k+1} - g^* \right),\,
y^{k+1} - w^k
\right\rangle \\
=\ 
&\frac{1}{\yc t} \left\langle 
\matA \left( x^k - x^* \right) 
- t \matA \left( \gradf{x^k} - \gradf{x^*} \right),\,
u^{k+1} - y^* 
\right\rangle \\
&\quad - \frac{1}{t} \left\langle 
g^{k+1} - g^*,\,
y^{k+1} - \left( 1 - \frac{1}{\yc} \right) y^k 
- \frac{1}{\yc} y^*
\right\rangle.
\end{aligned}
\end{equation*}
Combine \eqref{eq:y-acce-prox_x}, \eqref{eq:y-acce-prox_H}, and \eqref{eq:y-acce-prox_u_3}, and by the settings of the weights \({\Xi_x}\), \({\Xi_h}\), \({\Xi_u}\), and other parameters in \eqref{eq:y-accelerate-prox para}, we see that on the right hand side: \begin{itemize}
\item{the term \(\left\|\matA^{\top}\left(u^{k+1}-y^*\right)\right\|_2^2\) has a coefficient of \(\Xi_x \tilde t^2=\Xi_u\alpha{{\hat s}}\);}
\item{the term \(\left\langle x^{k}-x^*-t\left(\gradf{{x}^{k}}-\gradf{x^*}\right),\matA^{\top}\left(u^{k+1}-y^*\right)\right\rangle  \) has a coefficient of \(-{\Xi_x} 2  {\tilde t} + \frac{1}{\yc t}=0\);}
\item{the term \(\left\|y^{k+1}-w^k\right\|_2^2\) has a coefficient of \({\Xi_h} \left(-\frac1{\hat s} + \frac{s_{\max}^2}{2}\right) + {\Xi_u} \yc^2 = 0\);}
\item the coefficient in front of \(\|\matA^\top(w^k-y^*)\|_2^2\) equals
\begin{equation*}
\frac{1}{2}
-\Xi_u\left[\big(\yc^2+(1-\alpha)\yc\big)\hat s-\frac{1}{\yconv s_{\min}^2}\right]
= -\frac{1-\alpha}{2\yc}+\frac{s_{\max}^2/s_{\min}^2}{2\yc^2\yconv}\le 0,
\end{equation*}
where we used \(\hat s=1/s_{\max}^2\) and \(\yconv\ge \frac{2}{\yc}\frac{s_{\max}^2}{s_{\min}^2}\).
\end{itemize}
Therefore, \begin{equation*}\begin{aligned}
\Psi^{k+1}=\ &\Xi_x \cdot \left[\left\| x^{k+1} - x^*\right\|_2^2-2\left(t-\tilde t\right)\D{x^{k+1}}\right]\\
&+ H(y^{k+1}) + \frac{1}{t}\Dphi\left(y^{k+1},y^*\right) + {\Xi_u}\left\|{u}^{k+1}-y^*\right\|_{(I-\alpha {\hat s}\matA \matA^{\top})}^2\\ 
\le\ &\Xi_x (1-1/{\yconv})\left[\left\| x^k - x^*\right\|_2^2-2\left(t-\tilde t\right)\D{x^k}\right]+ (1-1/{\yconv})H(y^k)\\ 
&+  (1-1/{\yconv})\frac{1}{t}\Dphi\left(y^k,y^*\right) + {\Xi_u}(1-1/{\yconv})\left\|{u}^k-y^*\right\|_{(I-\alpha{\hat s} \matA \matA^{\top})}^2\\=\ &(1-1/{\yconv})\Psi^k,
\end{aligned}\end{equation*}
which concludes the proof.
\end{proof}

\section{Proof of Theorem~\ref{thm:x-accelerate-stoc-short}}
\label{appensec:proof x stoc alg}
We begin by stating the full version of \Cref{thm:x-accelerate-stoc-short}.

\begin{theorem}\label{thm:x-accelerate-stoc}
Consider applying Algorithm~\ref{alg:x-accelerate-stoc} to solve \eqref{eq:upper bound prob sto}. Let \(\xgt, \beta, t,{\hat s}>0\) satisfy
\begin{equation}\label{eq:x-accelerate-stoc para-0}
t = \frac{1 - 4\xgt-2(1+\frac{1}{\beta}){\hat s} \bar s_{\max}^2}{\bar{L} + 4\bar{L}\xgt}>0,\quad 1 - {\hat s} \bar s_{\max}^2 - 2(1+\beta){\hat s}^2 \bar s_{\max}^4\ge 0,
\end{equation}
and the remaining parameters be
\begin{equation}\label{eq:x-accelerate-stoc para-1}
\begin{aligned}
s &= \frac{\hat s}{t},
& \quad \Xi_f &= 1, \\
\xconv &\ge N\max\left(\frac{1}{s_{\min}^2 {\hat s}} \cdot \frac{1}{2\xgt},\ \sqrt{\frac{1}{\mu t}}+\frac{\bar L}{\mu}\cdot 4\xgt \right),
& \quad \xe &= \frac{1 + 4\bar L\xgt t}{N/\xconv + 4\bar L\xgt t}, \\
\xtau &= \frac{\xe - 1}{1 - N/\xconv},
& \quad \xga &= \frac{\xe - 1}{\xtau + 1}, \\
\Xi_y &= \frac{1}{2s} = \frac{t}{2\hat s},
& \quad \Xi_v &= \frac{1 + 4\bar L\xgt t}{2\xe^2 t}, \\
\xh &= 2\Xi_v \xe t = \frac{\Xi_v \xe t}{\Xi_y s}.
\end{aligned}
\end{equation}
Then, let \begin{equation*}
\Psi^k = {\Xi_y}\left\|y^k-y^*\right\|^2_{\left(I-(1-2\xgt)\frac{\hat s}{N} \matA \matA^{\top}\right)}+\D{x^k}+{\Xi_v}\left\|{v}^k-x^*\right\|_2^2,
\end{equation*} 
we have \begin{equation*}\bbE\Psi^{k+1}\le (1-1/\xconv)\bbE\Psi^k.\end{equation*}
\end{theorem}
(Notice that with our choices of \(\xgt\) and \(\hat s\), the matrix
\(I-(1-2\xgt)\frac{\hat s}{N}\matA\matA^{\top}\) is positive definite.)

\begin{corollary}\label{cor:x-accelerate-stoc}
Consider applying Algorithm~\ref{alg:x-accelerate-stoc} to solve \eqref{eq:upper bound prob sto}. Let
\begin{equation}\label{eq:x-accelerate-stoc cor para}
\begin{aligned}
\beta &= 7, 
& \quad {\hat s} &= \frac{7}{32} \frac{1}{\bar s_{\max}^2}, \\
\xgt &= \min\left(\frac{1}{10},\ \sqrt{\frac{4}{7}}\frac{\bar s_{\max}}{s_{\min}}\sqrt{\frac{\mu}{\bar L}}\right),
&\quad t &=\frac{1-8\xgt}{2+8\xgt}\frac{1}{\bar L},
\\\quad \xconv &= N\max\left(\frac{\bar s_{\max}^2}{s_{\min}^2 } \cdot \frac{16}{7\xgt},\ \sqrt{\frac{14\bar L}{\mu}}+\frac{\bar L}{\mu}\cdot 4\xgt \right).
\end{aligned}
\end{equation}
Define \(s,\xe, \xtau,\xga,\Xi_f,\Xi_y,\Xi_v,\xh\) as in \eqref{eq:x-accelerate-stoc para-1}, then
\begin{equation*}
\bbE\Psi^{k+1}\le (1-1/\xconv)\bbE\Psi^k.
\end{equation*}
\end{corollary}

\begin{remark}\label{remark:stoc_x_oraclecomplexity}
Apart from the one-time cost of computing and storing $\matA\hat x^{0}$ at initialization, each iteration of Algorithm~\ref{alg:x-accelerate-stoc} uses $O(1)$ oracle calls and block matrix multiplications (block matrix--vector products). In particular, it requires at most four block matrix multiplications (e.g., forming $\matA_i^\top y^k$ and evaluating $\matA_i(\matA_i^\top y^k+\gradfp{i}{z^k})$) and two oracle accesses per iteration.

Moreover, $\matA\hat x^{k+1}$ can be updated incrementally since only the $j$-th block of $\hat x$ changes:
\begin{equation*}
\matA\hat x^{k+1} = \matA\hat x^{k} + \matA_j\left(\hat x^{k+1}_{j}-\hat x^{k}_{j}\right).    
\end{equation*}

Similarly, if $\gradf{z^k}$ is maintained in memory, then only one oracle access is needed per iteration, since only $\gradfp{j}{z^{k+1}}$ must be updated when $j$ is sampled.
\end{remark}

We divide the proof into several parts. Specifically, we analyze the individual components of \(\Psi^{k+1}\) in Propositions~\ref{prop:x-accelerate-stoc-1}, \ref{prop:x-accelerate-stoc-2}, and \ref{prop:x-accelerate-stoc-3}, and then combine them to complete the full argument.

\begin{proposition}\label{prop:x-accelerate-stoc-1}
Consider applying Algorithm~\ref{alg:x-accelerate-stoc} to solve \eqref{eq:upper bound prob sto}. Using the parameters in \Cref{thm:x-accelerate-stoc}, we have 
\begin{equation}\label{eq:x-acce-stoc_df}
\begin{aligned}
\Dp{j}{x^{k+1}} 
&\le \Dp{j}{z^k}+2\xgt t\left\|\gradfp{j}{z^k}-\gradfp{j}{x^*}\right\|_2^2 \\
&\quad- t \left(1 - \frac{\bar L}{2} t - 2\xgt \right) 
  \left\| \gradfp{j}{z^k} + \matA_j^{\top} y^{k+1} \right\|_2^2 \\
&\quad - \xgt t \left\| \matA_j^{\top} \left(y^{k+1} - y^*\right) \right\|_2^2 
+ t \left\langle \matA_j^{\top} \left(y^{k+1} - y^*\right),\, 
  \gradfp{j}{z^k} + \matA_j^{\top} y^{k+1} \right\rangle.
\end{aligned}
\end{equation}
\end{proposition}
\begin{proof}
Similar to the proof of \Cref{prop:x-accelerate-prox-1}, the conclusion comes from the \(\bar L\)-smoothness of \(f_j\).
\end{proof}

\begin{proposition}\label{prop:x-accelerate-stoc-2}
Consider applying Algorithm~\ref{alg:x-accelerate-stoc} to solve \eqref{eq:upper bound prob sto}. Using the parameters in \Cref{thm:x-accelerate-stoc}, we have 
\begin{equation}\label{eq:x-acce-stoc_v_3}
\begin{aligned}
\bbE\left\|v^{k+1} - x^*\right\|_2^2
&\le \left(1 - \frac{1}{\xconv}\right) \bbE\left\|v^k - x^*\right\|_2^2
- 2\xe^2 \frac{t}{N}\, \bbE\D{z^k}
+ 2\xe(\xe - 1)\frac{t}{N}\, \bbE\D{x^k} \\
&\quad - 2\xe \frac{t}{N} \bbE\left\langle \matA^{\top} \left(y^{k+1} - y^*\right),\,
\hat{x}^k - x^* \right\rangle + \xe^2 \frac{t^2}{N} \bbE\left\| \nabla f(z^k) + \matA^{\top} y^{k+1} \right\|_2^2.
\end{aligned}
\end{equation}
\end{proposition}
\begin{proof}
Similar to \eqref{eq:x-acce-prox-2-init}, we have the following equality for the term \(\left\|v_j^{k+1}-x_j^*\right\|_2^2\)
\begin{equation}\label{eq:x-acce-stoc-2-init}
\begin{aligned}
&\left\|v_j^{k+1} - x_j^*\right\|_2^2 
=\ \left\| \xe z_j^k - (\xe - 1) x_j^k - x_j^* \right\|_2^2 
- 2 \xe t \left\langle 
  \gradfp{j}{z^k} - \gradfp{j}{x^*},\, 
  \xe z_j^k - (\xe - 1) x_j^k - x_j^* 
\right\rangle \\
&\quad - 2 \xe t \left\langle 
  \matA^{\top}_j \left( y^{k+1} - y^* \right),\, 
  \hat{x}_j^k - x_j^* 
\right\rangle 
+ \xe^2 t^2 \left\| \gradfp{j}{z^k} + \matA_j^{\top} y^{k+1} \right\|_2^2.
\end{aligned}
\end{equation}
Notice that \(j\) is independent of \(y^{k+1}\), \(z^k\), and \(x^k\), we could take the expectation on \(j\) and get:
\begin{equation*}
\begin{aligned}
\bbE\left\|v^{k+1} - x^*\right\|_2^2 
=\ &\frac{1}{N}\bbE\left\| \xe z^k - (\xe - 1) x^k - x^* \right\|_2^2 
- 2 \xe \frac{t}{N} \bbE\left\langle 
  \matA^{\top} \left( y^{k+1} - y^* \right),\, 
  \hat{x}^k - x^* 
\right\rangle\\
&\quad - 2 \xe \frac{t}{N} \bbE\left\langle 
  \gradf{z^k} - \gradf{x^*},\, 
  \xe z^k - (\xe - 1) x^k - x^* 
\right\rangle \\
&\quad + \xe^2 \frac{t^2}{N} \bbE\left\| \gradf{z^k}
+ \matA^{\top} y^{k+1} \right\|_2^2
 +\left(1-\frac{1}{N}\right)\bbE\left\|v^{k} - x^*\right\|_2^2.
\end{aligned}
\end{equation*}
Again, by \eqref{eq:x-acce-prox-2-mu-strong-convexity}, we have 
\begin{equation*}
\begin{aligned}
\bbE\left\|v^{k+1} - x^*\right\|_2^2 
\le \ &\frac{1}{N}\bbE\left\| \xe z^k - (\xe - 1) x^k - x^* \right\|_2^2 
- 2 \xe \frac{t}{N} \bbE\left\langle 
  \matA^{\top} \left( y^{k+1} - y^* \right),\, 
  \hat{x}^k - x^* 
\right\rangle\\
&\quad - 2\xe^2 \frac{t}{N}\, \bbE\D{z^k}
+ 2\xe(\xe - 1)\frac{t}{N}\, \bbE\D{x^k} -\mu\xe\frac{t}{N}\bbE\left\|z^k-x^*\right\|_2^2 \\
&\quad + \xe^2 \frac{t^2}{N} \bbE\left\| \gradf{z^k}
+ \matA^{\top} y^{k+1} \right\|_2^2
 +\left(1-\frac{1}{N}\right)\bbE\left\|v^{k} - x^*\right\|_2^2.    
\end{aligned}
\end{equation*}
Notice that \(\xconv \xe \mu t \ge N\) by \eqref{eq:x-accelerate-stoc para-1}. (\(\xconv \xe \mu t =\xconv \mu t \frac{1+4\bar L\xgt t}{4\bar L \xgt t + N/\xconv}\ge\xconv \mu t \frac{1}{4\bar L \xgt t + N/\xconv}\). The positive solution to \(\xconv \mu t \frac{1}{4\bar L \xgt t + N/\xconv}=N\) is \(\xconv^{+}=\frac{4N\bar L \xgt t+\sqrt{(4N\bar L\xgt t)^2+4N^2\mu t}}{2\mu t}\). By \eqref{eq:x-accelerate-stoc para-1}, \(\xconv\ge N\sqrt{\frac{1}{\mu t}}+4N\frac{\bar L}{\mu}\xgt\ge \xconv^{+}\).) Since \(\xe - 1 = (1-N/\xconv)\xtau\), by Jensen's inequality
\begin{equation*}\left\|\xe z^{k}-(\xe-1)x^k-x^*\right\|_2^2\le \left(1-\frac{N}{\xconv}\right)\left\|(1+\xtau)z^{k}-\xtau x^{k}-x^*\right\|_2^2 +\frac{N}{\xconv} \left\|z^k-x^*\right\|_2^2.\end{equation*}
Thus, combining all these, we conclude the proof for \Cref{prop:x-accelerate-stoc-2}. 
\end{proof}

\begin{proposition}\label{prop:x-accelerate-stoc-3}
Consider applying Algorithm~\ref{alg:x-accelerate-stoc} to solve \eqref{eq:upper bound prob sto}. Using the parameters in \Cref{thm:x-accelerate-stoc}, we have 
\begin{equation}\label{eq:x-acce-stoc_y}
\begin{aligned}
\bbE\left\|y^{k+1}-y^*\right\|_{\left(I-\frac{\hat s}{N}\matA \matA^{\top}\right)}^2&\le \bbE\left\|y^{k}-y^*\right\|_{\left(I-\frac{\hat s}{N} \matA \matA^{\top}\right)}^2+2\xh \frac{s}N\bbE\left\langle \matA^{\top}\left(y^{k+1}-y^*\right),{\hat x}^k - x^*\right\rangle   \\ &\quad -2\frac{{\hat s}} N \bbE\left\langle \matA^{\top}\left(y^{k+1}-y^*\right), \matA^{\top}y^{k+1}+\gradf{z^k}\right\rangle\\
&\quad - \bbE\left\|y^{k+1}-y^k\right\|_{\left(I-\frac{\hat s}{N}\matA \matA^{\top}\right)}^2 + 2\left(1+\beta\right)\frac{{\hat s}^2}{N}\sum_{l=1}^N\bbE\left\|\matA_l\matA_l^{\top} \left(y^{k+1}-y^k\right)\right\|_2^2\\
&\quad +2\left(1+\frac{1}{\beta}\right) \frac{{\hat s}^2}{N}\sum_{l=1}^N \bbE\left\|\matA_l\left( \matA_l^{\top} y^{k+1} +\gradfp{l}{z^k}\right)\right\|_2^2
\end{aligned}
\end{equation}
\end{proposition}
\begin{proof}
We have the following equality for \(\bbE\left\|y^{k+1}-y^*\right\|_2^2\):
\begin{equation*}
\begin{aligned}
\bbE\left\|y^{k+1}-y^*\right\|_2^2 &= \bbE\left\|y^k - y^*\right\|_2^2 - 2{\hat s} \bbE\left\|\matA_i^{\top}\left(y^k-y^*\right)\right\|_2^2 +\bbE\left\|y^{k+1}-y^k\right\|_2^2\\
&\quad + 2 \bbE\left\langle\xh \frac{s}{N}\matA \left(\hat x^k -x^*\right)-{\hat s} \matA_i\left(\gradfp{i}{z^k}-\gradfp{i}{x^*}\right), y^k-y^*\right\rangle\\
&=\bbE\left\|y^k - y^*\right\|_2^2 - 2\frac{{\hat s}}{N} \bbE\left\|\matA^{\top}\left(y^k-y^*\right)\right\|_2^2 +\bbE\left\|y^{k+1}-y^k\right\|_2^2\\
&\quad + 2 \bbE\left\langle\xh \frac{s}{N}\matA \left(\hat x^k -x^*\right)-\frac{{\hat s}}N \matA\left(\gradf{z^k}-\gradf{x^*}\right), y^k-y^*\right\rangle\\
&= \bbE\left\|y^k - y^*\right\|_2^2 - 2\frac{{\hat s}}{N} \bbE\left\|\matA^{\top}\left(y^k-y^*\right)\right\|_2^2 +\bbE\left\|y^{k+1}-y^k\right\|_2^2\\
&\quad + 2 \bbE\left\langle\xh \frac{s}{N}\matA \left(\hat x^k -x^*\right)-\frac{{\hat s}}N \matA\left(\gradf{z^k}+\matA^{\top}y^{k+1}\right), y^{k+1}-y^*\right\rangle\\
&\quad - 2  \bbE\left\langle\xh \frac{s}{N}\matA \left(\hat x^k -x^*\right)-\frac{{\hat s}}N \matA\left(\gradf{z^k}-\gradf{x^*}\right), y^{k+1}-y^k\right\rangle\\
&\quad + 2\frac{{\hat s}}{N} \bbE\left\|\matA^{\top}\left(y^{k+1}-y^*\right)\right\|_2^2.
\end{aligned}   
\end{equation*}
By applying \(\bbE\left\|\matA^{\top}\left(y^{k+1}-y^*\right)\right\|_2^2 = \bbE\left\|\matA^{\top}\left(y^{k}-y^*\right)\right\|_2^2 + 2\bbE\left\langle\matA^{\top}\left(y^k-y^*\right), \matA^{\top}\left(y^{k+1}-y^k\right)\right\rangle + \bbE\left\|\matA^{\top}\left(y^{k+1}-y^k\right)\right\|_2^2\) on the right hand side, we get 
\begin{equation*}
\begin{aligned}
&\bbE\left\|y^{k+1}-y^*\right\|_{\left(I-\frac{\hat s}{N}\matA \matA^{\top}\right)}^2\\ =&\ \bbE\left\|y^k - y^*\right\|_{\left(I-\frac{\hat s}{N}\matA \matA^{\top}\right)}^2 + \bbE\left\|y^{k+1}-y^k\right\|_{\left(I+\frac{\hat s}{N}\matA \matA^{\top}\right)}^2\\
&+2 \bbE\left\langle\xh \frac{s}{N}\matA \left(\hat x^k -x^*\right)-\frac{{\hat s}}N \matA\left(\gradf{z^k}+\matA^{\top}y^{k+1}\right), y^{k+1}-y^*\right\rangle\\
&- 2  \bbE\left\langle\xh \frac{s}{N}\matA \left(\hat x^k -x^*\right)-\frac{{\hat s}}N \matA\left(\gradf{z^k}+\matA^{\top}y^k\right), y^{k+1}-y^k\right\rangle.
\end{aligned}   
\end{equation*}
Here, by taking the expectation on \(i\),
\begin{equation*}
\begin{aligned}
&\bbE\left\langle\xh \frac{s}{N}\matA \left(\hat x^k -x^*\right)-\frac{{\hat s}}N \matA\left(\gradf{z^k}+\matA^{\top}y^k\right), y^{k+1}-y^k\right\rangle \\
=\ &\xh^2\frac{s^2}{N^2}\bbE\left\|\matA\left(\hat x^k -x^*\right)\right\|_2^2 + \frac{{\hat s}^2}{N^2}\bbE\left\|\matA\left(\gradf{z^k}+\matA^{\top}y^k\right)\right\|_2^2\\
&-2\xh\frac{s{\hat s}}{N^2}\bbE\left\langle\matA \left(\hat x^k -x^*\right), \matA\left(\gradf{z^k}+\matA^{\top}y^k\right)\right\rangle.
\end{aligned}
\end{equation*}
And, 
\begin{equation*}
\begin{aligned}
&\bbE\left\|y^{k+1}-y^k\right\|_2^2 \\
=\ &\xh^2\frac{s^2}{N^2}\bbE\left\|\matA\left(\hat x^k -x^*\right)\right\|_2^2 + \frac{{\hat s}^2}{N}\bbE\sum_{l=1}^N\left\|\matA_l\left(\gradfp{l}{z^k}+\matA_l^{\top}y^k\right)\right\|_2^2\\
&-2\xh\frac{s{\hat s}}{N^2}\bbE\left\langle\matA \left(\hat x^k -x^*\right), \matA\left(\gradf{z^k}+\matA^{\top}y^k\right)\right\rangle\\
\end{aligned}
\end{equation*}
Hence, 
\begin{equation*}
\begin{aligned}
&\bbE\left\|y^{k+1}-y^*\right\|_{\left(I-\frac{\hat s}{N}\matA \matA^{\top}\right)}^2\\ =\ &\bbE\left\|y^k - y^*\right\|_{\left(I-\frac{\hat s}{N}\matA \matA^{\top}\right)}^2 - \bbE\left\|y^{k+1}-y^k\right\|_{\left(I-\frac{\hat s}{N}\matA \matA^{\top}\right)}^2\\
&+2 \bbE\left\langle\xh \frac{s}{N}\matA \left(\hat x^k -x^*\right)-\frac{{\hat s}}N \matA\left(\gradf{z^k}+\matA^{\top}y^{k+1}\right), y^{k+1}-y^*\right\rangle\\
&+2\frac{{\hat s}^2}{N}\bbE\sum_{l=1}^N\left\|\matA_l\left(\gradfp{l}{z^k}+\matA_l^{\top}y^k\right)\right\|_2^2 - 2\frac{{\hat s}^2}{N^2}\bbE\left\|\matA\left(\gradf{z^k}+\matA^{\top}y^k\right)\right\|_2^2\\
\le\ & \bbE\left\|y^k - y^*\right\|_{\left(I-\frac{\hat s}{N}\matA \matA^{\top}\right)}^2 - \bbE\left\|y^{k+1}-y^k\right\|_{\left(I-\frac{\hat s}{N}\matA \matA^{\top}\right)}^2\\
&+2 \bbE\left\langle\xh \frac{s}{N}\matA \left(\hat x^k -x^*\right)-\frac{{\hat s}}N \matA\left(\gradf{z^k}+\matA^{\top}y^{k+1}\right), y^{k+1}-y^*\right\rangle\\
&+2\left(1+\frac{1}{\beta}\right) \frac{{\hat s}^2}{N}\sum_{l=1}^N \bbE\left\|\matA_l\left( \matA_l^{\top} y^{k+1} +\gradfp{l}{z^k}\right)\right\|_2^2\\
&+2\left(1+\beta\right)\frac{{\hat s}^2}{N}\sum_{l=1}^N\bbE\|\matA_l\matA_l^{\top} \left(y^{k+1}-y^k\right)\|_2^2,
\end{aligned}   
\end{equation*}
where the last inequality comes from omitting the last term and Young's inequality on \(\left\langle \matA_l\matA_l^{\top} \left(y^{k+1}-y^k\right), \matA_l\left(\matA_l^{\top} y^{k+1} +\gradfp{l}{z^k}\right)\right\rangle\).
\end{proof}

\begin{proof}[Proof of \Cref{thm:x-accelerate-stoc}]
Since only block \(j\) is updated, by taking the expectation of \eqref{eq:x-acce-stoc_df} over \(j\sim U(\{1,\dots,N\})\), we have
\begin{equation}\label{eq:x-acce-stoc_df-exp}
\begin{aligned}
\bbE\D{x^{k+1}} 
&\le \left(1-\frac{1}{N}\right)\bbE\D{x^k} + \frac{1}{N}(1 + 4\bar L \xgt t)\bbE \D{z^k}\\ 
& \quad - \left(1 - \frac{\bar L}{2} t - 2\xgt \right) 
  \frac{t}{N}\bbE\left\| \gradf{z^k} + \matA^{\top} y^{k+1} \right\|_2^2 \\
&\quad - \xgt \frac{t}{N} \bbE\left\| \matA^{\top} \left(y^{k+1} - y^*\right) \right\|_2^2 
+ \frac{t}N \bbE\left\langle \matA^{\top} \left(y^{k+1} - y^*\right),\, 
  \gradf{z^k} + \matA^{\top} y^{k+1} \right\rangle.
\end{aligned}
\end{equation}
Combine \eqref{eq:x-acce-stoc_df-exp}, \eqref{eq:x-acce-stoc_v_3}, and \eqref{eq:x-acce-stoc_y}, and by the settings of the weights \({\Xi_f}\), \({\Xi_y}\), \({\Xi_v}\), and other parameters in \eqref{eq:x-accelerate-stoc para-0} and \eqref{eq:x-accelerate-stoc para-1}, we see that on the right hand side: \begin{itemize}
\item{the term \(\bbE\D{x^k}\) has a coefficient of \({\Xi_f} \left(1-\frac{1}{N}\right)+{\Xi_v}2\xe(\xe-1)\frac{t}{N} ={\Xi_f}\left(1-\frac{1}{\xconv}\right)\);}
\item{the term \(\bbE\D{z^k}\) has a coefficient of \({\Xi_f} \frac{1+4\bar L \xgt t}{N}-{\Xi_v}2\xe^2\frac{t}{N} =0\);}
\item{the term \(\bbE\left\langle \matA^{\top} \left(y^{k+1} - y^*\right),\, 
  \gradf{z^k} + \matA^{\top} y^{k+1} \right\rangle\) has a coefficient of \(\Xi_f \cdot\frac{t}N-\Xi_y\cdot 2\frac{{\hat s}}{N}=0\);}
\item{the term \(\bbE\left\langle \matA^{\top} \left(y^{k+1} - y^*\right),\,
\hat{x}^k - x^* \right\rangle\) has a coefficient of \(-2\Xi_v \cdot \xe \frac{t}{N}+\Xi_y\cdot 2\xh \frac{s}{N}=0;\)}
\item{
on the right hand side, the sum of the squared norm terms related to  \(\gradf{z^k} + \matA^{\top} y^{k+1}\) is 
\begin{equation*}
\begin{aligned}
    &\left[-{\Xi_f}\left(1 - \frac{\bar L}{2} t - 2\xgt \right) 
  \frac{t}{N}+{\Xi_v}\xe^2\frac{t^2}{N}\right]\bbE\left\|\gradf{z^k} + \matA^{\top} y^{k+1}\right\|_2^2\\
  + &\quad 2{\Xi_y}\left(1+\frac{1}{\beta}\right) \frac{{\hat s}^2}{N}\sum_{l=1}^N \bbE\left\|\matA_l\left( \matA_l^{\top} y^{k+1} +\gradfp{l}{z^k}\right)\right\|_2^2\\
\le\  &{\Xi_f}\left[-\left(1 - \frac{\bar L}{2} t - 2\xgt \right) 
  +\frac{1+4\bar L \xgt t}{2}+\left(1+\frac{1}{\beta}\right)\bar s_{\max}^2{\hat s}\right]\frac{t}{N}\bbE\left\|\gradf{z^k} + \matA^{\top} y^{k+1}\right\|_2^2\\
=\ & \frac{\Xi_f}{2}\left[{\bar L (1+4\xgt) t -(1-4\xgt)+2\left(1+\frac{1}{\beta}\right)\bar s_{\max}^2{\hat s}}\right]\bbE\left\|\gradf{z^k} + \matA^{\top} y^{k+1}\right\|_2^2 \le 0;
\end{aligned}
\end{equation*}}
\item{
on the right hand side, the sum of the squared norm terms related to  \( y^{k+1}-y^k\) is
\begin{equation*}
\begin{aligned}
- \Xi_y\bbE\left\|y^{k+1}-y^k\right\|_{\left(I-\frac{\hat s}{N}\matA \matA^{\top}\right)}^2+\Xi_y 2\left(1+\beta\right)\frac{{\hat s}^2}{N}\sum_{l=1}^N\bbE\|\matA_l\matA_l^{\top} \left(y^{k+1}-y^k\right)\|_2^2\le0
\end{aligned}
\end{equation*}
as \(1 - {\hat s} \bar s_{\max}^2 - 2(1+\beta){\hat s}^2 \bar s_{\max}^4\ge 0\).
}
\end{itemize}
Hence, 
\begin{equation*}
\begin{aligned}
&{\Xi_y}\bbE\left\|y^{k+1}-y^*\right\|^2_{\left(I-\frac{\hat s}{N} \matA \matA^{\top}\right)}+\bbE\D{x^{k+1}}+{\Xi_v}\bbE\left\|{v}^{k+1}-x^*\right\|_2^2\\
\le\ &{\Xi_y}\bbE\left\|y^{k}-y^*\right\|^2_{\left(I-\frac{\hat s}{N} \matA \matA^{\top}\right)}+\left(1-\frac{1}{\xconv}\right)\bbE\D{x^{k}}\\
&+\left(1-\frac{1}{\xconv}\right){\Xi_v}\bbE\left\|{v}^{k}-x^*\right\|_2^2 - {\Xi_f}\xgt \frac{t}{N}\bbE\left\|\matA^{\top}\left(y^{k+1}-y^*\right)\right\|_2^2\\
\le\ &{\Xi_y}\bbE\left\|y^{k}-y^*\right\|^2_{\left(I-(1-2\xgt)\frac{\hat s}{N} \matA \matA^{\top}\right)}+\left(1-\frac{1}{\xconv}\right)\bbE\D{x^{k}}\\
&+\left(1-\frac{1}{\xconv}\right){\Xi_v}\bbE\left\|{v}^{k}-x^*\right\|_2^2 - {\Xi_y}2\xgt \frac{{\hat s}}{N}\bbE\left\|\matA^{\top}\left(y^{k+1}-y^*\right)\right\|_2^2\\
&- {\Xi_y}2\xgt s_{\min}^2\frac{{\hat s}}{N}\bbE\left\|y^{k}-y^*\right\|^2.
\end{aligned}  
\end{equation*}
Finally, since \(I-(1-2\xgt)\frac{\hat s}{N}\matA\matA^\top \preceq I\) and \(\xconv \ge \frac{N}{2\xgt s_{\min}^2\hat s}\), we have
\begin{equation*}
\Xi_y\|y^{k}-y^*\|^2_{\left(I-(1-2\xgt)\frac{\hat s}{N}\matA\matA^\top\right)}
-2\Xi_y\xgt s_{\min}^2\frac{\hat s}{N}\|y^{k}-y^*\|_2^2
\le
\left(1-\frac{1}{\xconv}\right)\Xi_y\|y^{k}-y^*\|^2_{\left(I-(1-2\xgt)\frac{\hat s}{N}\matA\matA^\top\right)},
\end{equation*}
which together with the previous inequality yields
\(\bbE\Psi^{k+1}\le \left(1-\frac{1}{\xconv}\right)\bbE\Psi^k\).
\end{proof}

\section{Proof of Theorem~\ref{thm:y-accelerate-stoc-short}}
\label{appensec:proof y stoc alg}
We begin by stating the full version of \Cref{thm:y-accelerate-stoc-short}.
\begin{theorem}\label{thm:y-accelerate-stoc}
Consider applying Algorithm~\ref{alg:y-accelerate-stoc} to solve \eqref{eq:upper bound prob sto}. Let 
\begin{equation}\label{eq:y-accelerate-stoc para}
\begin{aligned}
\hat s & = \frac{1}{1+1/\beta}\frac{1}{\bar s_{\max}^2}, 
& \quad t & = \frac{1}{2\bar L}, \\
\alpha &\in(0,1), &\beta &\in(0,1),\\
\yeta & = \frac{\yc - 1}{1 - 1/\yconv}, & \yc&>1,\\
\tilde t & = \frac{\alpha t}{\yc}, & \quad \yconv &= N\max\left(\frac{1+1/\beta}{\yc(1-\alpha)}\cdot\frac{\bar s_{\max}^2}{s_{\min}^2},\ 2\frac{\yc}{\alpha}\cdot\frac{\bar L}{\mu}\right),\\
\quad s & = \frac{\hat s}{t} & \quad \Xi_h & = 1, \\
\quad \Xi_u & = \frac{1}{2\yc^2 \hat s},
& \Xi_x & = \Xi_u \cdot \frac{\yc s}{\tilde t} = \frac{1}{2\alpha t^2}.
\end{aligned}
\end{equation}
with \(\alpha, \beta,\yeta, \yc\) satisfying \(\yc(\yc-1)\ge\alpha(1+\yeta)(\yc-1)+ \beta \yc^2\). Then, let 
\begin{equation*}
\begin{aligned}
\Psi^k =\ &\Xi_x \cdot \left[\left\| x^k - x^*\right\|_2^2-2\left(t-\tilde t\right)\D{x^k}\right] \\
&+ H(y^k) + \Xi_u \left\| u^k - y^* \right\|^2_{ \left(I - \alpha\frac{{\hat s}}N \matA \matA^{\top} \right)},
\end{aligned}
\end{equation*}
we have 
\begin{equation*}\bbE\Psi^{k+1}\le (1-1/\yconv)\bbE\Psi^k.\end{equation*}
\end{theorem}

(Notice that with our choices of \(\alpha\), \(\hat s\), \(t\), and \(\tilde t\), the matrix
\(I-\alpha\frac{\hat s}{N}\matA\matA^{\top}\) is positive definite and \(\left\| x^k - x^*\right\|_2^2-2\left(t-\tilde t\right)\D{x^k}\ge \left(1-\bar L\left(t-\tilde t\right)\right)\left\| x^k - x^*\right\|_2^2\ge 0\).)

\begin{corollary}\label{cor:y-accelerate-stoc}
Consider applying Algorithm~\ref{alg:y-accelerate-stoc} to solve \eqref{eq:upper bound prob sto}. Let
\begin{equation}\label{eq:y-accelerate-stoc cor para}
\begin{aligned}
\alpha &= \frac12, 
& \quad \beta &= \frac13, \\
\yc &= \max\left(\frac{1}{1-\sqrt{\frac23}},\ \sqrt{2}\frac{\bar s_{\max}}{s_{\min}}\sqrt{\frac{\mu}{\bar L}}\right),
& \quad \yconv &= N\max\left(\frac{8}{\yc}\cdot\frac{\bar s_{\max}^2}{s_{\min}^2},\ 4\yc\cdot\frac{\bar L}{\mu}\right).
\end{aligned}
\end{equation}
Define \(\hat s,t,\tilde t,s,\yeta,\Xi_h,\Xi_u,\Xi_x\) as in \eqref{eq:y-accelerate-stoc para}, then
\begin{equation*}
\bbE\Psi^{k+1}\le (1-1/\yconv)\bbE\Psi^k.
\end{equation*}
\end{corollary}

We divide the proof into several parts. Specifically, we analyze the individual components of \(\Psi^{k+1}\) in Propositions~\ref{prop:y-accelerate-stoc-1}, \ref{prop:y-accelerate-stoc-2}, and \ref{prop:y-accelerate-stoc-3}, and then combine them to complete the full argument.

\begin{proposition}\label{prop:y-accelerate-stoc-1}
Consider applying Algorithm~\ref{alg:y-accelerate-stoc} to solve \eqref{eq:upper bound prob sto}. Using the parameters in \Cref{thm:y-accelerate-stoc}, we have 
\begin{equation}\label{eq:y-acce-stoc_x}
\begin{aligned}
  &\left\|x_j^{k+1} - x_j^*\right\|_2^2
    - 2 \left(t - \tilde t\right)\, \Dp{j}{x^{k+1}}\\
  &\qquad \le
    \left[
      \left\|x_j^{k} - x_j^*\right\|_2^2
      - 2 \left(t - \tilde t\right)\, \Dp{j}{x^{k}}
    \right]
    \left(1 - 2\mu {\tilde t}\left(1-\bar L t\right)\right)\\
  &\qquad\quad
    - 2 \tilde t \left\langle
        x_j^{k} - x_j^* - t\left(\gradfp{j}{x^{k}} - \gradfp{j}{x^*}\right),
        \matA_j^\top \left(u^{k+1} - y^*\right)
      \right\rangle
    + \tilde t^2 \left\| \matA_j^\top \left(u^{k+1} - y^*\right) \right\|_2^2.
\end{aligned}
\end{equation}
\end{proposition}
\begin{proof}
Similar to the proof of \Cref{prop:y-accelerate-prox-1}, the conclusion comes from the \(\bar L\)-smoothness and the \(\mu\)-strong convexity of \(f_j\).
\end{proof}

\begin{proposition}\label{prop:y-accelerate-stoc-2}
Consider applying Algorithm~\ref{alg:y-accelerate-stoc} to solve \eqref{eq:upper bound prob sto}. Using the parameters in \Cref{thm:y-accelerate-stoc}, we have 
\begin{equation}\label{eq:y-acce-stoc_H}
\begin{aligned}
H\left( \tilde y^{k+1} \right)
\le\ 
& H\left( w^k \right) 
- \left( \frac{N}{{\hat s}} - \frac{N \bar s_{\max}^2}{2} - \frac{N \bar s_{\max}^2}{2\beta}\right)
\left\| \tilde y^{k+1} - w^k \right\|_2^2 \\
\ &+ \left\langle 
\frac{s}{{\hat s}} \matA \left( x^k - x^* \right) 
- \matA \left( \gradf{{x}^k} - \gradf{x^*} \right) 
,\ \tilde y^{k+1} - w^k\right\rangle \\
\ & +\frac{\beta}{2N\bar s_{\max}^2}\left\|\matA \matA^{\top}(w^k-y^k) - N\matA_i\matA_i^{\top}(w^k-y^k)\right\|_2^2.
\end{aligned}
\end{equation}
\end{proposition}
\begin{proof}
For the term \(H\left(\tilde y^{k+1}\right)\), since \(H\) is \(N\bar s_{\max}^2\)-smooth, and \(\nabla H\left(w^k\right)=\matA\matA^{\top}\left(w^k-y^*\right),\)
\begin{equation*}
\begin{aligned}
H\left( \tilde y^{k+1} \right)
\le\ 
&H\left( w^k \right) 
+ \left\langle 
\matA \matA^{\top} \left( w^k - y^* \right),\,
\tilde y^{k+1} - w^k 
\right\rangle 
+ \frac{N \bar s_{\max}^2}{2} 
\left\| \tilde y^{k+1} - w^k \right\|_2^2 \\
=\ &H\left( w^k \right) 
- \left( \frac{N}{{\hat s}} - \frac{N \bar s_{\max}^2}{2} \right)
\left\| \tilde y^{k+1} - w^k \right\|_2^2 \\
\ &+ \left\langle 
\frac{s}{{\hat s}} \matA \left( x^k - x^* \right) 
- \matA \left( \gradf{{x}^k} - \gradf{x^*} \right) 
,\ \tilde y^{k+1} - w^k\right\rangle. \\
\ & +\left\langle\matA \matA^{\top}(w^k-y^k) - N\matA_i\matA_i^{\top}(w^k-y^k),\ \tilde y^{k+1}-w^k\right\rangle\\
\le\ & H\left( w^k \right) 
- \left( \frac{N}{{\hat s}} - \frac{N \bar s_{\max}^2}{2} - \frac{N \bar s_{\max}^2}{2\beta}\right)
\left\| \tilde y^{k+1} - w^k \right\|_2^2 \\
\ &+ \left\langle 
\frac{s}{{\hat s}} \matA \left( x^k - x^* \right) 
- \matA \left( \gradf{{x}^k} - \gradf{x^*} \right) 
,\ \tilde y^{k+1} - w^k\right\rangle. \\
\ & +\frac{\beta}{2N\bar s_{\max}^2}\left\|\matA \matA^{\top}(w^k-y^k) - N\matA_i\matA_i^{\top}(w^k-y^k)\right\|_2^2,
\end{aligned}
\end{equation*}
where the second inequality comes from the inequality: \begin{equation*}
\begin{aligned}
&\left\langle\matA \matA^{\top}(w^k-y^k) - N\matA_i\matA_i^{\top}(w^k-y^k),\ \tilde y^{k+1}-w^k\right\rangle \\
\le\  &\frac{N \bar s_{\max}^2}{2\beta}\left\| \tilde y^{k+1} - w^k \right\|_2^2+\frac{\beta}{2N\bar s_{\max}^2}\left\|\matA \matA^{\top}(w^k-y^k) - N\matA_i\matA_i^{\top}(w^k-y^k)\right\|_2^2.    
\end{aligned}
\end{equation*}
\end{proof}

\begin{proposition}\label{prop:y-accelerate-stoc-3}
Consider applying Algorithm~\ref{alg:y-accelerate-stoc} to solve \eqref{eq:upper bound prob sto}. Using the parameters in \Cref{thm:y-accelerate-stoc}, we have 
\begin{equation}\label{eq:y-acce-stoc_u_3}
\begin{aligned}
&\bbE\left\| u^{k+1} - y^* \right\|_2^2 \\
\le\ 
&\left(1-\frac{1}{\yconv}\right)\bbE\left\| u^k - y^* \right\|_{\left(I-\alpha\frac{{\hat s}}{N}\matA\matA^{\top}\right)}^2 
+ \yc^2 \bbE\left\| \tilde y^{k+1} - w^k \right\|_2^2 \\
&+2\frac{\yc}{N}\bbE\left\langle \yc {w}^{k} - (\yc-1) y^{k}-y^*,s\matA \left(x^{k}-x^*\right)-{\hat s} \matA \left( \gradf{x^{k}}-\gradf{x^*}\right)\right\rangle\\
&+ \left[\yc(\yc-1)-\alpha(1+\yeta)(\yc-1)\right]\frac{{\hat s}}{N}\bbE\left\|\matA^{\top}\left(w^k-y^k\right)\right\|_2^2\\
&+ 2\left( \yc - \alpha \right) (\yc - 1) \frac{{\hat s}}N \bbE H\left(y^k\right) - \left[\left( \yc^2 + (1-\alpha) \yc \right) \frac{{\hat s}}{N} - \frac{1}{\yconv s_{\min}^2} \right]
\bbE\left\| \matA^{\top} \left( w^k - y^* \right) \right\|_2^2.
\end{aligned}
\end{equation}
\end{proposition}
\begin{proof}
Notice that \(u^k = (1+\yeta) w^k -\yeta y^k\) and \(u^{k+1}=(1+\yeta) w^{k+1} - \yeta y^{k+1}\). For the term \(\left\|u^{k+1}-y^*\right\|_2^2\), we have the following equality,
\begin{equation}\label{eq:y-acce-stoc_u_start}
\begin{aligned}
&\bbE\left\|{u}^{k+1}-y^*\right\|_2^2= \bbE\left\|\yc \tilde {y}^{k+1} - (\yc-1) y^{k}-y^*\right\|_2^2\\
=\ & \bbE\left\|\yc {w}^{k} - (\yc-1) y^{k}-y^*\right\|_2^2 + 2\yc\bbE\left\langle \yc {w}^{k} - (\yc-1) y^{k}-y^*,\tilde y^{k+1}-w^k\right\rangle  + \yc^2\bbE\left\|\tilde y^{k+1}-w^k\right\|_2^2.\\
=\ & \bbE\left\|\yc {w}^{k} - (\yc-1) y^{k}-y^*\right\|_2^2 +\yc^2\bbE\left\|\tilde y^{k+1}-w^k\right\|_2^2\\
& +2\frac{\yc}{N}\bbE\left\langle \yc {w}^{k} - (\yc-1) y^{k}-y^*,s\matA \left(x^{k}-x^*\right)-{\hat s} \matA \left( \gradf{x^{k}}-\gradf{x^*}\right)\right\rangle\\
&- 2\frac{\yc}{N}\bbE\left\langle \yc {w}^{k} - (\yc-1) y^{k}-y^*,{\hat s} \matA\matA^{\top} \left( {y}^k -y^*\right) +N{{\hat s}} \matA_i\matA_i^{\top} (w^k-y^k)\right\rangle\\
=\ & \bbE \left\|\yc {w}^{k} - (\yc-1) y^{k}-y^*\right\|_2^2 +\yc^2\bbE\left\|\tilde y^{k+1}-w^k\right\|_2^2\\
& +2\frac{\yc}{N}\bbE\left\langle \yc {w}^{k} - (\yc-1) y^{k}-y^*,s\matA \left(x^{k}-x^*\right)-{\hat s} \matA \left( \gradf{x^{k}}-\gradf{x^*}\right)\right\rangle\\
&- 2\frac{\yc}{N}\bbE\left\langle \yc {w}^{k} - (\yc-1) y^{k}-y^*,{\hat s} \matA\matA^{\top} \left( {w}^k -y^*\right)\right\rangle,
\end{aligned}   
\end{equation}
where the last equality comes from the independence of \(y^k,w^k\) with \(i\). Also, notice that \eqref{eq:y-acce-prox_u_start_2} and \eqref{eq:y-acce-prox_u_start_3} still hold. Combining \eqref{eq:y-acce-stoc_u_start}, \eqref{eq:y-acce-prox_u_start_2}, and \eqref{eq:y-acce-prox_u_start_3}, similar as the derivation of \eqref{eq:y-acce-prox_long_ineq}, we have the following inequality:
\begin{equation}\label{eq:y-acce-stoc_long_ineq}
\begin{aligned}
&\bbE\left\| u^{k+1} - y^* \right\|_2^2 
+ \left( 1 - \frac{1}{\yconv} \right) \alpha \frac{{\hat s}}{N} 
\bbE\left \| \matA^{\top} \left( u^k - y^* \right) \right\|_2^2 \\
\le\ 
&\bbE\left\| \yc w^k - (\yc - 1) y^k - y^* \right\|_2^2 
+ \yc^2 \bbE\left\| \tilde y^{k+1} - w^k \right\|_2^2 \\
&+2\frac{\yc}{N}\bbE\left\langle \yc {w}^{k} - (\yc-1) y^{k}-y^*,s\matA \left(x^{k}-x^*\right)-{\hat s} \matA \left( \gradf{x^{k}}-\gradf{x^*}\right)\right\rangle\\
&+ \left[\yc(\yc-1)-\alpha(1+\yeta)(\yc-1)\right]\frac{{\hat s}}{N}\bbE\left\|\matA^{\top}\left(w^k-y^k\right)\right\|_2^2\\
&+ \left( \yc - \alpha \right) (\yc - 1) \frac{{\hat s}}N 
\bbE\left\| \matA^{\top} \left( y^k - y^* \right) \right\|_2^2 - \left( \yc^2 + (1-\alpha) \yc \right) \frac{{\hat s}}{N} 
\bbE\left\| \matA^{\top} \left( w^k - y^* \right) \right\|_2^2.
\end{aligned}
\end{equation}
Since \(\yc - 1 = (1-1/\yconv)\yeta\), by Jensen's inequality
\begin{equation*}\begin{aligned}
\left\|\yc w^{k}-(\yc-1)y^k-y^*\right\|_2^2&\le \left(1-\frac1{\yconv}\right)\left\|(1+\yeta)w^{k}-\yeta y^{k}-y^*\right\|_2^2 +\frac1{\yconv} \left\|{w}^k-y^*\right\|_2^2\\
&\le \left(1-\frac1{\yconv}\right)\left\|{u}^k-y^*\right\|_2^2 +\frac1{\yconv s_{\min}^2} \left\|\matA^{\top}\left({w}^k-y^*\right)\right\|_2^2.
\end{aligned}\end{equation*}
Thus, combining all these, we conclude the proof for Proposition~\ref{prop:y-accelerate-stoc-3}.
\end{proof}

\begin{proof}[Proof of \Cref{thm:y-accelerate-stoc}]
By \Cref{prop:y-accelerate-stoc-1}, when taking the expectation on \(j\), we have
\begin{equation}\label{eq:y-acce-stoc_x_exp}
\begin{aligned}
  &\bbE\left[\left\|x^{k+1} - x^*\right\|_2^2
    - 2 \left(t - \tilde t\right)\, \D{x^{k+1}}\right]\\
  \le\ & 
    \bbE\left[
      \left\|x^{k} - x^*\right\|_2^2
      - 2 \left(t - \tilde t\right)\, \D{x^{k}}
    \right]
    \left(1 - 2\mu {\tilde t}(1-\bar Lt)/N\right)\\
  &
    - 2 \frac{\tilde t}{N}\bbE \left\langle
        x^{k} - x^* - t\left(\gradf{x^{k}} - \gradf{x^*}\right),
        \matA^\top \left(u^{k+1} - y^*\right)
      \right\rangle
    + \frac{\tilde t^2}{N} \bbE\left\| \matA^\top \left(u^{k+1} - y^*\right) \right\|_2^2,
\end{aligned}
\end{equation}
as \(u^{k+1}\) and \(x^k\) are independent with \(j\). Similarly, using \eqref{eq:y-acce-stoc_H}, we have
\begin{equation}\label{eq:y-acce-stoc_H_exp}
\begin{aligned}
&\bbE H\left(y^{k+1}\right) \\
=\ & \left(1-\frac1N\right)\bbE H\left(y^{k}\right) + \frac{1}{N} \bbE H\left(\tilde y^{k+1}\right)\\
\le\ &  \frac{1}{N} \bbE H\left(w^k\right)
- \left( \frac{1}{{\hat s}} - \frac{\bar s_{\max}^2}{2} - \frac{\bar s_{\max}^2}{2\beta}\right)
\bbE\left\| \tilde y^{k+1} - w^k \right\|_2^2 \\
\ &+ \frac{1}{N}\bbE\left\langle 
\frac{s}{{\hat s}} \matA \left( x^k - x^* \right) 
- \matA \left( \gradf{{x}^k} - \gradf{x^*} \right) 
,\ \tilde y^{k+1} - w^k\right\rangle \\
\ & +\frac{\beta}{2N^2\bar s_{\max}^2}\bbE\left\|\matA \matA^{\top}(w^k-y^k) - N\matA_i\matA_i^{\top}(w^k-y^k)\right\|_2^2 +\left(1-\frac1N\right)\bbE H\left(y^{k}\right)\\
\le\ & \frac{1}{2N}\bbE\left\|\matA^{\top}\left(w^k-y^*\right)\right\|_2^2
- \left( \frac{1}{{\hat s}} - \frac{\bar s_{\max}^2}{2} - \frac{\bar s_{\max}^2}{2\beta}\right)
\bbE\left\| \tilde y^{k+1} - w^k \right\|_2^2 \\
\ &+ \frac{1}{N}\bbE\left\langle 
\frac{s}{{\hat s}} \matA \left( x^k - x^* \right) 
- \matA \left( \gradf{{x}^k} - \gradf{x^*} \right) 
,\ \tilde y^{k+1} - w^k\right\rangle \\
\ & +\frac{\beta}{2N}\bbE\left\|\matA^{\top}(w^k-y^k)\right\|_2^2 +\left(1-\frac1N\right)\bbE H\left(y^{k}\right),
\end{aligned}
\end{equation}
where the last inequality comes from 
\(\bbE_i\left[N\matA_i\matA_i^{\top}(w^k-y^k)\right]=\matA\matA^{\top}(w^k-y^k)\), which means 
\begin{equation*}
\begin{aligned}
    &\bbE\left\|\matA \matA^{\top}(w^k-y^k) - N\matA_i\matA_i^{\top}(w^k-y^k)\right\|_2^2 \le \bbE\left\|N\matA_i\matA_i^{\top}(w^k-y^k)\right\|_2^2 \\ 
    \le\ & N^2\bar s_{\max}^2 \bbE\left\|\matA_i^{\top}(w^k-y^k)\right\|_2^2= N\bar s_{\max}^2 \bbE\left\|\matA^{\top}(w^k-y^k)\right\|_2^2.    
\end{aligned}
\end{equation*}
Combine \eqref{eq:y-acce-stoc_u_3}, \eqref{eq:y-acce-stoc_x_exp}, and \eqref{eq:y-acce-stoc_H_exp}, and by the settings of the weights \({\Xi_x}\), \({\Xi_h}\), \({\Xi_u}\), and other parameters in \eqref{eq:y-accelerate-stoc para}, we see that on the right hand side: \begin{itemize}
\item{the term \(\bbE\left\|\matA^{\top}\left(u^{k+1}-y^*\right)\right\|_2^2\) has a coefficient of \(\Xi_x \frac{\tilde t^2}{N}=\Xi_u\alpha\frac{{\hat s}}{N}\);}
\item{notice we have 
\begin{equation*}
\begin{aligned}
    &\Xi_h \cdot \frac{1}{N}\bbE\left\langle 
\frac{s}{{\hat s}} \matA \left( x^k - x^* \right) 
- \matA \left( \gradf{{x}^k} - \gradf{x^*} \right) 
,\ \tilde y^{k+1} - w^k\right\rangle\\ 
+\ &\Xi_u \cdot 2\frac{\yc}{N}\bbE\left\langle \yc {w}^{k} - (\yc-1) y^{k}-y^*,s\matA \left(x^{k}-x^*\right)-{\hat s} \matA \left( \gradf{x^{k}}-\gradf{x^*}\right)\right\rangle\\
=\ & \Xi_u \cdot 2\frac{\yc}{N} \bbE\left\langle u^{k+1}-y^*,s\matA \left(x^{k}-x^*\right)-{\hat s} \matA \left( \gradf{x^{k}}-\gradf{x^*}\right)\right\rangle,
\end{aligned}
\end{equation*}
the term \(\bbE\left\langle \matA^{\top}\left(u^{k+1}-y^*\right), x^{k}-x^*-t\left( \gradf{x^{k}}-\gradf{x^*}\right)\right\rangle\) has a coefficient of \(-{\Xi_x} 2  \frac{\tilde t}{N} + \Xi_u \cdot 2\frac{\yc s}{N}=0\);}
\item{the term \(\bbE\left\|\tilde y^{k+1}-w^k\right\|_2^2\) has a coefficient of \(-{\Xi_h} \left( \frac{1}{{\hat s}} - \frac{\bar s_{\max}^2}{2} - \frac{\bar s_{\max}^2}{2\beta}\right) + {\Xi_u} \yc^2 = -\frac{1}{2{\hat s}} + \frac{\bar s_{\max}^2}{2} + \frac{\bar s_{\max}^2}{2\beta}=0\);}
\item{
the term \(\bbE\left\|\matA^{\top}(w^k-y^k)\right\|_2^2\) has a coefficient \(\Xi_h \frac{\beta}{2N} - \Xi_u\left[\yc(\yc-1)-\alpha(1+\yeta)(\yc-1)\right]\frac{{\hat s}}{N} = -\Xi_u\left[\yc(\yc-1)-\alpha(1+\yeta)(\yc-1)-\beta \yc^2\right]\frac{{\hat s}}{N} \le 0\), when \(\yc(\yc-1)\ge\alpha(1+\yeta)(\yc-1)+ \beta \yc^2\);
}
\item{
the term \(\bbE H\left(y^k\right)\) has a coefficient \(\Xi_h \left(1-\frac1N\right)+ \Xi_u 2\left( \yc - \alpha \right) (\yc - 1) \frac{{\hat s}}N = 1 - \frac1N + \frac{1}{N}\frac{(\yc-\alpha)(\yc-1)}{\yc^2} \le 1-\frac{1}{N\yc} \le 1-\frac{1}{\yconv}\);
}
\item{
the term \(\bbE\left\|\matA^{\top}\left({w}^k-y^*\right)\right\|_2^2\) has a coefficient
\begin{equation*}\begin{aligned}
&-{\Xi_u}\left[\left(\yc^2+(1-\alpha)\yc\right)\frac{\hat s}{N}-\frac1{\yconv s_{\min}^2}\right]+{\Xi_h}\frac1{2N} \\
\le \ &-{\Xi_u}\left((1-\alpha)\yc\frac1{N\bar s_{\max}^2(1+1/\beta)}-\frac1{\yconv s_{\min}^2}\right)\\
\le\ &0,
\end{aligned}\end{equation*}
}
\end{itemize}
which concludes the proof.
\end{proof}

\section{Extension to the \texorpdfstring{$x$-side}{x-side} Stochastic Block-Coordinate Algorithm}
\label{appensubsec:x-stoc-nonsep}
In this section, we consider the extension of Algorithm~\ref{alg:x-accelerate-stoc} on the nonseparable objective:
\begin{equation}\label{eq:upper bound prob sto nonseparable}
\min_{x_i\in\bbR^{m_i}, i=1,\dots,N}\max_{y\in\bbR^n}\quad F(x,y)=f(x_1,\dots,x_N)+y^{\top} \left(\sum_{i=1}^N\matA_i x_i-b\right)= f(x)+y^{\top}\left(\matA x-b\right).
\end{equation}
We consider the similar assumptions in \Cref{assum:Assumption Linear Stoc} for \eqref{eq:upper bound prob sto nonseparable}, except the second item for the separable objective is replaced with the corresponding one for \(f(x)\) in \Cref{assum:Assumption Linear New} 
\begin{assumption}\label{assum:Assumption Linear Stoc nonseparable}
Suppose problem \eqref{eq:upper bound prob sto nonseparable} satisfies the following properties:
\begin{enumerate}
\item For $i=1,\ldots,N$, $x_i \in \bbR^{m_i}$, and $y \in \bbR^n$. Let $\matA \in \bbR^{n\times m}$ be the coupling matrix with full row rank and minimal singular value $s_{\min} > 0$. For each block matrix $\matA_i$, its maximal singular value is no larger than $\bar s_{\max}$.
\item $f:\bbR^{m}\to\bbR$ is globally $\mu$-strongly convex for some $\mu > 0$, and globally $L$-smooth for some $L \ge \mu>0$.
\end{enumerate}
\end{assumption}

\begin{algorithm}[t]
\caption{$x$-side stochastic block-coordinate accelerated primal--dual algorithm (nonseparable primal objective \(f\))}
\label{alg:x-accelerate-stoc-nonseparable}
\begin{algorithmic}[1]
\REQUIRE Parameters $t, s, \hat s, \xh, \omega, \xtau, \xga > 0, \xe > 1.$
\STATE Initialize $x^{0}, z^{0}, v^{0} \in \bbR^{m}$, $y^{0} \in \bbR^{n}$
\REPEAT
  \STATE $\hat x^{k} = \xe z^{k} - (\xe-1)x^{k}$
  \STATE Randomly sample $i \sim U(\{1,\ldots,N\})$
  \STATE $y^{k+1} = y^{k} + \frac{\xh s}{N}(\matA \hat x^{k} - b)- \hat s \matA_i(\matA_i^{\top} y^{k} + \nabla_i f(z^{k}))$
  \STATE Randomly sample $j \sim U(\{1,\ldots,N\})$
  \STATE $x_j^{k+1} = z_j^{k} - t(\nabla_j f(z^{k}) + \matA_j^{\top} y^{k+1})$
  \STATE $x_l^{k+1} = z_l^{k}$ for $l \in \{1,\ldots,N\}\backslash j$
  \STATE $z^{k+1} = (1+\xga)\left(z^{k} + \omega(x^{k+1}-z^{k})\right) - \xga x^{k}$
  \STATE $v^{k+1} = (1+\xtau)z^{k+1} - \xtau x^{k+1}$
\UNTIL{convergence}
\end{algorithmic}
\end{algorithm}

\begin{theorem}\label{thm:x-accelerate-stoc-2}
Suppose problem \eqref{eq:upper bound prob sto nonseparable} satisfies \Cref{assum:Assumption Linear Stoc nonseparable}. Consider applying Algorithm~\ref{alg:x-accelerate-stoc-nonseparable} to solve \eqref{eq:upper bound prob sto nonseparable}. Let \(\xgt, \beta, t,{\hat s}>0\) satisfy
\begin{equation}\label{eq:x-accelerate-stoc-nonseparable para-0}
t = \frac{1 - 4\xgt-2(1+\frac{1}{\beta}){\hat s} \bar s_{\max}^2}{{L} + 4{L}\frac{\xgt}{N}}>0,\quad 1 - {\hat s} \bar s_{\max}^2 - 2(1+\beta){\hat s}^2 \bar s_{\max}^4\ge 0,
\end{equation}
and the remaining parameters be
\begin{equation}\label{eq:x-accelerate-stoc-nonseparable para-1}
\begin{aligned}
s &= \frac{\hat s}{t},
& \quad \Xi_f &= 1, \\
\xconv &\ge N\max\left(\frac{1}{s_{\min}^2 {\hat s}} \cdot \frac{1}{2\xgt},\ \sqrt{\frac{1}{\mu t}}+\frac{L}{\mu}\cdot 4\xgt \right),
& \quad \xe &= \frac{1 + 4L\frac{\xgt}{N} t}{1/\xconv + 4 L\frac{\xgt}{N} t}, \\
\xtau &= \frac{\xe - 1}{1 - 1/\xconv},
& \quad \xga &= \frac{\xe - 1}{\xtau + 1}, \\
\Xi_y &= \frac{1}{2s} = \frac{t}{2\hat s},
& \quad \Xi_v &= N^2\frac{1 + 4 L\frac{\xgt}{N} t}{2\xe^2 t}, \\
\xh & = \frac{\Xi_v \xe t}{N\Xi_y s}, & \quad \omega & = \frac{\xtau + \frac{\xe}{N}}{(1+\xga)(1+\xtau)}. 
\end{aligned}
\end{equation}
Then, let \begin{equation*}
\Psi^k = {\Xi_y}\left\|y^k-y^*\right\|^2_{\left(I-(1-2\xgt)\frac{\hat s}{N} \matA \matA^{\top}\right)}+\D{x^k}+{\Xi_v}\left\|{v}^k-x^*\right\|_2^2,
\end{equation*} 
we have \begin{equation*}\bbE\Psi^{k+1}\le (1-1/\xconv)\bbE\Psi^k.\end{equation*}
\end{theorem}
\begin{proof}
It suffices to prove the following two inequalities:
\begin{equation}\label{eq:x-acce-stoc-2_df}
\begin{aligned}
\bbE\D{x^{k+1}} 
&\le (1+4L\frac{\xgt}{N} t)\bbE\D{z^k}- \frac{t}{N}\left(1 - \frac{L}{2} t - 2\xgt \right) 
  \bbE \left\| \gradf{z^k} + \matA^{\top} y^{k+1} \right\|_2^2 \\
&\quad - \frac{\xgt}{N} t \bbE \left\| \matA^{\top} \left(y^{k+1} - y^*\right) \right\|_2^2 
+ \frac{t}{N} \bbE \left\langle \matA^{\top} \left(y^{k+1} - y^*\right),\, 
  \gradf{z^k} + \matA^{\top} y^{k+1} \right\rangle,
\end{aligned}    
\end{equation}
and
\begin{equation}\label{eq:x-acce-stoc-2_v}
\begin{aligned}
\bbE\left\|v^{k+1} - x^*\right\|_2^2
&\le \left(1 - \frac{1}{\xconv}\right) \bbE\left\|v^k - x^*\right\|_2^2
- 2\frac{\xe^2}{N} \frac{t}{N}\, \bbE\D{z^k}
+ 2\frac{\xe(\xe - 1)}{N}\frac{t}{N}\, \bbE\D{x^k} \\
&\quad - 2\frac{\xe}{N} \frac{t}{N} \bbE\left\langle \matA^{\top} \left(y^{k+1} - y^*\right),\,
\hat{x}^k - x^* \right\rangle + \frac{\xe^2}{N^2} \frac{t^2}{N} \bbE\left\| \nabla f(z^k) + \matA^{\top} y^{k+1} \right\|_2^2.
\end{aligned}    
\end{equation}
Once \eqref{eq:x-acce-stoc-2_df} and \eqref{eq:x-acce-stoc-2_v} are established, the remainder of the proof follows the same Lyapunov-combination argument as in the proof of Theorem~\ref{thm:x-accelerate-stoc}: we similarly apply \Cref{prop:x-accelerate-stoc-3} (notice \Cref{prop:x-accelerate-stoc-3} still holds with \(\gradfp{l}{z^k}\) replaced by \(\nabla_{l}f\left(z^k\right)\)), then combine the components
\(\left\|y^k-y^*\right\|^2_{\left(I-(1-2\xgt)\frac{\hat s}{N}\matA \matA^{\top}\right)}\), \(\D{x^k}\), and \(\left\|v^k-x^*\right\|_2^2\),
and verify that the cross terms cancel to obtain \(\bbE\Psi^{k+1}\le (1-1/\xconv)\bbE\Psi^k\).

Similar to the proof of \Cref{prop:x-accelerate-stoc-1}, we have 
\begin{equation*}
\begin{aligned}
\D{x^{k+1}} 
&\le \D{z^k}+2\xgt t\left\|\nabla_j f\left(z^k\right)-\nabla_j f\left(x^*\right)\right\|_2^2 \\
&\quad- t \left(1 - \frac{L}{2} t - 2\xgt \right) 
  \left\| \nabla_j f\left(z^k\right) + \matA_j^{\top} y^{k+1} \right\|_2^2 \\
&\quad - \xgt t \left\| \matA_j^{\top} \left(y^{k+1} - y^*\right) \right\|_2^2 
+ t \left\langle \matA_j^{\top} \left(y^{k+1} - y^*\right),\, 
  \nabla_j f\left(z^k\right) + \matA_j^{\top} y^{k+1} \right\rangle.
\end{aligned}
\end{equation*}
By taking the expectation on \(j\), and applying \(\left\|\gradf{z^k}-\gradf{x^*}\right\|_2^2\le 2L \D{z^k}\), we get \eqref{eq:x-acce-stoc-2_df}.

Similar to \eqref{eq:x-acce-stoc-2-init}, we have:
\begin{equation*}
\begin{aligned}
&\left\|v^{k+1} - x^*\right\|_2^2 
= \left\| \xe z^{k} - (\xe - 1) x^k - x^* +\frac{\xe}{N}\left(x^{k+1}-z^k\right)\right\|_2^2 \\
=\  &\left\| \xe z^k - (\xe - 1) x^k - x^* \right\|_2^2 
- 2 \frac{\xe}{N} t \left\langle 
  \nabla_j f\left(z^k\right) - \nabla_j f\left(x^*\right),\, 
  \xe z_j^k - (\xe - 1) x_j^k - x_j^* 
\right\rangle \\
&\quad - 2 \frac{\xe}{N} t \left\langle 
  \matA_j^{\top} \left( y^{k+1} - y^* \right),\, 
  \hat{x}_j^k - x_j^* 
\right\rangle 
+ \frac{\xe^2}{N^2} t^2 \left\| \nabla_j f\left(z^k\right) + \matA_j^{\top} y^{k+1} \right\|_2^2,
\end{aligned}
\end{equation*}
where the first equation comes from the setting of \(\omega\) in \eqref{eq:x-accelerate-stoc-nonseparable para-1}.
By taking the expectation on \(j\), we get
\begin{equation*}
\begin{aligned}
\bbE\left\|v^{k+1} - x^*\right\|_2^2 
=\ &\bbE\left\| \xe z^k - (\xe - 1) x^k - x^* \right\|_2^2 
- 2 \frac{\xe}{N} \frac{t}{N} \bbE\left\langle 
  \matA^{\top} \left( y^{k+1} - y^* \right),\, 
  \hat{x}^k - x^* 
\right\rangle\\
&\quad - 2 \frac{\xe}{N} \frac{t}{N} \bbE\left\langle 
  \gradf{z^k} - \gradf{x^*},\, 
  \xe z^k - (\xe - 1) x^k - x^* 
\right\rangle \\
&\quad + \frac{\xe^2}{N^2} \frac{t^2}{N} \bbE\left\| \gradf{z^k}
+ \matA^{\top} y^{k+1} \right\|_2^2.
\end{aligned}
\end{equation*}
And the remainder proof to get \eqref{eq:x-acce-stoc-2_v} is similar to Proposition~\ref{prop:x-accelerate-stoc-2}.
\end{proof}

\section{Additional Experimental Details and Results}
\label{appensec:supp numerical}
\paragraph{Hardware and runtime.} All experiments were conducted on a MacBook Pro (Model Identifier: Mac16,6) running macOS 15.5, equipped with an Apple M4 Max chip (16 cores: 12 performance and 4 efficiency) and 64 GB of memory. We did not use GPU acceleration.

Wall-clock runtimes were measured as the elapsed time to execute the corresponding Jupyter notebooks end-to-end (including computation and plotting).

Generating the left and right panels of Figure~\ref{Fig_cst_main} takes 15.3 s and 16.0 s, respectively.

For Figure~\ref{Fig_cst_stoc}, the left and right panels take 1,895.5 s and 1,933.9 s, respectively.

For Figure~\ref{Fig_cst_stoc_appen}, the left and right panels take 2,151.2 s and 1,278.7 s, respectively. 

For Figure~\ref{Fig_NSE}, the left and right panels take 87.8 s and 484.8 s, respectively.

Generating the left and right panels of Figure~\ref{Fig_cst_appen} takes 117.9 s and 116.6 s, respectively.

For Figure~\ref{Fig_cst_metric}, the left and right panels take 13.0 s and 13.5 s, respectively.

Tables~\ref{tab:log-ci-errors-cst1} and \ref{tab:log-ci-errors-cst2} take 142.8 s and 144.3 s, respectively.

Tables~\ref{tab:log-ci-errors-nse1} and \ref{tab:log-ci-errors-nse2} take 855.2 s and 4,453.8 s, respectively.

Finally, generating the left and right panels of Figure~\ref{fig_qp} takes 15.0 s and 13.7 s, respectively.

Overall, generating all results referenced above takes 13,748.9 s (about 3.8 hours).

For the stochastic convergence plots (Figures~\ref{Fig_cst_stoc} and \ref{Fig_cst_stoc_appen}), we log the error at a fixed BMM stride of 40,000, which corresponds to logging every 100 iterations for \texttt{PAPC}/\texttt{x-DAPD}/\texttt{y-DAPD}, every 10,000 iterations for \texttt{x-SBC-DAPD} (\texttt{x-SBC-DAPD-i}), and every 6,666 iterations for \texttt{y-SBC-DAPD} (\texttt{y-SBC-DAPD-i}), yielding 2000 plotted points per curve. \texttt{CAPD} logs every outer iteration; due to its Chebyshev inner loop this yields 2001 points on the left panel and 201 points on the right panel.

For Figure~\ref{Fig_cst_stoc}, we match all methods by the total number of block matrix multiplications (BMMs), which is the dominant arithmetic cost in our complexity model. However, wall-clock runtime is not perfectly proportional to BMM counts in our Python/Jupyter implementation: the stochastic variants expend the same BMM budget via many more (lighter) iterations, incurring additional interpreter and function-call overhead (as well as bookkeeping and memory-management costs). Consequently, the stochastic methods can take substantially longer in wall-clock time than the deterministic methods despite having matched BMM complexity. This effect is implementation-dependent and should diminish with a more optimized implementation.

\subsection{Supplementary Discussion and Results to Section~\ref{subsec:cst}}
\label{appensec:supp cst}
Note that the complexity of \texttt{CAPD} is \(O\left(\frac{s_{\max}}{s_{\min}}\sqrt{\frac{L}{\mu}}\log(\frac{1}{\epsilon})\right)\) in terms of the inner loop iterations (or the number of matrix multiplications, i.e. \(\matA\) times \(x\) or \(\matA^{\top}\) times y), and \(O\left(\sqrt{\frac{L}{\mu}}\log(\frac{1}{\epsilon})\right)\) in terms of the outer loop iterations (or the number of gradient computations). Since for this problem, the matrix multiplications take many more calculations compared with the elementwise gradient computations, we denote the number of iterations to be the inner loop iterations of \texttt{CAPD}. It is fair to compare the number of iterations since all these algorithms require two matrix multiplications per iteration (and 1 gradient computation per iteration except \texttt{CAPD}) as shown in Remark~\ref{remark:num M_MT_grad}.

\Cref{Fig_cst_appen} shows the same simulations as \Cref{Fig_cst_main}, extended over a larger number of iterations.
\begin{figure}[t]
\begin{minipage}{0.5\linewidth}
\centering
\includegraphics[width=1\linewidth]{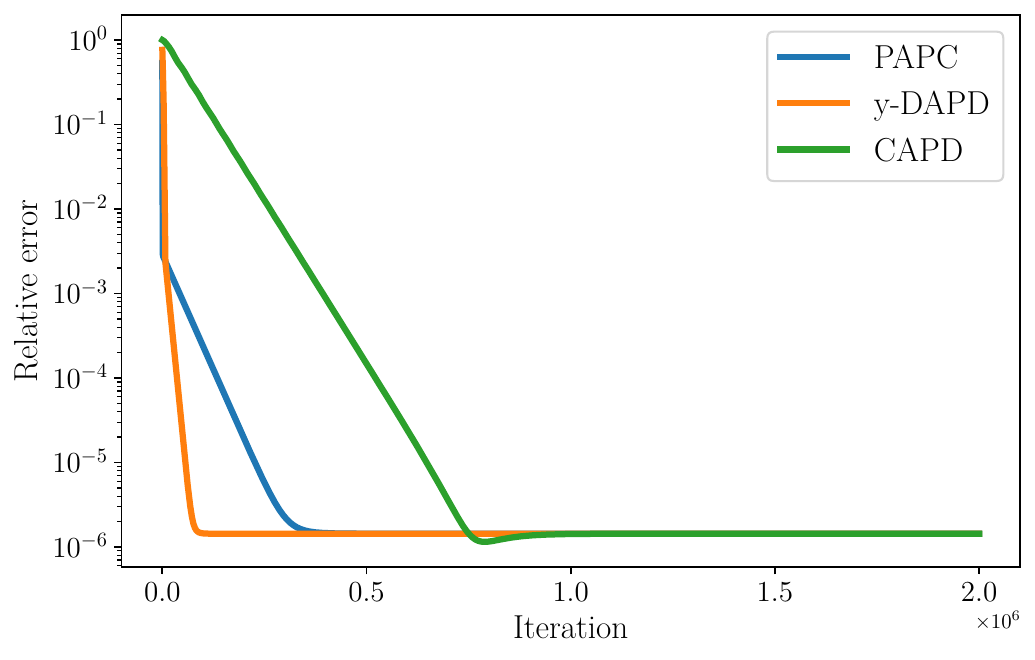}
\end{minipage}%
\begin{minipage}{0.5\linewidth}
\centering
\includegraphics[width=1\linewidth]{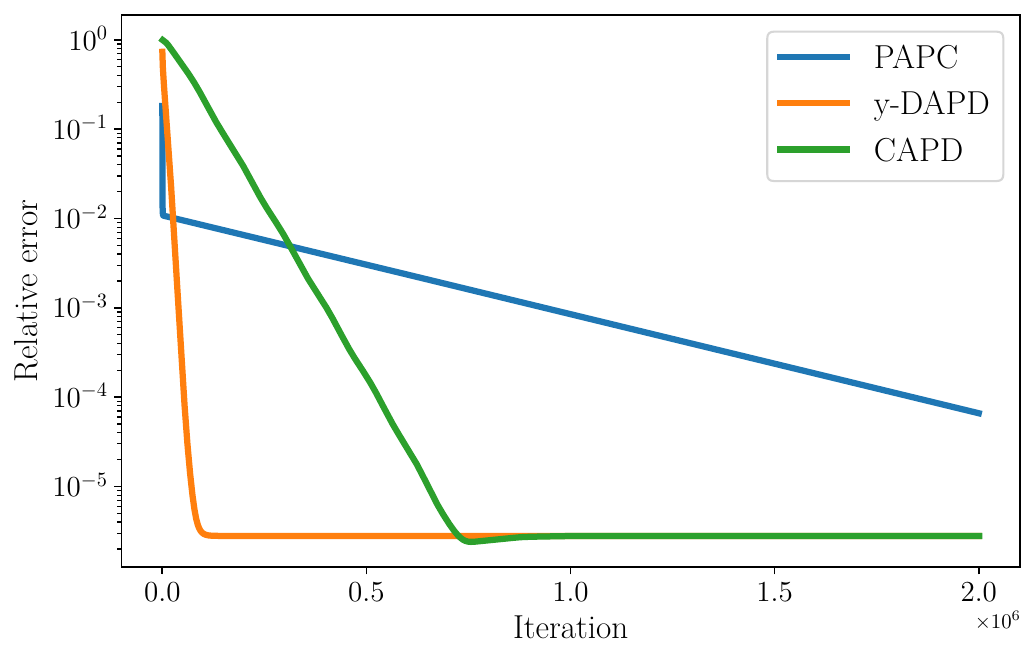}
\end{minipage}
\caption{Results for the Compressed-sensing-type experiment with longer iterations. Left: \(\frac{s_{\max}^2}{s_{\min}^2} = 10^5\), \(\frac{L}{\mu}=10^4\); right: \(\frac{s_{\max}^2}{s_{\min}^2} = 10^6\), \(\frac{L}{\mu}=10^3\).}
\label{Fig_cst_appen}
\end{figure}

We also conducted 20 independent simulations with randomly generated problem instances and algorithm executions with the same settings as Figure~\ref{Fig_cst_main}. The 95\% confidence intervals of the relative errors after \(10^5\) iterations are reported in \Cref{tab:log-ci-errors-cst1,tab:log-ci-errors-cst2}. Since the relative errors vary across several orders of magnitude and are inherently multiplicative, we apply a base-10 logarithmic transformation to stabilize the variance and better capture the distribution. Confidence intervals are computed on the log-transformed errors and exponentiated for presentation.

\begin{table}[tb]
  \caption{95\% confidence intervals of the relative error (log-transformed and exponentiated), \(\frac{s_{\max}^2}{s_{\min}^2} = 10^5\), \(\frac{L}{\mu}=10^4\).}
  \label{tab:log-ci-errors-cst1}
  \centering
  \begin{tabular}{lc}
    \toprule
    Method & 95\% Confidence Interval \\
    \midrule
    \texttt{PAPC} & [2.299e-04, 5.625e-04] \\
    \texttt{CAPD} & [1.808e-01, 1.854e-01] \\
    \texttt{y-DAPD} & [2.005e-07, 5.786e-07] \\
    \bottomrule
  \end{tabular}
\end{table}

\begin{table}[tb]
  \caption{95\% confidence intervals of the relative error (log-transformed and exponentiated), \(\frac{s_{\max}^2}{s_{\min}^2} = 10^6\), \(\frac{L}{\mu}=10^3\).}
  \label{tab:log-ci-errors-cst2}
  \centering
  \begin{tabular}{lc}
    \toprule
    Method & 95\% Confidence Interval \\
    \midrule
    \texttt{PAPC} & [9.188e-03, 2.318e-02] \\
    \texttt{CAPD} & [2.049e-01, 2.131e-01] \\
    \texttt{y-DAPD} & [7.371e-07, 2.769e-06] \\
    \bottomrule
  \end{tabular}
\end{table}

One may wonder why some methods in our convergence plots (e.g., Figure~\ref{Fig_cst_main}) eventually enter a plateau where the residual stops decreasing, even though our theory guarantees linear convergence. This behavior is due to how the residual is evaluated in our numerical experiments. For most synthetic instances, the true optimal solution is not available in closed form, so we use a Gurobi solution as the reference point. Consequently, the residual is only as accurate as this reference solution. Once an algorithm reaches the accuracy level of the reference—especially for faster methods—the measured residual can no longer decrease and appears to plateau.

To confirm that the iterates continue to converge, we also report KKT-based metrics
\(\max(\|\nabla f(x)+\matA^{\top}y\|_2,\ \|\matA x-b\|_2)\).
\cref{Fig_cst_metric} plots these metrics versus iterations for the runs in \cref{Fig_cst_main}. We observe that the metrics keep decreasing linearly to below \(10^{-11}\) even after the residual to the reference solution has plateaued. In contrast, the corresponding KKT metrics of the Gurobi reference solutions are \(\left\|\nabla f\left(x^{\mathrm{ref}}\right)+\matA^{\top}y^{\mathrm{ref}}\right\|_2=4.2\times 10^{-4}\) and \(\left\|\matA x^{\mathrm{ref}}-b\right\|_2=1.7\times 10^{-11}\) for the left panel, and \(\left\|\nabla f\left(x^{\mathrm{ref}}\right)+\matA^{\top}y^{\mathrm{ref}}\right\|_2=8.0\times 10^{-5}\) and \(\left\|\matA x^{\mathrm{ref}}-b\right\|_2=2.4\times 10^{-10}\) for the right panel. (We omit \texttt{CAPD} in this plot because it does not explicitly output the dual variable \(y\) during the iterations.)

\begin{figure}[t]
\begin{minipage}{0.5\linewidth}
\centering
\includegraphics[width=\linewidth]{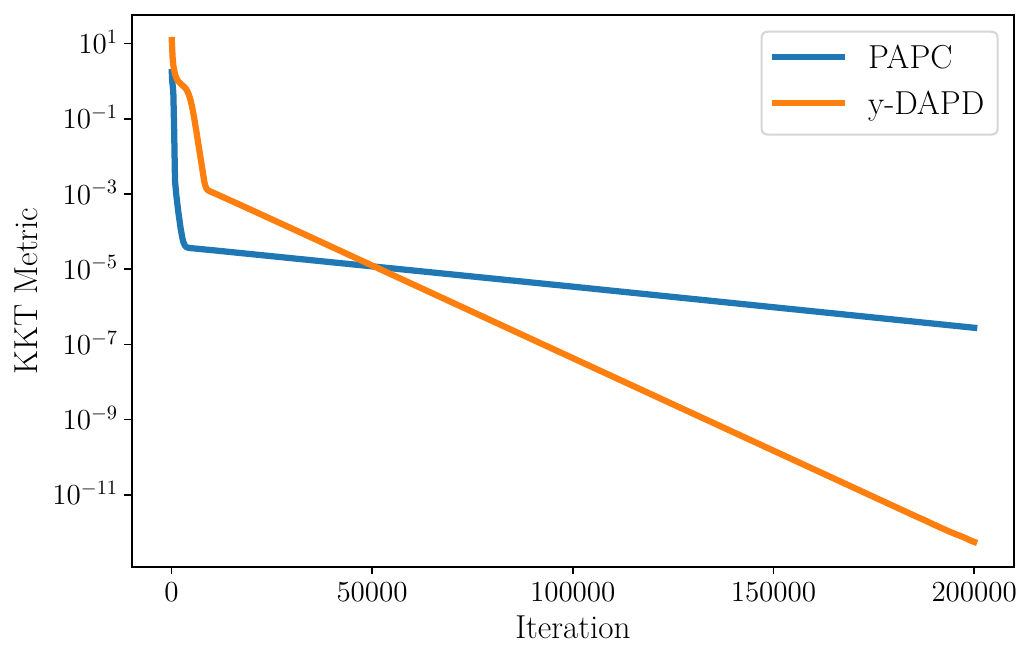}
\end{minipage}%
\begin{minipage}{0.5\linewidth}
\centering
\includegraphics[width=\linewidth]{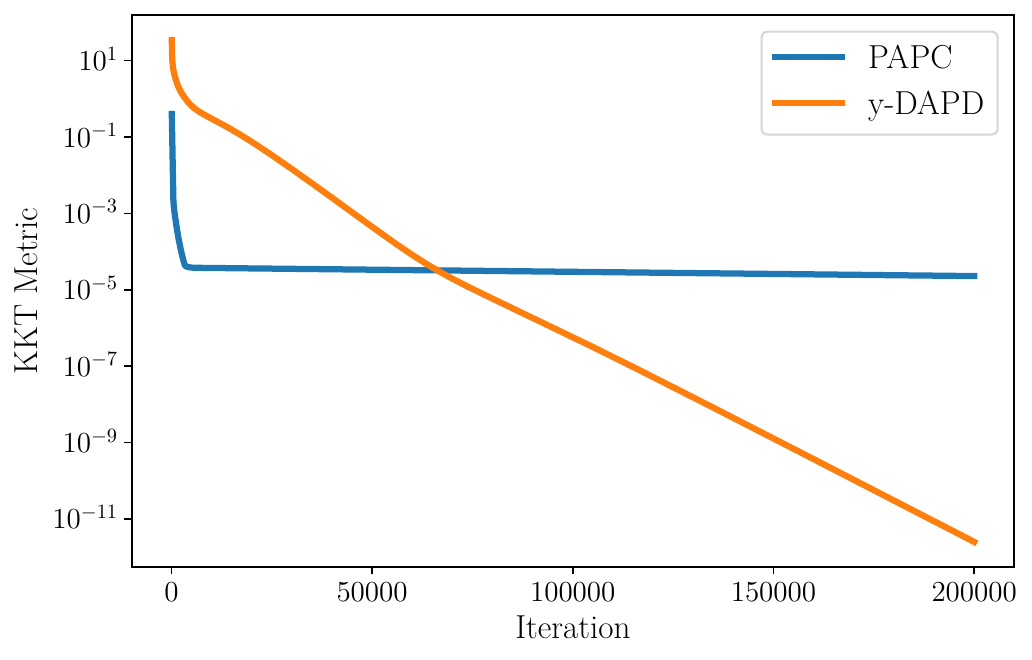}
\end{minipage}
\caption{Convergences of KKT metrics for the compressed-sensing-type experiment, Left: \(\frac{s_{\max}^2}{s_{\min}^2} = 10^5\), \(\frac{L}{\mu}=10^4\); right: \(\frac{s_{\max}^2}{s_{\min}^2} = 10^6\), \(\frac{L}{\mu}=10^3\).}
\label{Fig_cst_metric}
\end{figure}

\Cref{Fig_cst_stoc_appen} shows the same simulations as \Cref{Fig_cst_stoc}, with additional comparisons of Algorithms~\ref{alg:x-accelerate-stoc} and~\ref{alg:y-accelerate-stoc} corresponding to the case \(j=i\). We observe that the performance difference between using two independent samples \((i,j)\) and a single shared sample \((i=j)\) is minor.
\begin{figure}[t]
\begin{minipage}{0.5\linewidth}
\centering
\includegraphics[width=1\linewidth]{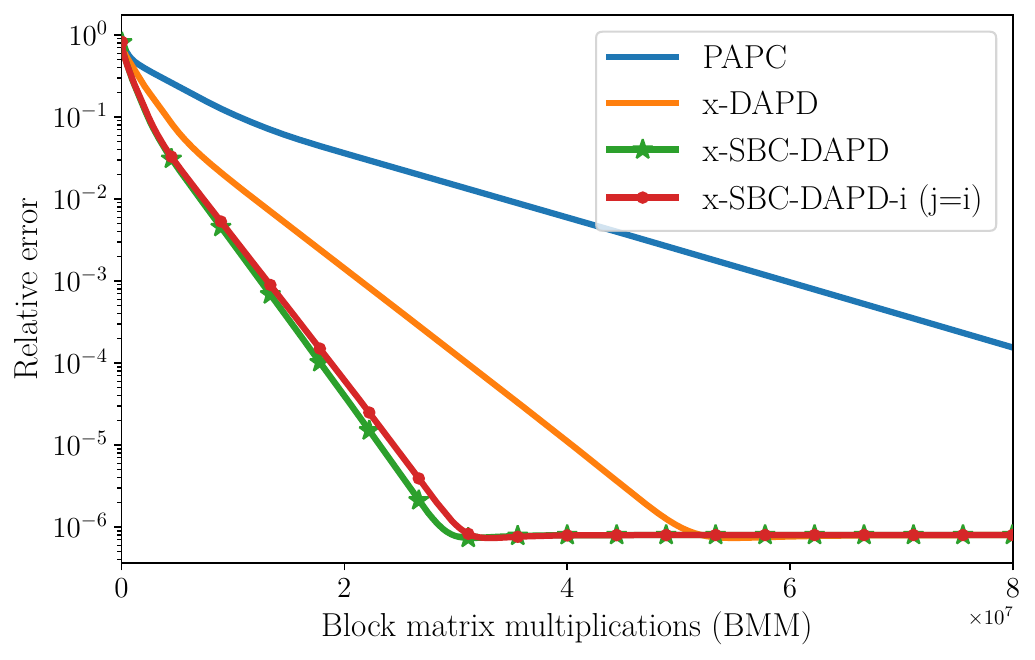}
\end{minipage}%
\begin{minipage}{0.5\linewidth}
\centering
\includegraphics[width=1\linewidth]{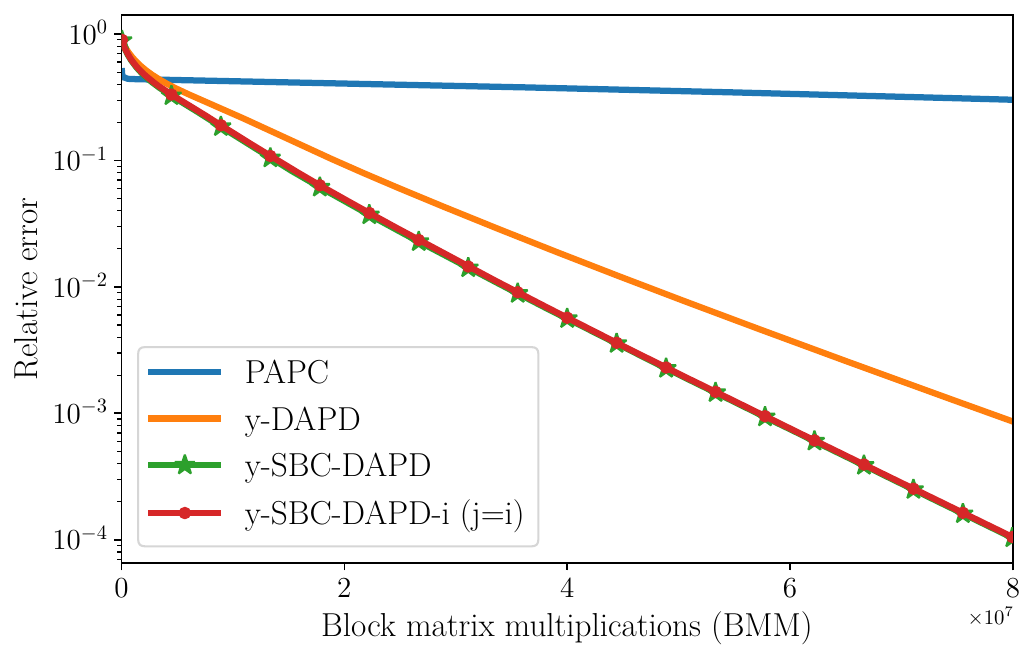}
\end{minipage}
\caption{Results for the compressed-sensing-type experiment with \texttt{SBC-DAPD} with sampling strategy \(j=i\). Left: \(\frac{s_{\max}^2}{s_{\min}^2} = 10^4\), \(\frac{L}{\mu}=10^5\); right: \(\frac{s_{\max}^2}{s_{\min}^2} = 10^6\), \(\frac{L}{\mu}=10^3\).}
\label{Fig_cst_stoc_appen}
\end{figure}
\subsection{Supplementary Discussion and Results to Section~\ref{subsec:nse}}
\label{appensec:supp nse}
The problem \eqref{eq:nse reg short} is equivalent to:
\begin{equation}\label{eq:nse reg}
\begin{aligned}
&\min _{\X \in \mathbb{R}^{\p \times \p}} f\left(\mathrm{vec}(\X)\right)+\lambda\sum_{t=1}^{\T}\left\|\s_{t}-\X \s_{t-1}\right\|_2\\=& \min_{\X \in \mathbb{R}^{\p \times \p}} \max_{Y\in\bbR^{\p\times \T}}f\left(\mathrm{vec}(\X)\right)+\sum_{t=1}^{\T}\left[y_t^{\top}\left(\s_{t}-\X \s_{t-1}\right)-I_{\left\|y\right\|_2\le \lambda}\left(y_t\right)\right]\\=&
\min_{\X \in \mathbb{R}^{\p \times \p}} \max_{Y\in\bbR^{\p\times \T}:\left\|\mathrm{vec}(Y)\right\|_{2,\infty}\le \lambda}f\left(\mathrm{vec}(\X)\right)+\mathrm{vec}(Y)^{\top}\left[\mathrm{vec}(S_{1,\T})-\left(S_{0,\T-1}^{\top}\otimes I_{\p}\right) \mathrm{vec}(\X)\right]
\end{aligned}
\end{equation}
where \(Y=(y_1,\dots,y_{\T})\) is the dual variable, \(S_{1,\T}=\left(\s_1,\dots,\s_{\T}\right)\), \(S_{0,\T-1}=\left(\s_0,\dots,\s_{\T-1}\right)\), \(I_{\p}\) is the \(\p\times\p\) identity matrix, \(\otimes\) is the Kronecker product, and \(\left\|\mathrm{vec}(Y)\right\|_{2,\infty}=\max_{t}\left\|y_t\right\|_2\).

Despite the size of the bilinear coupling matrix \(\left(S_{0,\T-1}^{\top}\otimes I_{\p}\right)\in \bbR^{\p\T\times \p^2}\), its singular values consist of \(\p\) copies of singular values of \(S_{0,\T-1}\in\bbR^{\p\times\T}\), which enables efficient calculation or estimation of the condition number of the bilinear coupling matrix. In addition, the first-order updates of \(X\) and \(Y\) can be done efficiently in Algorithms~\ref{alg:x-accelerate-prox} and \ref{alg:y-accelerate-prox} without doing \(\mathrm{vec}(Y)^{\top}\left(S_{0,\T-1}^{\top}\otimes I_{\p}\right)\) or \(\left(S_{0,\T-1}^{\top}\otimes I_{\p}\right) \mathrm{vec}(X)\) directly. For example, \begin{equation*}
    \left(S_{0,\T-1}^{\top}\otimes I_{\p}\right) \mathrm{vec}(\X)=\begin{pmatrix}
    \X\s_0\\
    \vdots\\
    \X\s_{\T-1}
\end{pmatrix}.
\end{equation*}

For the experiments, we adopt a similar simulation method as in \citep{yalcin2025subgradient}: we generate a \(\p\times\p\) random matrix with approximately 20\% non-zero entries, where the non-zero values follow a standard Gaussian distribution, and normalize it to have a unit \(2\)-norm to get \(\bar {\X}\). The initial vector \(s_0\) is drawn from the multivariate normal distribution \(\mathcal{N}(0, I_p)\). Given the probability $\rho$, disturbance vectors are set to zero with probability $1-\rho$. With probability $\rho$, the disturbance $\bar d_t$ is defined as $\bar d_t:=\ell_t \hat{d}_t$. Then, we sample $\ell_t \sim \mathcal{N}\left(0, \sigma_t^2\right)$, where $\sigma_t^2:=\min \left\{\left\|\s_t\right\|_2^2, 1 / \p\right\}$, and we sample $\hat{d}_t \sim$ uniform $\left(S^{\p-1}\right)$. We set \(\rho=0.7\) and \(\lambda=1000\).

We also conducted 20 independent simulations with randomly generated problem instances and algorithm executions. The 95\% confidence intervals of the relative errors after \(5\times10^5\) iterations are reported in \Cref{tab:log-ci-errors-nse1,tab:log-ci-errors-nse2}. Confidence intervals are computed on the log-transformed errors and exponentiated for presentation.

\begin{table}[tb]
  \caption{95\% confidence intervals of the relative error (log-transformed and exponentiated), \(\p=40\), \(\T=10\), \(\frac{L}{\mu}=10^9\)}
  \label{tab:log-ci-errors-nse1}
  \centering
  \begin{tabular}{lc}
    \toprule
    Method & 95\% Confidence Interval \\
    \midrule
    \texttt{PAPC} & [3.99e-02, 8.36e-02] \\
    \texttt{x-DAPD} & [5.47e-03, 1.58e-02] \\
    \texttt{y-DAPD} & [1.41e-01, 1.89e-01] \\
    \bottomrule
  \end{tabular}
\end{table}

\begin{table}[tb]
  \caption{95\% confidence intervals of the relative error (log-transformed and exponentiated), \(\p=100\), \(\T=20\), \(\frac{L}{\mu}=10^4\)}
  \label{tab:log-ci-errors-nse2}
  \centering
  \begin{tabular}{lc}
    \toprule
    Method & 95\% Confidence Interval \\
    \midrule
    \texttt{PAPC} & [8.28e-03, 4.27e-02] \\
    \texttt{x-DAPD} & [1.60e-02, 6.46e-02] \\
    \texttt{y-DAPD} & [1.54e-05, 3.69e-04] \\
    \bottomrule
  \end{tabular}
\end{table}

\subsection{Additional Experiments}
\label{appensec:exp}

\subsubsection{Quadratic programming with \texorpdfstring{\(\ell_1\)}{l1}-norm on the dual variable}
\label{subsec:l1}

We illustrate the effect of proximal terms on algorithm performance in the following experiments. Specifically, we consider a quadratic programming problem with an \(\ell_1\) penalty on the dual variable:
\begin{equation}\label{eq:l1 numerical}
\min_{x\in\bbR^m}\max_{y\in\bbR^n}\quad f(x) + y^{\top}(\matA x -b) - \nu \left\|y\right\|_1, 
\end{equation}
which is equivalent to 
\begin{equation*}
\begin{aligned}
\min\quad &f(x)\\
\text{s.t.}\quad &\left\|\matA x -b\right\|_{\infty}\le \nu.
\end{aligned}
\end{equation*}

Here, \(f:\bbR^m\to\bbR\) is a strongly convex quadratic function \(f(x)=\frac{1}{2}x^{\top}Hx-c^{\top}x\), where \(H\in \mathcal{S}_{++}^{m}\) is a symmetric positive definite matrix. It is constructed as \(H = P\Lambda P^{\top}\), where \(P\) is an orthogonal matrix obtained from the QR decomposition of a Gaussian random matrix, and \(\Lambda\) is diagonal with entries sampled i.i.d. from \(\mathcal{U}[0,1]\), then rescaled so that \(\min_i \lambda_i = \mu\) and \(\max_i \lambda_i = L\). The matrix \(\matA\in\bbR^{n\times m}\) is generated similarly as \(\matA = USV^{\top}\), where \(U\in\bbR^{n\times n}\) and \(V\in \bbR^{m\times m}\) are orthogonal, and the singular values in \(S\) are scaled to satisfy \(\min_i s_i = s_{\min}\), \(\max_i s_i = s_{\max}\). Vectors \(b\in\bbR^n\) and \(c\in\bbR^m\) are sampled from standard Gaussian distributions.

In the experiment shown in the left panel of Figure~\ref{fig_qp}, we set \(m = 300\), \(n = 100\), \(L = 1000\), \(\mu = 1\), \(s_{\max} = 1000\), \(s_{\min} = 1\), and \(\nu = 0.01\). We observe that \texttt{y-DAPD} converges significantly faster than the other methods. \texttt{x-DAPD} performs comparably to \texttt{PAPC}. We expect similar behavior from the intermediate algorithm in \citep{salim2022optimal} if an appropriate proximal mapping is incorporated, as all three algorithms have an \(O\left(\frac{s_{\max}^2}{s_{\min}^2}\log\left(\frac{1}{\epsilon}\right)\right)\) iteration complexity.

\subsubsection{Quadratic programming with inequality constraints}
\label{subsec:ineq}

We now consider a quadratic programming problem with inequality constraints:
\begin{equation}\label{eq:ineq numerical}
\min_{x\in\bbR^m}\max_{y\in\bbR^n}\quad f(x) + y^{\top}(\matA x -b) - I_{\bbR^n_+}(y),
\end{equation}
where \(f(x) = \frac{1}{2}x^{\top}Hx - c^{\top}x\) and \(H\) is generated as in \Cref{subsec:l1}. We first generate a ground truth solution \(x^*\) from a Gaussian distribution. Then we partition the constraints into active and inactive sets such that:
\begin{itemize}
    \item Active constraints satisfy \(\matA^{\mathrm{active}}x^* = b^{\mathrm{active}}\),
    \item Inactive constraints satisfy \(\matA^{\mathrm{inactive}}x^* < b^{\mathrm{inactive}}\).
\end{itemize}

Both \(\matA^{\mathrm{active}}\) and \(\matA^{\mathrm{inactive}}\) are generated similarly to \(\matA\), with controlled singular values: \([s_{\min}^{\mathrm{active}}, s_{\max}^{\mathrm{active}}]\) and \([s_{\min}^{\mathrm{inactive}}, s_{\max}^{\mathrm{inactive}}]\), respectively. The inactive component is constructed by computing \(\bar b^{\mathrm{inactive}} = \matA^{\mathrm{inactive}}x^*\), and perturbing it as \(b_i^{\mathrm{inactive}} = \bar b_i^{\mathrm{inactive}} + |\varepsilon_i \bar b_i^{\mathrm{inactive}}|\), where \(\varepsilon_i \sim \mathcal{N}(0, 1)\). The dual variable for the active part, \(y^{\mathrm{active}}\), is generated from the absolute values of a Gaussian vector, and the primal vector \(c\) is then computed by \(c = Hx^* + \left(\matA^{\mathrm{active}}\right)^{\top}y^{\mathrm{active}}\).

This construction ensures that the condition number of the full constraint matrix \(\matA\) aligns with that of the active part, avoiding artificially easy or ill-conditioned cases.

In the experiment shown in the right panel of Figure~\ref{fig_qp}, we set \(m = 300\), \(n^{\mathrm{active}} = 50\), \(n^{\mathrm{inactive}} = 50\), \(L = 1000\), \(\mu = 1\), and for both parts \(s_{\min} = 1\), \(s_{\max} = 1000\). The computed minimal and maximal singular values of \(\matA\) are 0.8771 and \(1.0766 \times 10^3\), respectively. We again observe that \texttt{y-DAPD} outperforms other methods, while \texttt{x-DAPD} achieves convergence similar to \texttt{PAPC}. Note that the final accuracy in this experiment is significantly higher because \(x^*\) is known exactly, rather than approximated by a reference solve whose accuracy limits the error floor.
\begin{figure}[t]
\begin{minipage}{0.5\linewidth}
\centering
\includegraphics[width=\linewidth]{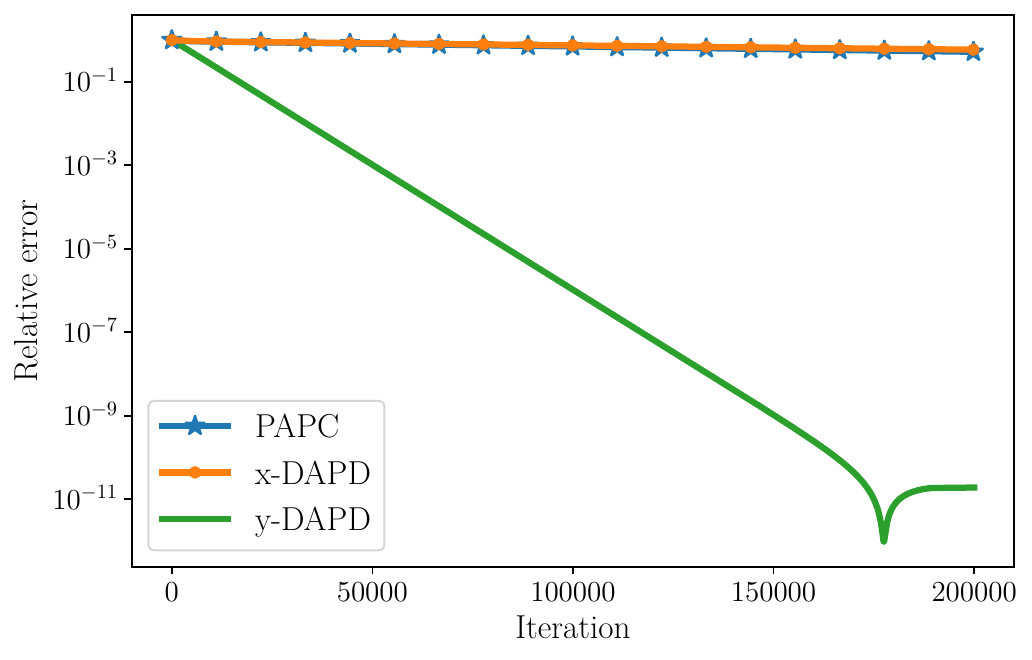}
\end{minipage}%
\begin{minipage}{0.5\linewidth}
\centering
\includegraphics[width=\linewidth]{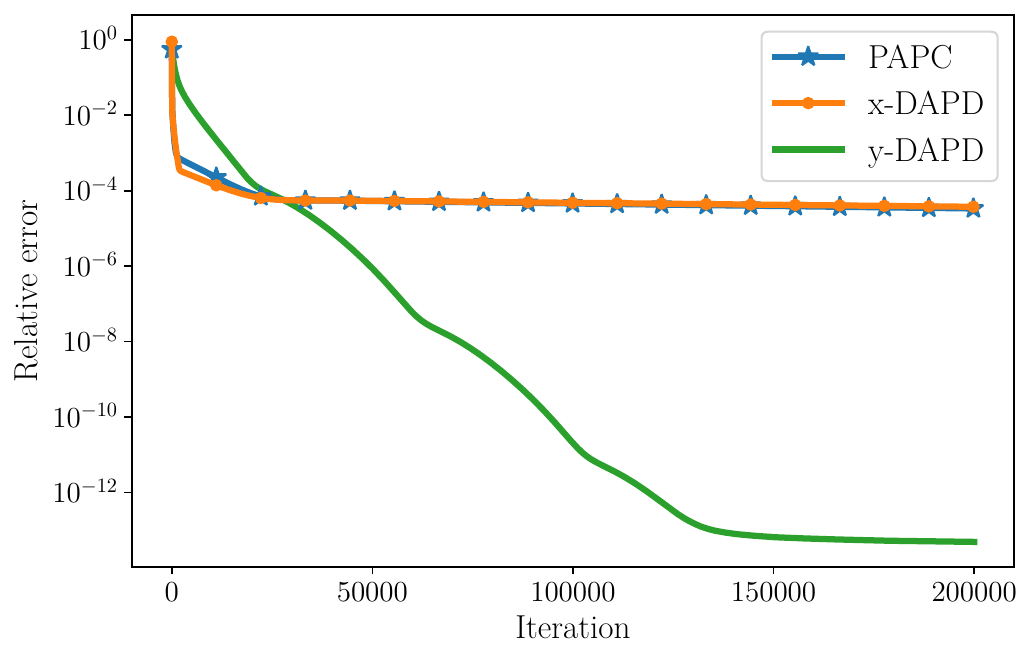}
\end{minipage}
\caption{Left: results for \Cref{subsec:l1}, \eqref{eq:l1 numerical}, \(\frac{s_{\max}^2}{s_{\min}^2} = 10^6\), \(\frac{L}{\mu}=10^3\); right: results for \Cref{subsec:ineq}, \eqref{eq:ineq numerical}, \(\frac{s_{\max}^2}{s_{\min}^2} \approx 10^6\), \(\frac{L}{\mu}=10^3\).}
\label{fig_qp}
\end{figure}

\section{Detailed Constructions, Proofs, and Discussions of the Lower Bound Certificates}
\label{appensec:detail construction}
In this section, we establish lower iteration bounds for deterministic first-order methods applied to problems~\eqref{eq:upper bound prob} and \eqref{eq:problem_lower_2} under Assumptions~\ref{assum:Assumption Linear Lower 1} and~\ref{assum:Assumption Linear Lower 2}, and stochastic block-coordinate methods applied to problem~\eqref{eq:upper bound prob sto} under  Assumptions~\ref{assum:Assumption Linear Stoc}. Theorems~\ref{thm:full c1} and~\ref{thm:full c2} present the lower bounds for the deterministic cases, and Corollary~\ref{cor:full c1 stoc} extends Theorem~\ref{thm:full c1} to the block-wise case. Before stating our results, let us first introduce the definitions of the primal function, dual function, saddle points, and first-order algorithm class.
\begin{definition}\label{def:primal_dual}
We define ${\Phi_x}(\cdot)$ to be the primal function and ${\Phi_y}(\cdot)$ to be the dual function of the saddle point problem \eqref{eq:upper bound prob} (or \eqref{eq:upper bound prob sto}, \eqref{eq:problem_lower_2}), respectively, with the following definitions:
\begin{equation}
{\Phi_x}(x):=\sup_{y\in\bbR^n} F(x, y) \quad \text {and} \quad {\Phi_y}(y):=\inf_{x\in\bbR^m} F(x, y).   
\end{equation}
\end{definition}
\begin{definition}\label{def:optsol}
We call \((x^*,y^*)\) a {saddle point} of \eqref{eq:upper bound prob} (or \eqref{eq:upper bound prob sto}, \eqref{eq:problem_lower_2}) if 
\begin{equation*}\begin{aligned}
x^*\in\arg\min_{x\in\bbR^m}{\Phi_x}(x) \quad\text{and}\quad
y^*\in\arg\max_{y\in\bbR^n}{\Phi_y}(y).
\end{aligned}\end{equation*}
\end{definition}
Another possible definition is based on the conditions $x^*\in\arg\min_{x\in\bbR^m} F(x,y^*)$ and $y^*\in\arg\max_{y\in\bbR^n} F(x^*,y)$,
or the corresponding first order conditions. We state the equivalence of these definitions and existence and uniqueness of saddle points under our assumptions in Proposition~\ref{prop:saddle_cond_1}. Proposition~\ref{prop:saddle_cond_1} may be considered a special case of Lemma 1 of \citet{kovalev2022accelerated}.
\begin{proposition}\label{prop:saddle_cond_1}
For problem \eqref{eq:upper bound prob} under Assumption~\ref{assum:Assumption Linear Lower 1}
(or \eqref{eq:problem_lower_2} under Assumption~\ref{assum:Assumption Linear Lower 2}),
there exists a unique saddle point $(x^*,y^*)$. Furthermore, $(x^*,y^*)$ is a saddle point if and only if
\begin{equation}\label{eq:saddle_cond_1}
\nabla f(x^*)+\matA^{\top} y^*=0,\qquad
\matA x^*=b \quad (\text{resp. } \matA x^*-\nabla g(y^*)=0).
\end{equation}
\end{proposition}

For our lower bound results of the deterministic methods, we refer to the algorithm class of interest as the first-order algorithm class, defined formally in Definition~\ref{def:first-order-alg}. 
\begin{definition}[First-order algorithm class]\label{def:first-order-alg}
In each iteration, the sequence $\left\{\left(x^k, y^k\right)\right\}_{k=0,1,\dots}$ is generated so that $\left(x^k, y^k\right) \in \mathcal{H}_x^k \times \mathcal{H}_y^k$, with $\mathcal{H}_x^0=\mathrm{span}\left\{x^0\right\}, \mathcal{H}_y^0=\mathrm{span}\left\{y^0\right\}$, and
\begin{equation*}
\left\{\begin{array}{l}
\mathcal{H}_x^{k+1}:=\mathrm{span}\left\{x^i, \nabla_x F\left(\bar{x}^i, \bar{y}^i\right): \forall \bar{x}^i \in \mathcal{H}_x^i, \bar{y}^i \in \mathcal{H}_y^i, 0 \leq i \leq k\right\} \\
\mathcal{H}_y^{k+1}:=\mathrm{span}\left\{y^i, \nabla_y F\left(\bar{x}^i, \bar{y}^i\right): \forall \bar{x}^i \in \mathcal{H}_x^i, \bar{y}^i \in \mathcal{H}_y^i, 0 \leq i \leq k\right\}
\end{array}\right..
\end{equation*}
\end{definition}
\begin{remark}\label{remark:phi_eq_0}
In Assumptions~\ref{assum:Assumption Linear Lower 1} and~\ref{assum:Assumption Linear Lower 2}, we set \(\phi(y)=0\) since this is the standard restriction used in lower-bound constructions (e.g., \citet{nesterov2018lectures}, \citet{ouyang2021lower}), and the resulting lower bounds continue to hold for general \(\phi \ne 0\) because a lower bound proved for a subclass automatically transfers to any superset.
\end{remark}
\subsection{Affinely Constrained Strongly-Convex Case (Theorem~\ref{thm:main_c1_short})}
\label{appensubsec:lower1}

We begin the construction process with the following lemma, which provides a lower bound on the duality gap in terms of \(\left\| y - y^* \right\|_2^2\).
\begin{lemma}\label{lemma:duality gap}
Under Assumption~\ref{assum:Assumption Linear Lower 1}, for any \((x, y)\),
\begin{equation*}
    {\Phi_x}(x) - {\Phi_y}(y) \ge \frac{s_{\min}^2}{2L} \left\| y - y^* \right\|_2^2.
\end{equation*}
Similarly, under Assumption~\ref{assum:Assumption Linear Lower 2}, for any \((x, y)\),
\begin{equation*}
    {\Phi_x}(x) - {\Phi_y}(y) \ge \frac{s_{\min}^2}{2L_x} \left\| y - y^* \right\|_2^2.
\end{equation*}
\end{lemma}

\begin{proof}
Under Assumption~\ref{assum:Assumption Linear Lower 1}, the function \(f\) is \(L\)-smooth. Then,
\begin{equation*}
{\Phi_y}(y) = \inf_{x \in \bbR^m} F(x, y) 
= -b^{\top}y + \inf_x \left[ f(x) + x^\top \matA^\top y \right] 
= -b^{\top}y - f^*\left(-\matA^\top y\right),
\end{equation*}
where \(f^*\) is the Fenchel conjugate of \(f\). By Theorem 6 of \citet{kakade2009duality}, \(f^*\) is \(\frac{1}{L}\)-strongly convex. Since \(s_{\min} > 0\), it follows that \({\Phi_y}(y)\) is \(\frac{s_{\min}^2}{L}\)-strongly concave.

Hence, for any \((x, y)\), the duality gap satisfies
\begin{equation*}
{\Phi_x}(x) - {\Phi_y}(y)
= {\Phi_x}(x) - {\Phi_x}(x^*) + {\Phi_y}(y^*) - {\Phi_y}(y)
\ge {\Phi_y}(y^*) - {\Phi_y}(y)
\ge \frac{s_{\min}^2}{2L} \left\| y - y^* \right\|_2^2.
\end{equation*}
The same reasoning applies with \(L_x\) and general convex \(g(y)\) under Assumption~\ref{assum:Assumption Linear Lower 2}.
\end{proof}

To proceed with the construction of the lower bound certificates, we first introduce the matrix \(A \in \bbR^{n \times n}\), along with its powers \(A^2\) and \(A^4\), which will be used in the derivation of lower bound results (see, e.g., \citet{zhang2022lower,ouyang2021lower,nesterov2018lectures}).

\begin{equation}\label{eq:Matrix A}
\begin{aligned}
&A = \begin{pmatrix}
& & & & 1 \\
& & & 1 & -1 \\ 
& & 1 & -1 &\\
& \iddots & \iddots & & \\
1&-1&&&\\
\end{pmatrix},\quad
A^2 = \begin{pmatrix}
1 & -1 & & & \\
-1 & 2 & -1 & & \\
& \ddots & \ddots & \ddots & \\
& & -1 & 2 & -1 \\
& & & -1 & 2
\end{pmatrix},\\
&A^4=\begin{pmatrix}
2 & -3 & 1 & & & & \\
-3 & 6 & -4 & 1 & & & \\
1 & -4 & 6 & -4 & 1 & & \\
& \ddots & \ddots & \ddots & \ddots & \ddots & \\
& & 1 & -4 & 6 & -4 & 1 \\
& & & 1 & -4 & 6 & -4 \\
& & & & 1 & -4 & 5
\end{pmatrix}.
\end{aligned}
\end{equation}

As stated in \citet{zhang2022lower}, these matrices satisfy the following properties.

\begin{proposition}\label{prop:A properties}
The matrix \(A\) in \eqref{eq:Matrix A} satisfies:
\begin{enumerate}
    \item \(A\) is nonsingular;
    \item \(A = A^\top\);
    \item \(\left\|A\right\|_2 \le 2\);
    \item (Zero-chain property) For any vector \(v \in \bbR^n\), if \(v \in \mathrm{span}\{e_i : i \le k\}\) for some \(1 \le k \le n-1\), then \(A^2 v \in \mathrm{span}\{e_i : i \le k+1\}\).
\end{enumerate}
\end{proposition}

Assume without loss of generality that the initial point is \(x^0 = y^0 = 0\) \(\left(\mathcal{H}_x^0 = \mathcal{H}_y^0 = \{0\}\right)\). We consider the following instance of problem~\eqref{eq:upper bound prob} under Assumption~\ref{assum:Assumption Linear Lower 1}, with \(m = 2n\). Let \(x = \begin{pmatrix} x_1 \\ x_2 \end{pmatrix}\), where \(x_1, x_2 \in \bbR^n\):

\begin{equation}\label{eq:problem 1}
\begin{aligned}
    f(x) &= \frac{1}{2} \begin{pmatrix} x_1^\top & x_2^\top \end{pmatrix}
    \begin{pmatrix}
    L I & 0 \\
    0 & \mu I
    \end{pmatrix}
    \begin{pmatrix} x_1 \\ x_2 \end{pmatrix}
    - h^\top x_2, \\
    \matA &= \begin{pmatrix} -s_{\min} I & \hs A \end{pmatrix}, \\
    b &= 0,
\end{aligned}
\end{equation}
where \(\hs = \frac{\sqrt{s_{\max}^2 - s_{\min}^2}}{2}\).

The following lemma shows that the problem setting in \eqref{eq:problem 1} satisfies Assumption~\ref{assum:Assumption Linear Lower 1}.

\begin{lemma}\label{lemma:assum1}
The matrix \(\begin{pmatrix} -s_{\min} I & \hs A \end{pmatrix}\) has largest singular value at most \(s_{\max}\) and minimal singular value strictly greater than \(s_{\min}\).
\end{lemma}

\begin{proof}
For the matrix \(\begin{pmatrix} -s_{\min} I & \hs A \end{pmatrix}\), the singular values are \(\sqrt{s_{\min}^2 + \hs^2 \lambda_i^2}\), where \(\lambda_i\) are the eigenvalues of \(A\). These lie in the interval \((s_{\min}, s_{\max}]\) since \(\lambda_i^2 \in (0, 4]\) and \(\hs = \frac{\sqrt{s_{\max}^2 - s_{\min}^2}}{2}\).
\end{proof}

The partial derivatives of \(F(x, y)\) are:
\begin{equation*}
\begin{aligned}
\nabla_{x_1} F(x, y) &= L x_1  - s_{\min} y, \\
\nabla_{x_2} F(x, y) &= \mu x_2  + \hs A y - h, \\
\nabla_{y} F(x, y)   &= -s_{\min} x_1 + \hs A x_2.
\end{aligned}
\end{equation*}

Therefore, by Definition~\ref{def:first-order-alg}, the subspaces evolve as follows:
\begin{equation*}
\left\{
\begin{aligned}
\mathcal{H}_{x_1}^{1} &= \mathrm{span}\{0\} \\
\mathcal{H}_{x_2}^{1} &= \mathrm{span}\{h\} \\  
\mathcal{H}_{y}^{1}   &= \mathrm{span}\{0\}
\end{aligned}
\right. , \quad
\left\{
\begin{aligned}
\mathcal{H}_{x_1}^{2} &= \mathrm{span}\{0\} \\
\mathcal{H}_{x_2}^{2} &= \mathrm{span}\{h\} \\  
\mathcal{H}_{y}^{2}   &= \mathrm{span}\{A h\}
\end{aligned}
\right. , \quad
\left\{
\begin{aligned}
\mathcal{H}_{x_1}^{3} &= \mathrm{span}\{A h\} \\
\mathcal{H}_{x_2}^{3} &= \mathrm{span}\{h, A^2 h\} \\  
\mathcal{H}_{y}^{3}   &= \mathrm{span}\{A h\}
\end{aligned}
\right., \quad
\left\{
\begin{aligned}
\mathcal{H}_{x_1}^{4} &= \mathrm{span}\{A h\} \\
\mathcal{H}_{x_2}^{4} &= \mathrm{span}\{h, A^2 h\} \\  
\mathcal{H}_{y}^{4}   &= \mathrm{span}\{A h, A^3h\}
\end{aligned}
\right.,\ \dots
\end{equation*}

By induction, we obtain the following lemma:

\begin{lemma}\label{lemma:subspace c1}
For problem~\eqref{eq:upper bound prob} under \Cref{assum:Assumption Linear Lower 1} with \(f, \matA, b\) as defined in \eqref{eq:problem 1}, if the sequence of iterates satisfies Definition~\ref{def:first-order-alg}, then for any \(k \ge 2\),
\begin{equation*}
\mathcal{H}_{y}^{2k} = \mathcal{H}_{y}^{2k+1} = \mathrm{span}\{A^{2i} A h : i = 0, \ldots, k-1\}.
\end{equation*}
\end{lemma}

We next characterize the saddle point \((x^*, y^*)\) of the problem.

\begin{lemma}\label{lemma:yast c1}
Suppose \(L \ge \mu > 0\) and \(s_{\max} \ge \sqrt{5} s_{\min} > 0\). Then, the saddle point of problem~\eqref{eq:upper bound prob} with the specification in \eqref{eq:problem 1} is given by
\begin{equation}\label{eq:yast c1}
\begin{aligned}
x_1^* &= \frac{\hs}{\mu s_{\min}} (I + \alpha A^2)^{-1} A h, \\
x_2^* &= \mu^{-1} (I + \alpha A^2)^{-1} h, \\
y^{{*}} &= \frac{L}{\mu}\frac{\hs}{s_{\min}^2}(I+\alpha A^2)^{-1}Ah,
\end{aligned}
\end{equation}
where
\begin{equation}\label{eq:a c1}
\alpha=\frac{L}{\mu}\frac{\hs^2}{s_{\min}^2}.
\end{equation}
\end{lemma}

\begin{proof}
The expressions for \(x_1^*, x_2^*, y^*\) in \eqref{eq:yast c1} satisfy the first-order optimality conditions:
\begin{equation*}
\begin{pmatrix}
L I & 0 & -s_{\min} I \\
 0 & \mu I & \hs A \\
s_{\min} I & -\hs A & 0
\end{pmatrix}
\begin{pmatrix}
x_1^* \\
x_2^* \\
y^*
\end{pmatrix}
=
\begin{pmatrix}
0 \\
h \\
0
\end{pmatrix}.    
\end{equation*}
Hence, this verifies that \((x^*, y^*)\) is the saddle point of problem~\eqref{eq:upper bound prob} with \eqref{eq:problem 1}.
\end{proof}

The following lemma illustrates the construction of an “approximate” solution \(\hat{y}^*\), which satisfies \(\left\| \hat{y}^* - y^* \right\|_2 = O(q^n)\) and \(\hat{y}_i^* = q^i\) for \(i = 1, \ldots, n\), and for which obtaining a lower bound on \(\left\| y^k - \hat{y}^* \right\|_2\) is simpler.

\begin{lemma}\label{lemma:approximate c1}
Suppose \(L \ge \mu > 0\) and \(s_{\max} \ge \sqrt{5}s_{\min} > 0\). Let
\begin{equation}\label{eq:q c1}
q = 1 - \frac{\sqrt{1 + 4\alpha} - 1}{2\alpha}
\end{equation}
be one of the roots of the equation \(\alpha q^2 - (1 + 2\alpha) q + \alpha = 0\). Let \(\hat{h} = ((1 + \alpha)q - \alpha q^2, 0, \ldots, 0)^\top \in \bbR^n\), and define
\begin{equation}\label{eq:h c1}
h=\frac{\mu}{L}\frac{s_{\min}^2}{\hs}A^{-1}\hat h.
\end{equation}
Then, an approximate solution \((\hat{x}^*, \hat{y}^*)\) can be constructed such that
\begin{equation}\label{eq:yq construction1}
\hat{y}_i^* = q^i \quad \text{for } i = 1, \ldots, n,
\end{equation}
and the approximation error satisfies
\begin{equation*}
\left\| y^* - \hat{y}^* \right\|_2 \le \alpha q^{n+1}.
\end{equation*}
\end{lemma}

\begin{proof}
By Lemma~\ref{lemma:yast c1} and the choice of \(h\) in \eqref{eq:h c1}, the saddle point \(y^*\) satisfies the linear system \((I + \alpha A^2) y^* = \hat{h}\), which expands to:
\begin{equation}\label{eq:linear system c1}
\left\{\begin{aligned}
&(1+\alpha) y_1^*-\alpha y_2^*=(1+\alpha)q-\alpha q^2& \\
&-\alpha y_1^*+(1+2\alpha) y_2^*-\alpha y_3^*=0& \\
&\vdots& \\
&-\alpha y_{n-2}^*+(1+2\alpha) y_{n-1}^*-\alpha y_n^*=0& \\
&-\alpha y_{n-1}^*+(1+2\alpha) y_n^*=0&
\end{aligned}\right.
\end{equation}
By construction \eqref{eq:yq construction1}, the vector \(\hat{y}^*\) satisfies the first \(n-1\) equations exactly. The final equation becomes
\begin{equation*}
-\alpha \hat{y}_{n-1}^* + (1 + 2\alpha) \hat{y}_n^* = \alpha q^{n+1}.
\end{equation*}
Thus, the residual vector is
\begin{equation*}
y^* - \hat{y}^* = \alpha q^{n+1} (I + \alpha A^2)^{-1} e_n,
\end{equation*}
which implies the bound \(\left\| y^* - \hat{y}^* \right\|_2 \le \alpha q^{n+1}\).
\end{proof}

By \Cref{lemma:subspace c1}, the zero-chain property in \Cref{prop:A properties}, and the choice of \(h\) in \eqref{eq:h c1}, we can see that
\begin{equation}\label{eq:Hy c1}
\mathcal{H}_{y}^{2k}   = \mathcal{H}_{y}^{2k+1} = \mathrm{span}\{A^{2i} A h : i = 0, \ldots, k-1\} = \mathrm{span}\left\{ e_1, e_2, \ldots, e_{k} \right\},
\end{equation}
for any \(k \ge 1\). This implies that the only possible nonzero entries of \(y^{2k}, y^{2k+1}\) are within the first \(k\) components. This structure will be useful in lower bounding the errors \(\left\| y^{2k} - y^* \right\|_2\), and similarly for \(\left\| y^{2k+1} - y^* \right\|_2\).

\begin{lemma}\label{lemma:iterate c1}
Suppose \(L \ge \mu > 0\) and \(s_{\max} \ge \sqrt{5}s_{\min} > 0\). Assume \(1 \le k \le \frac{n}{2}\) and \(n \ge 2\log_{q^{-1}}\left( (2 + 2\sqrt{2})\alpha \right)\). Then,
\begin{equation}\label{eq:iterate c1}
\left\|y^{2k} - y^*\right\|_2 \ge \frac{q^{k}}{2\sqrt{2}} \left\|y^0 - y^*\right\|_2,
\end{equation}
where \(y^0 = 0\) is the initialization.
\end{lemma}

\begin{proof}
By \eqref{eq:Hy c1}, we have
\begin{equation*}
\left\|y^{2k }- \hat{y}^*\right\|_2^2 \ge \sum_{i = k+1}^n q^{2i} = q^{2k} \sum_{i = 1}^{n - k } q^{2i} \ge \frac{q^{2k}}{2} \sum_{i = 1}^{n} q^{2i} = \frac{q^{2k}}{2} \left\|\hat{y}^*\right\|_2^2 = \frac{q^{2k}}{2} \left\|y^0 - \hat{y}^*\right\|_2^2,
\end{equation*}
where the inequality uses \(k \le \frac{n}{2} \) and \(q < 1\). Then, using \(n \ge 2\log_{q^{-1}}\left( (2 + 2\sqrt{2})\alpha \right)\),
\begin{equation*}
\begin{aligned}
\left\|\hat{y}^* - y^*\right\|_2 &\le \alpha q^{n+1} \le \frac{q^{k}}{2\sqrt{2}} \left\|y^0 - y^*\right\|_2 \cdot 2\sqrt{2} \alpha q^{n - k + 1} \\
&\le \frac{q^{k-1}}{2\sqrt{2}} \left\|y^0 - y^*\right\|_2 \left/ \left(1 + \frac{1}{\sqrt{2}}\right)\right..
\end{aligned}
\end{equation*}

Therefore,
\begin{equation*}
\begin{aligned}
\left\|y^{2k} - y^*\right\|_2 &\ge \left\|y^{2k} - \hat{y}^*\right\|_2 - \left\|\hat{y}^* - y^*\right\|_2 \\
&\ge \frac{q^{k}}{\sqrt{2}} \left\|y^0 - \hat{y}^*\right\|_2 - \left\|\hat{y}^* - y^*\right\|_2 \\
&\ge \frac{q^{k}}{\sqrt{2}} \left\|y^0 - y^*\right\|_2 - \left(1 + \frac{q^{k}}{\sqrt{2}}\right) \left\|\hat{y}^* - y^*\right\|_2 \\
&\ge \frac{q^{k}}{\sqrt{2}} \left\|y^0 - y^*\right\|_2 - \left(1 + \frac{1}{\sqrt{2}}\right) \left\|\hat{y}^* - y^*\right\|_2 \\
&\ge \frac{q^{k}}{2\sqrt{2}} \left\|y^0 - y^*\right\|_2.
\end{aligned}
\end{equation*}
\end{proof}

Combining Lemmas~\ref{lemma:duality gap}, \ref{lemma:approximate c1}, and \ref{lemma:iterate c1}, we can express the lower bound for first-order methods on problems under Assumption~\ref{assum:Assumption Linear Lower 1} in the following theorem.

\begin{theorem}\label{thm:full c1}
Let positive parameters \(L \ge \mu > 0\) and \(s_{\max} \ge \sqrt{5}s_{\min} > 0\) be given. Let \(k \ge 1\) be an integer. Then there exists a problem instance of the form \eqref{eq:upper bound prob}, satisfying Assumption~\ref{assum:Assumption Linear Lower 1}, with \(f, \matA, b\) specified in \eqref{eq:problem 1} and \(h\) defined in \eqref{eq:h c1}, such that
\begin{equation*}
n \ge \max \left\{ 2 \log_{q^{-1}}\left( (2 + 2\sqrt{2}) \alpha \right),\ k\right\},
\end{equation*}
where \(\alpha\) is defined in \eqref{eq:a c1} and \(q\) in \eqref{eq:q c1}.

Then, for this problem, any approximate solution \((x^k, y^k) \in \mathcal{H}_x^k \times \mathcal{H}_y^k\) generated by a first-order method as described in Definition~\ref{def:first-order-alg} satisfies the following lower bounds:
\begin{equation}\label{eq:thm:full c1}
{\Phi_x}\left({x}^{k}\right)-{\Phi_y} \left({y}^{k}\right) \geq q^{k}\frac{s_{\min}^2 \left\| y^0 - y^* \right\|_2^2}{16 L} \quad \text {and} \quad\left\|y^{k}-y^*\right\|_2 \ge   \frac{q^{\frac{k}{2}}}{2\sqrt{2}} \left\|y^0-y^*\right\|_2.
\end{equation}
\end{theorem}

\begin{remark}[Lower bound and comparison]\label{remark:comparison c1}
If we require the duality gap to be at most \(\epsilon>0\), then any method in our class needs at least
\begin{equation}\label{eq:lower c1}
k \ge 
\frac{\log\left(\frac{s_{\min}^2\|y^{*}-y^0\|_2^2}{16L\epsilon}\right)}{\log\left(q^{-1}\right)}
= \Omega\left(\frac{s_{\max}}{s_{\min}}\sqrt{\frac{L}{\mu}}\log\left(\frac{1}{\epsilon}\right)\right).
\end{equation}
This follows from
\begin{equation*}
\log\left(q^{-1}\right)
= \Theta(1-q)
= \Theta\left(\frac{1}{\sqrt{\alpha}}\right)
= \Theta\left(\frac{s_{\min}}{s_{\max}}\sqrt{\frac{\mu}{L}}\right),
\end{equation*}
for sufficiently large \(\frac{L}{\mu}\) and \(\frac{s_{\max}}{s_{\min}}\).
The problem is equivalent to the problem with linear equalities studied in \citet{salim2022optimal}:
\begin{equation*}
\min_x f(x) \quad \text{s.t.} \quad Mx=b.
\end{equation*}
Let \(\kappa := \frac{L}{\mu}\) and \(\chi := \frac{s_{\max}^2}{s_{\min}^2}\) (notations from \cite{salim2022optimal}). Then
\begin{equation*}
\Omega\left(\frac{s_{\max}}{s_{\min}}\sqrt{\frac{L}{\mu}}\log\left(\frac{1}{\epsilon}\right)\right)
=
\Omega\left(\sqrt{\kappa\chi}\log\left(\frac{1}{\epsilon}\right)\right),
\end{equation*}
which also appears in Theorem~1 of \citet{salim2022optimal} and is covered by the unified analysis in Theorem~2 of \citet{kovalev2024linear}.
While \citet{kovalev2024linear} establishes more general lower bounds, we opt for a simpler, more direct construction. In particular, it allows us to explicitly specify the size and dimensionality of the hard instances, quantities that are infinite-dimensional in \citet{scaman2017optimal} and not explicitly parameterized in \citet{kovalev2024linear}. This concreteness provides complementary value in terms of interpretability and accessibility.
\end{remark}

\begin{remark}[First-order algorithm class]\label{remark:first-order-alg}
Definition~\ref{def:first-order-alg} follows the first-order oracle model of \citet{zhang2022lower}. In view of the bilinear coupling in \eqref{eq:upper bound prob}, it is also natural to consider the following slightly more general class of first-order algorithms:
\begin{equation}\label{eq:def2:first-order-alg}
\left\{
\begin{aligned}
\mathcal{H}_x^{k+1} &:= \mathrm{span}\left\{x^i,\ \nabla f\left(\bar{x}^i\right),\ \matA^{\top}\bar{y}^i : x^i, \bar{x}^i \in \mathcal{H}_x^i,\ \bar{y}^i \in \mathcal{H}_y^i,\ 0 \le i \le k \right\}, \\
\mathcal{H}_y^{k+1} &:= \mathrm{span}\left\{y^i,\ b,\ \matA\bar{x}^i : \bar{x}^i \in \mathcal{H}_x^i,\ y^i\in \mathcal{H}_y^i,\ 0 \le i \le k \right\}.
\end{aligned}
\right.
\end{equation}
Because the optimality certificate in \eqref{eq:problem 1} depends on the iterates only through the subspaces \(\mathcal{H}_x^k\) and \(\mathcal{H}_y^k\), \Cref{lemma:subspace c1} applies to any iterate sequence satisfying \eqref{eq:def2:first-order-alg}, and therefore Theorem~\ref{thm:full c1} remains valid.
Moreover, in this setting the oracle terms in \eqref{eq:def2:first-order-alg} do not provide additional directions that would enlarge the relevant subspaces for our hard instance, so our analysis also recovers the matrix--vector multiplication lower bound
\(\Omega\left(\frac{s_{\max}}{s_{\min}}\sqrt{\frac{L}{\mu}}\log\left(\frac{1}{\epsilon}\right)\right)\) (cf.~\citet{salim2022optimal,kovalev2024linear}); the lower bound
\(\Omega\left(\sqrt{\frac{L}{\mu}}\log\left(\frac{1}{\epsilon}\right)\right)\) for oracle access to \(\nabla f\) is classical.
An analogous argument applies to the second certificate and analysis for Theorem~\ref{thm:full c2} (with \(b\) replaced by \(\nabla g\left(\bar{y}^i\right)\), \(\bar{y}^i \in \mathcal{H}_y^i\), \(0 \le i \le k\)).
\end{remark}

\subsection{Block-wise Affinely Constrained Strongly-Convex Case (\cref{cor:main_stoc_short})}
\label{appensubsec:lower stoc}
This subsection establishes the block-wise lower bound in \Cref{cor:main_stoc_short}. We first define the
block-coordinate linear-span oracle model, and then show how the deterministic lower-bound construction
extends to this setting.
\begin{definition}[Block-coordinate first-order method class]
\label{def:block-first-order-alg}
A block-coordinate first-order method generates iterates
$\{(x_1^k,\ldots,x_N^k,y^k)\}_{k\ge 0}$ such that
$x_i^k \in \mathcal{H}_{x,i}^k$ for all $i$ and $y^k \in \mathcal{H}_y^k$, where
$\mathcal{H}_{x,i}^0=\mathrm{span}\{x_i^0\}$ and $\mathcal{H}_y^0=\mathrm{span}\{y^0\}$, and
there exists an index sequence $\{i_k\}_{k\ge 0}$, with $i_k\in\{1,\ldots,N\}$ such that
\begin{equation}\label{eq:def:block-first-order-alg}
\left\{
\begin{aligned}
\mathcal{H}_{x,i_k}^{k+1}
&:= \mathrm{span}\left\{
x_{i_k}^t,\ \nabla f_{i_k}(\bar x_{i_k}^t),\ \matA_{i_k}^{\top}\bar y^t :
x_{i_k}^t, \bar x_{i_k}^t \in \mathcal{H}_{x,i_k}^{t},\ \bar y^t \in \mathcal{H}_{y}^{t},\ 0 \le t \le k
\right\}, \\
\mathcal{H}_{x,\ell}^{k+1} &:= \mathcal{H}_{x,\ell}^{k}\quad \text{for all }\ell\ne i_k,\\
\mathcal{H}_{y}^{k+1}
&:= \mathrm{span}\left\{
y^t,\ b,\ \matA_{i_k}\bar x_{i_k}^t :
\bar x_{i_k}^t \in \mathcal{H}_{x,i_k}^{t},\ y^t \in \mathcal{H}_{y}^{t},\ 0 \le t \le k
\right\}.
\end{aligned}
\right.
\end{equation}
\end{definition}

\begin{corollary}\label{cor:full c1 stoc}
Let positive parameters $\bar L \ge \mu > 0$ and $\bar s_{\max} \ge \sqrt{5}s_{\min} > 0$ be given.
Let $k \ge 1$ be an integer. Then there exists a problem instance of the form
\eqref{eq:upper bound prob sto}, satisfying Assumption~\ref{assum:Assumption Linear Stoc}, with
$f_i,\matA_i,b$ specified in \eqref{eq:problem 1 stoc} and $h$ defined as in \eqref{eq:h c1}, such that
\begin{equation*}
n \ge N \max\left\{2\log_{q^{-1}}\left((2+2\sqrt{2})\alpha\right),\ \left\lfloor \frac{k}{N}\right\rfloor\right\},
\end{equation*}
where $\alpha =  \frac{\bar L}{\mu}\frac{\hs^2}{s_{\min}^2}$, $q=1-\frac{\sqrt{1+4\alpha}-1}{2\alpha}$,
and $\hs = \frac{\sqrt{\bar s_{\max}^2-s_{\min}^2}}{2}$.
For this problem, any approximate solution $(x_1^k,\ldots,x_N^k,y^k)$ generated by a block-coordinate
first-order method as in Definition~\ref{def:block-first-order-alg} satisfies
\begin{equation}\label{eq:cor c1 stoc y}
\left\|y^{k}-y^*\right\|_2
\ge \frac{q^{\left\lfloor k/N\right\rfloor/2}}{2\sqrt{2}} \left\|y^0-y^*\right\|_2.
\end{equation}
In particular, to achieve $\|y^{k}-y^*\|_2 \le \epsilon$, any such method needs at least
\begin{equation}\label{eq:cor c1 extension}
k \ge \frac{2N}{\log(q^{-1})}\log\left(\frac{\|y^0-y^*\|_2}{2\sqrt{2}\epsilon}\right)
= \Omega\left(N\frac{\bar s_{\max}}{s_{\min}}\sqrt{\frac{\bar L}{\mu}}\log\left(\frac{1}{\epsilon}\right)\right)
\end{equation}
block-coordinate iterations.
\end{corollary}

\begin{proof}
Assume \(n\) is divisible by \(N\); otherwise, truncate to the nearest multiple of \(N\). We construct the instance by taking multiple copies of the problem in \eqref{eq:problem 1}. Namely let \(n_i=\frac{n}{N}\), 
$m_i=2n_i=\frac{2n}{N}$ and write $x_i^{\top}=(x_{i,1}^{\top},x_{i,2}^{\top})$. Define
\begin{equation}\label{eq:problem 1 stoc}
\begin{aligned}
f_i(x_i)
&= \frac{1}{2}
\begin{pmatrix} x_{i,1}^\top & x_{i,2}^\top \end{pmatrix}
\begin{pmatrix}
\bar{L} I & 0 \\
0 & \mu I
\end{pmatrix}
\begin{pmatrix} x_{i,1} \\ x_{i,2} \end{pmatrix}
- h^\top x_{i,2}, \\
\matA_i
&= \begin{pmatrix}
\vdots & \vdots \\
0 & 0 \\
-s_{\min} I & \hs A \\
0 & 0 \\
\vdots & \vdots
\end{pmatrix}, \\
b &= 0,
\end{aligned}
\end{equation}
where the nonzero block row of $\matA_i$ is the $i$-th block row.

For the instance \eqref{eq:problem 1 stoc}, which is a separable problem in the form of \eqref{eq:upper bound prob sto}, satisfying \Cref{assum:Assumption Linear Stoc}, each oracle call indexed by $i$ only accesses
$\nabla f_i(\cdot)$, $\matA_i^{\top}(\cdot)$, and $\matA_i(\cdot)$ for that same index $i$, as captured by
\eqref{eq:def:block-first-order-alg}. The claim then follows directly from Theorem~\ref{thm:full c1} on each copy.
\end{proof}

Corollary~\ref{cor:full c1 stoc} is a direct extension of \Cref{thm:full c1} to the block-wise setting, where the lower-bound certificate is obtained by stacking copies of \eqref{eq:problem 1}.
Moreover, since the index sequence \(\{i_k\}_{k\ge 0}\) in \eqref{eq:def:block-first-order-alg} is arbitrary, the class includes randomized sampling rules where \(i_k\) is drawn each iteration. The variant that samples two independent indices \((i,j)\) per iteration can be accommodated by interpreting one iteration as two consecutive updates in \eqref{eq:def:block-first-order-alg}.

\subsection{Convex--Concave Case (\Cref{thm:main_c2_short})}
\label{appensubsec:lower2}
In this subsection we prove Theorem~\ref{thm:main_c2_short} under Assumption~\ref{assum:Assumption Linear Lower 2}. Our construction follows the classical linear span hard instance framework and produces an explicit finite dimensional smooth convex--concave saddle point problem with $m=n$ and a full rank coupling matrix whose singular values lie between $s_{\min}$ and $s_{\max}$. 

Assume without loss of generality that the initial point is \(x^0 = y^0 = 0\) \(\left(\mathcal{H}_x^0 = \mathcal{H}_y^0 = \{0\}\right)\). We consider the following instance of problem~\eqref{eq:problem_lower_2} with \(m = n = 2\ell\). Let \(x = \begin{pmatrix} x_1 \\ x_2 \end{pmatrix}\), \(y = \begin{pmatrix} y_1 \\ y_2 \end{pmatrix}\), where \(x_1, x_2, y_1, y_2 \in \bbR^\ell\), and let \(A \in \bbR^{\ell \times \ell}\). Define:
\begin{equation}\label{eq:problem 2}
\begin{aligned}
    f(x) &= \frac{L_x}{2} \, x_1^\top x_1, \\
    \matA &= \begin{pmatrix}
    \hs A & s_{\min} I \\
    -s_{\min} I & \hs A
    \end{pmatrix}, \\
    g(y) &= \frac{L_y}{2} \, y_1^\top y_1 - h^\top y_2,
\end{aligned}
\end{equation}
where \(\hs = \frac{\sqrt{s_{\max}^2 - s_{\min}^2}}{2}\).

The following lemma shows that the setting in \eqref{eq:problem 2} satisfies Assumption~\ref{assum:Assumption Linear Lower 2}.

\begin{lemma}\label{lemma:assum2}
The matrix
\begin{equation*}
\begin{pmatrix}
\hs A & s_{\min} I \\
- s_{\min} I & \hs A
\end{pmatrix}
\end{equation*}
has largest singular value at most \(s_{\max}\), and minimal singular value strictly greater than \(s_{\min}\).
\end{lemma}

\begin{proof}
The singular values of the matrix
\begin{equation*}
\begin{pmatrix}
\hs A & s_{\min} I \\
- s_{\min} I & \hs A
\end{pmatrix}
\end{equation*}
are given by \(\sqrt{s_{\min}^2 + \hs^2 \lambda_i^2}\), where \(\lambda_i\) are the eigenvalues of \(A\). Since \(\lambda_i^2 \in (0, 4]\) and \(\hs = \frac{\sqrt{s_{\max}^2 - s_{\min}^2}}{2}\), these singular values lie in the interval \((s_{\min}, s_{\max}]\), as claimed.
\end{proof}

The partial derivatives of \(F(x, y)\) are:
\begin{equation*}
\begin{aligned}
\nabla_{x_1} F(x, y) &= L_x x_1 + \hs A y_1 - s_{\min} y_2, \\
\nabla_{x_2} F(x, y) &= s_{\min} y_1 + \hs A y_2, \\
\nabla_{y_1} F(x, y) &= \hs A x_1 + s_{\min} x_2 - L_y y_1, \\
\nabla_{y_2} F(x, y) &= -s_{\min} x_1 + \hs A x_2 + h.
\end{aligned}
\end{equation*}

Therefore, by Definition~\ref{def:first-order-alg}, the subspaces evolve as follows:
\begin{equation*}\begin{aligned}\left\{\begin{aligned}
\mathcal{H}_{x_1}^{1}&=\mathrm{span}\{0\}\\
\mathcal{H}_{x_2}^{1}&=\mathrm{span}\{0\}\\  
\mathcal{H}_{y_1}^{1}&=\mathrm{span}\{0\}\\
\mathcal{H}_{y_2}^{1}&=\mathrm{span}\{h\}  
\end{aligned}\right.,\quad \left\{\begin{aligned}
\mathcal{H}_{x_1}^{2}&=\mathrm{span}\{h\}\\
\mathcal{H}_{x_2}^{2}&=\mathrm{span}\{Ah\}\\  
\mathcal{H}_{y_1}^{2}&=\mathrm{span}\{0\}\\ 
\mathcal{H}_{y_2}^{2}&=\mathrm{span}\{h\} 
\end{aligned}\right. ,\quad \left\{\begin{aligned}
\mathcal{H}_{x_1}^{3}&=\mathrm{span}\{h\}\\
\mathcal{H}_{x_2}^{3}&=\mathrm{span}\{Ah\}\\  
\mathcal{H}_{y_1}^{3}&=\mathrm{span}\{Ah\}\\ 
\mathcal{H}_{y_2}^{3}&=\mathrm{span}\{h,A^2h\}
\end{aligned}\right.,\\
\left\{\begin{aligned}
\mathcal{H}_{x_1}^{4}&=\mathrm{span}\{h,A^2h\}\\
\mathcal{H}_{x_2}^{4}&=\mathrm{span}\{Ah,A^3h\}\\  
\mathcal{H}_{y_1}^{4}&=\mathrm{span}\{Ah\}\\
\mathcal{H}_{y_2}^{4}&=\mathrm{span}\{h,A^2h\}  
\end{aligned}\right.,\quad \left\{\begin{aligned}
\mathcal{H}_{x_1}^{5}&=\mathrm{span}\{h,A^2h\}\\
\mathcal{H}_{x_2}^{5}&=\mathrm{span}\{Ah,A^3h\}\\  
\mathcal{H}_{y_1}^{5}&=\mathrm{span}\{Ah,A^3h\}\\ 
\mathcal{H}_{y_2}^{5}&=\mathrm{span}\{h,A^2h,A^4h\} 
\end{aligned}\right. ,\dots
\end{aligned}
\end{equation*}

By induction, we have the following lemma:

\begin{lemma}\label{lemma:subspace c2}
For problem~\eqref{eq:problem_lower_2} with \(f, \matA, g\) specified in \eqref{eq:problem 2}, if the sequence of iterates satisfies Definition~\ref{def:first-order-alg}, then for any \(k \ge 2\),
\begin{equation*}
\left\{
\begin{aligned}
\mathcal{H}_{y_1}^{2k}   &= \mathrm{span}\{A^{2i} A h : i = 0, \ldots, k - 2\}, \\
\mathcal{H}_{y_2}^{2k}   &= \mathrm{span}\{A^{2i} h : i = 0, \ldots, k - 1\}
\end{aligned}
\right., \quad
\left\{
\begin{aligned}
\mathcal{H}_{y_1}^{2k+1} &= \mathrm{span}\{A^{2i} A h : i = 0, \ldots, k - 1\}, \\
\mathcal{H}_{y_2}^{2k+1} &= \mathrm{span}\{A^{2i} h : i = 0, \ldots, k\}
\end{aligned}
\right..
\end{equation*}
\end{lemma}

Before presenting our results, we first establish the following useful lemma, which provides a closed-form expression for the inverse of a structured block matrix.

\begin{lemma}\label{lemma:inverse}
Let \( S \in \mathbb{R}^{N \times N} \) be a non-singular matrix, and let \( H, G \succeq 0 \) be symmetric matrices of the same dimension. Suppose that the matrix \( S + H S^{-\top} G \) is also non-singular. Then, the inverse of the block matrix
\begin{equation*}
\begin{pmatrix}
H & S \\
- S^{\top} & G
\end{pmatrix}
\end{equation*}
is given by
\begin{equation*}
\begin{pmatrix}
H & S \\
- S^{\top} & G
\end{pmatrix}^{-1}
=
\begin{pmatrix}
S^{-\top} G (S + H S^{-\top} G)^{-1} & -S^{-\top} + S^{-\top} G (S + H S^{-\top} G)^{-1} H S^{-\top} \\
(S + H S^{-\top} G)^{-1} & (S + H S^{-\top} G)^{-1} H S^{-\top}
\end{pmatrix}.
\end{equation*}
\end{lemma}

Using Lemma~\ref{lemma:inverse}, we obtain the following result for the \(y^*\) component of the saddle point.

\begin{lemma}\label{lemma:yast c2}
Suppose \(L_x, L_y > 0\) and \(s_{\max} \ge \sqrt{5}s_{\min} > 0\). Then, the \(y^*\) component in the saddle point of problem~\eqref{eq:problem_lower_2} with the specification in \eqref{eq:problem 2} is given by
\begin{equation}\label{eq:yast c2}
\begin{aligned}
&y_1^* = -\frac{\hs}{s_{\min}} A y_2^*,\\
&(A^4+\alpha A^2 + \beta I)y_2^* = \frac{L_x s_{\min}^2}{\hs ^4} h,
\end{aligned}
\end{equation}
where
\begin{equation}\label{eq:a c2}
\alpha = 2 \cdot \frac{s_{\min}^2}{\hs^2} + \frac{L_x L_y}{\hs^2}, 
\quad 
\beta = \frac{s_{\min}^4}{\hs^4}.
\end{equation}
\end{lemma}

\begin{proof}
The saddle point \((x_1^*, x_2^*, y_1^*, y_2^*)\) of problem~\eqref{eq:problem_lower_2} with \eqref{eq:problem 2} satisfies the first-order optimality condition:
\begin{equation*}
\begin{pmatrix}
L_x I & 0 & \hs A & -s_{\min} I \\
0 & 0 & s_{\min} I & \hs A \\
-\hs A & -s_{\min} I & L_y I & 0 \\
s_{\min} I & -\hs A & 0 & 0 \\
\end{pmatrix}
\begin{pmatrix}
x_1^* \\ x_2^* \\ y_1^* \\ y_2^*
\end{pmatrix}
=
\begin{pmatrix}
0 \\ 0 \\ 0 \\ h
\end{pmatrix}.
\end{equation*}

We rewrite the system in block form:
\begin{equation*}
\begin{pmatrix}
H & S \\
- S^\top & G
\end{pmatrix}
\begin{pmatrix}
x^* \\
y^*
\end{pmatrix}
=
\begin{pmatrix}
0 \\
\bar{h}
\end{pmatrix},
\end{equation*}
where
\begin{equation*}
H = 
\begin{pmatrix}
L_x I & 0 \\
0 & 0
\end{pmatrix}, \quad
S = 
\begin{pmatrix}
\hs A & -s_{\min} I \\
s_{\min} I & \hs A
\end{pmatrix} = \matA^{\top}, \quad
G = 
\begin{pmatrix}
L_y I & 0 \\
0 & 0
\end{pmatrix}, \quad
\bar{h} = 
\begin{pmatrix}
0 \\
h
\end{pmatrix}.
\end{equation*}

By Lemma~\ref{lemma:inverse}, the solution for \(y^*\) is
\begin{equation*}
y^* = (S + H S^{-\top} G)^{-1} H S^{-\top} \bar{h}.
\end{equation*}

Now compute \(S^{-\top} = \matA^{-1}\). Since \(A\) is symmetric, we have
\begin{equation*}
\matA^{-1} = 
\begin{pmatrix}
\hs A (s_{\min}^2 I + \hs^2 A^2)^{-1} 
& -\left( s_{\min} I + \frac{\hs^2}{s_{\min}} A^2 \right)^{-1} \\
\left( s_{\min} I + \frac{\hs^2}{s_{\min}} A^2 \right)^{-1} 
& \hs A (s_{\min}^2 I + \hs^2 A^2)^{-1}
\end{pmatrix}.
\end{equation*}

Then,
\begin{equation*}
S + H S^{-\top} G = 
\begin{pmatrix}
\hs A + \hs L_x L_y A (s_{\min}^2 I + \hs^2 A^2)^{-1} 
& -s_{\min} I \\
s_{\min} I 
& \hs A
\end{pmatrix}.
\end{equation*}

Again, by applying Lemma~\ref{lemma:inverse} to \(S+H S^{-T} G\),
\begin{equation*}
(S + H S^{-T} G)^{-1} =
\begin{pmatrix}
 * & *\\
-\left(s_{\min}^3I+\left(2s_{\min}\hs^2+\frac{\hs^2 L_xL_y}{s_{\min}}\right)A^2+\frac{\hs^4}{s_{\min}}A^4\right)^{-1}\left(s_{\min}^2 I +\hs^2 A^2\right) & *
\end{pmatrix},
\end{equation*}
where \(*\) denotes irrelevant blocks.

Hence, combining all parts, we obtain
\begin{equation}\label{eq:y2 in c2}
\begin{aligned}
y_2^* 
&= 
\left[
- \left(
s_{\min}^3 I 
+ \left( 2 s_{\min} \hs^2 
+ \frac{\hs^2 L_x L_y}{s_{\min}} \right) A^2 
+ \frac{\hs^4}{s_{\min}} A^4
\right)^{-1}
\left( s_{\min}^2 I + \hs^2 A^2 \right)
\right] \\
&\quad \cdot L_x 
\left[
- \left( s_{\min} I 
+ \frac{\hs^2}{s_{\min}} A^2 \right)^{-1}
\right] h \\
&= \left( A^4 + \alpha A^2 + \beta I \right)^{-1} 
\cdot \frac{L_x s_{\min}^2}{\hs^4} h,
\end{aligned}
\end{equation}
where \(\alpha\) and \(\beta\) are defined in \eqref{eq:a c2}. Finally, \(y_1^* = -\frac{\hs}{s_{\min}} A y_2^*\) follows from the second row of the first-order condition.
\end{proof}

The following lemma shows how to construct an “approximate” solution \(\hat{y}^*\), where \(\hat{y}_2^*\) satisfies \(\left\| \hat{y}_2^* - y_2^* \right\|_2 = O(q^\ell)\) and \(\hat{y}_{2,i}^* = q^i\) for \(i = 1, \ldots, \ell\).

\begin{lemma}\label{lemma:approximate c2}
Suppose \(L_x, L_y > 0\) and \(s_{\max} \ge \sqrt{5}s_{\min} > 0\). Let \(\alpha, \beta\) be defined as in \eqref{eq:a c2}. Then, there exists a real root \(q \in (0,1)\) of the quartic equation
\begin{equation}\label{eq:quartic c2}
1 - (4 + \alpha) q + (6 + 2\alpha + \beta) q^2 - (4 + \alpha) q^3 + q^4 = 0,
\end{equation}
satisfying
\begin{equation}\label{eq:q bounds c2}
1 - \left( \frac{1}{2} + \sqrt{\frac{\alpha}{2\beta} + \frac{1}{4}} \right)^{-1}
< q <
1 - \left( \frac{1}{2} + \sqrt{\frac{\alpha}{\beta} + \frac{1}{4}} \right)^{-1}.
\end{equation}

Let
\begin{equation*}
\hat{h} = \left( (2 + \alpha + \beta) q - (3 + \alpha) q^2 + q^3,\, q - 1,\, 0, \ldots, 0 \right)^\top \in \mathbb{R}^\ell,
\end{equation*}
and define
\begin{equation}\label{eq:h c2}
h = \frac{\hs^4}{L_x s_{\min}^2} \hat{h}.
\end{equation}

Then, an approximate solution \((\hat{x}^*, \hat{y}^*)\) can be constructed with \(\hat{y}_2^*\) given by
\begin{equation}\label{eq:yq construction2}
\hat{y}_{2,i}^* = q^i \quad \text{for } i = 1, \ldots, \ell.
\end{equation}

Moreover, the approximation error satisfies
\begin{equation*}
\left\| y_2^* - \hat{y}_2^* \right\|_2 \le \frac{7 + \alpha}{\beta} q^\ell.
\end{equation*}
\end{lemma}

\begin{proof}
Let \(r = \frac{1}{1 - q}\). The quartic equation \eqref{eq:quartic c2} transforms to the polynomial equation
\begin{equation*}
P(r) := 1 + \alpha r + (\beta - \alpha) r^2 - 2 \beta r^3 + \beta r^4 = 0.
\end{equation*}

Let
\begin{equation*}
\bar{r} = \frac{1}{2} + \sqrt{ \frac{\alpha}{\beta} + \frac{1}{4} },
\quad
\underline{r} = \frac{1}{2} + \sqrt{ \frac{\alpha}{2\beta} + \frac{1}{4} }.
\end{equation*}
Then \(P(\bar{r}) = 1 > 0\), and
\begin{equation*}
P(\underline{r}) = 1 - \frac{\alpha^2}{4\beta}
= 1 - \frac{ \left( 2\frac{s_{\min}^2}{\hs^2} + \frac{L_x L_y}{\hs^2} \right)^2 }{4 \cdot \frac{s_{\min}^4}{\hs^4}} < 0.
\end{equation*}
Therefore, there exists a root \(r \in (\underline{r}, \bar{r})\), which implies the existence of a root \(q \in (0,1)\) satisfying \eqref{eq:q bounds c2}. (Actually, the root has a closed-form expression: \(
q
= \frac{4+\alpha-\sqrt{\alpha^2-4\beta}-\sqrt{\left(4+\alpha-\sqrt{\alpha^2-4\beta}\right)^2-16}}{4}
\).)

By Lemma~\ref{lemma:yast c2} and the definition of \(h\) in \eqref{eq:h c2}, the vector \(y_2^*\) satisfies
\begin{equation*}
(A^4 + \alpha A^2 + \beta I) y_2^* = \hat{h},
\end{equation*}
which corresponds to the linear system
\begin{equation}\label{eq:linear system c2}
\left\{
\begin{aligned}
&(2 + \alpha + \beta) y_{2,1}^* - (3 + \alpha) y_{2,2}^* + y_{2,3}^* = (2 + \alpha + \beta)q - (3 + \alpha)q^2 + q^3, \\
&-(3 + \alpha) y_{2,1}^* + (6 + 2\alpha + \beta) y_{2,2}^* - (4 + \alpha) y_{2,3}^* + y_{2,4}^* = q - 1, \\
&y_{2,1}^* - (4 + \alpha) y_{2,2}^* + (6 + 2\alpha + \beta) y_{2,3}^* - (4 + \alpha) y_{2,4}^* + y_{2,5}^* = 0, \\
&\vdots \\
&y_{2,\ell-4}^* - (4 + \alpha) y_{2,\ell-3}^* + (6 + 2\alpha + \beta) y_{2,\ell-2}^* - (4 + \alpha) y_{2,\ell-1}^* + y_{2,\ell}^* = 0, \\
&y_{2,\ell-3}^* - (4 + \alpha) y_{2,\ell-2}^* + (6 + 2\alpha + \beta) y_{2,\ell-1}^* - (4 + \alpha) y_{2,\ell}^* = 0, \\
&y_{2,\ell-2}^* - (4 + \alpha) y_{2,\ell-1}^* + (5 + 2\alpha + \beta) y_{2,\ell}^* = 0.
\end{aligned}
\right.
\end{equation}

By construction \eqref{eq:yq construction2}, the approximate solution \(\hat{y}_2^*\) satisfies the first \(\ell - 2\) equations exactly. The last two equations yield residuals:
\begin{equation*}
\begin{aligned}
\hat{y}_{2,\ell-3}^* - (4 + \alpha) \hat{y}_{2,\ell-2}^* + (6 + 2\alpha + \beta) \hat{y}_{2,\ell-1}^* - (4 + \alpha) \hat{y}_{2,\ell}^* &= q^{\ell + 1}, \\
\hat{y}_{2,\ell-2}^* - (4 + \alpha) \hat{y}_{2,\ell-1}^* + (5 + 2\alpha + \beta) \hat{y}_{2,\ell}^* &= -q^\ell + (4 + \alpha) q^{\ell + 1} - q^{\ell + 2}.
\end{aligned}
\end{equation*}

Therefore, the residual error satisfies
\begin{equation*}
\beta \left\| \hat{y}_2^* - y_2^* \right\|_2
\le \left\| (A^4 + \alpha A^2 + \beta I)(\hat{y}_2^* - y_2^*) \right\|_2
\le (7 + \alpha) q^\ell,
\end{equation*}
which implies
\begin{equation*}
\left\| \hat{y}_2^* - y_2^* \right\|_2 \le \frac{7 + \alpha}{\beta} q^\ell.
\end{equation*}
\end{proof}

By \Cref{lemma:subspace c2}, the zero-chain property in \Cref{prop:A properties}, and the choice of \(h\) in \eqref{eq:h c2}, we can see that
\begin{equation}\label{eq:Hy c2}
\mathcal{H}_{y_2}^{2k-1}, \mathcal{H}_{y_2}^{2k} \subseteq \mathrm{span}\left\{ h, A^2 h, \ldots, A^{2k-2} h \right\} = \mathrm{span}\left\{ e_1, e_2, \ldots, e_{k+1} \right\},
\end{equation}
for any \(k \ge 1\). This implies that the only possible nonzero entries of \(y_2^{2k-1}, y_2^{2k}\) are within the first \(k+1\) components. This structure will be useful in lower bounding the errors \(\left\| y_2^{2k} - y_2^* \right\|_2\), and similarly for \(\left\| y_2^{2k-1} - y_2^* \right\|_2\).

\begin{lemma}\label{lemma:iterate c2}
Suppose \(L_x, L_y > 0\) and \(s_{\max} \ge \sqrt{5} s_{\min} > 0\). Assume \(1 \le k \le \frac{\ell}{2} - 1\) and \(\ell \ge 2 \log_{q^{-1}} \left( \left( 2 + 2\sqrt{2} \right) \frac{7 + \alpha}{\beta} \right)\). Then,
\begin{equation}\label{eq:iterate c2}
\left\| y^{2k} - y^* \right\|_2 
\ge \left\| y_2^{2k} - y_2^* \right\|_2 
\ge \frac{q^{k+1}}{2\sqrt{2}} \left\| y_2^0 - y_2^* \right\|_2 
\ge \frac{q^{k+1}}{2\sqrt{2}} \cdot \frac{s_{\min}}{s_{\max}} \left\| y^0 - y^* \right\|_2,
\end{equation}
where \(y^0 = \begin{pmatrix} y_1^0 \\ y_2^0 \end{pmatrix} = 0\) is the initialization.
\end{lemma}

\begin{proof}
By \eqref{eq:Hy c2}, we have
\begin{equation*}
\begin{aligned}
\left\| y_2^{2k} - \hat{y}_2^* \right\|_2^2 
&\ge \sum_{i = k+2}^\ell q^{2i} 
= q^{2k+2} \sum_{i = 1}^{\ell - k - 1} q^{2i} \\
&\ge \frac{q^{2k+2}}{2} \sum_{i = 1}^\ell q^{2i} 
= \frac{q^{2k+2}}{2} \left\| \hat{y}_2^* \right\|_2^2 
= \frac{q^{2k+2}}{2} \left\| y_2^0 - \hat{y}_2^* \right\|_2^2,
\end{aligned}
\end{equation*}
where the last inequality uses the fact that \(k \le \frac{\ell}{2} - 1\) and \(q < 1\).

Since \(\ell \ge 2 \log_{q^{-1}} \left( \left( 2 + 2\sqrt{2} \right) \frac{7 + \alpha}{\beta} \right)\), we have
\begin{equation*}
\begin{aligned}
\left\| \hat{y}_2^* - y_2^* \right\|_2 
&\le \frac{7 + \alpha}{\beta} q^\ell 
\le \frac{q^{k+1}}{2\sqrt{2}} \left\| y_2^0 - y_2^* \right\|_2 
\cdot 2\sqrt{2} \cdot \frac{7 + \alpha}{\beta} q^{\ell - k - 1} \\
&\le \frac{q^{k+1}}{2\sqrt{2}} \left\| y_2^0 - y_2^* \right\|_2 \left/ \left( 1 + \frac{1}{\sqrt{2}} \right) \right..
\end{aligned}
\end{equation*}

Then, we conclude
\begin{equation*}
\begin{aligned}
\left\| y_2^{2k} - y_2^* \right\|_2 
&\ge \left\| y_2^{2k} - \hat{y}_2^* \right\|_2 - \left\| \hat{y}_2^* - y_2^* \right\|_2 \\
&\ge \frac{q^{k+1}}{\sqrt{2}} \left\| y_2^0 - \hat{y}_2^* \right\|_2 - \left( 1 + \frac{q^{k+1}}{\sqrt{2}} \right) \left\| \hat{y}_2^* - y_2^* \right\|_2 \\
&\ge \frac{q^{k+1}}{\sqrt{2}} \left\| y_2^0 - y_2^* \right\|_2 - \left( 1 + \frac{1}{\sqrt{2}} \right) \left\| \hat{y}_2^* - y_2^* \right\|_2 \\
&\ge \frac{q^{k+1}}{2\sqrt{2}} \left\| y_2^0 - y_2^* \right\|_2.
\end{aligned}
\end{equation*}

Finally, since \(y_1^* = -\frac{\hs}{s_{\min}} A y_2^*\), we have
\begin{equation*}
\left\| y_1^* \right\|_2 
\le \frac{2\hs}{s_{\min}} \left\| y_2^* \right\|_2 
\end{equation*}
so
\begin{equation*}
\left\| y^* \right\|_2^2=\left\|y_1^*\right\|_2^2 + \left\|y_2^*\right\|_2^2\le \left(1+\frac{4 \hs^2}{s_{\min}^2}\right) \left\|y_2^*\right\|_2^2=\frac{s_{\max}^2}{s_{\min}^2}\left\|y_2^*\right\|_2^2
\end{equation*}
This completes the proof.
\end{proof}

Combining Lemmas~\ref{lemma:duality gap}, \ref{lemma:approximate c2}, and \ref{lemma:iterate c2}, we obtain the following lower bound for first-order methods on problems under Assumption~\ref{assum:Assumption Linear Lower 2}.

\begin{theorem}\label{thm:full c2}
Let positive parameters \(L_x, L_y > 0\) and \(s_{\max} \ge \sqrt{5} s_{\min} > 0\) be given. Let \(k \ge 1\) be an integer. Then there exists a problem instance of the form \eqref{eq:problem_lower_2}, satisfying Assumption~\ref{assum:Assumption Linear Lower 2}, with \(f, \matA, g\) specified in \eqref{eq:problem 2} and \(h\) defined in \eqref{eq:h c2}, such that
\begin{equation*}
\ell \ge \max \left\{ 2 \log_{q^{-1}} \left( (2 + 2\sqrt{2}) \cdot \frac{7 + \alpha}{\beta} \right),\ k + 1 \right\},
\end{equation*}
where \(\alpha, \beta\) are defined in \eqref{eq:a c2} and \(q\) is the solution to \eqref{eq:quartic c2} satisfying the bounds in \eqref{eq:q bounds c2}.

Then, for this problem, any approximate solution \((x^k, y^k) \in \mathcal{H}_x^k \times \mathcal{H}_y^k\) generated by a first-order method as described in Definition~\ref{def:first-order-alg} satisfies the following lower bounds:
\begin{equation}\label{eq:thm:full c2}
{\Phi_x}\left({x}^k\right)-{\Phi_y} \left({y}^k\right) \geq q^{k+3}\frac{s_{\min}^4 \left\| y^0 - y^* \right\|_2^2}{16 L_xs_{\max}^2} \quad \text {and} \quad\left\|y^{k}-y^*\right\|_2\ge \frac{q^{\frac{k+3}{2}}}{2\sqrt{2}}\frac{s_{\min}}{s_{\max}}\left\|y^0-y^*\right\|_2.
\end{equation}
\end{theorem}

\begin{remark}[Lower bound and comparison]\label{remark:comparison c2}
As a result, if we require \(\|y^k-y^{*}\|_2 \le \epsilon\), then the number of iterations needed is at least
\begin{equation}\label{eq:lower c2}
k \ge 2
\frac{\log\left(\frac{s_{\min}\|y^0-y^{*}\|_2}{2\sqrt{2}s_{\max}\epsilon}\right)}{\log\left(q^{-1}\right)}
=
\Omega\left(\sqrt{\frac{s_{\max}^2}{s_{\min}^2}+\frac{L_xL_y s_{\max}^2}{s_{\min}^4}}\log\left(\frac{1}{\epsilon}\right)\right),
\end{equation}
where the equality follows from
\begin{equation*}
\log\left(q^{-1}\right)
= \Theta(1-q)
= \Theta\left(\sqrt{\frac{\beta}{\alpha}}\right)
= \Theta\left(\sqrt{\frac{s_{\min}^4}{s_{\max}^2s_{\min}^2 + L_xL_y s_{\max}^2}}\right),
\end{equation*}
for sufficiently large \(\frac{s_{\max}}{s_{\min}}\).
This lower bound is also implied by Theorem~2 of \citet{kovalev2024linear} in the convex--concave case; compared with their framework, our construction is more direct and yields an explicit finite-dimensional hard instance with fully specified parameters, which makes the dependence on \(s_{\min}, s_{\max}, L_x,\) and \(L_y\) transparent.
Finally, note that a dependence on \(\frac{s_{\max}}{s_{\min}}\) (or \(\frac{L_{xy}}{\mu_{xy}}\) in the notation of \citet{kovalev2024linear}) is also needed in addition to \(\sqrt{\frac{L_xL_y s_{\max}^2}{s_{\min}^4}}\), since it is possible that \(L_xL_y \le s_{\min}^2\).
\end{remark}


\end{document}